\documentclass[preprint,11pt]{amsart}
\usepackage[left=1.35in,top=1.25in,right=1.35in,bottom=1.25in]{geometry}

\usepackage{amssymb}
\usepackage{amsmath}
\usepackage{mathrsfs}
\usepackage{xcolor}
\usepackage{mathtools}
\usepackage{enumerate}
\usepackage[colorlinks=true,citecolor=blue,linkcolor=blue]{hyperref}

\usepackage{hyperref}

\usepackage{epigraph}
\newcommand{\Sc}{\mathcal{S}}
\newcommand{\Uc}{\mathcal{U}}
\newcommand{\Bc}{\mathcal{B}}
\newcommand{\Lc}{\mathcal{L}}
\newcommand{\Kc}{\mathcal{K}}
\newcommand{\gf}{\mathfrak{g}}
\newcommand{\VN}{\mathrm{VN}}
\newcommand{\len}{{\rm len}}

\DeclareMathOperator{\supp}{supp}
\DeclareMathOperator{\lcm}{lcm}

\DeclareMathOperator{\Res}{Res}

\DeclarePairedDelimiter\floor{\lfloor}{\rfloor}

\newcommand{\R}{{\mathbb{R}}}
\newcommand{\Z}{{\mathbb{Z}}}
\newcommand{\C}{{\mathbb{C}}}

\newcommand{\vol}{\mathrm{vol}}

\numberwithin{equation}{section}

\newtheorem{theorem}{Theorem}[section]
\newtheorem{proposition}[theorem]{Proposition}
\newtheorem{corollary}[theorem]{Corollary}
\newtheorem{definition}[theorem]{Definition}
\newtheorem{lemma}[theorem]{Lemma}

\newtheorem{example}[theorem]{Example}
\newtheorem{remark}[theorem]{Remark}

\newcommand{\reva}{\color{blue}}

\title{Weyl asymptotic formulas in the nilpotent Lie group setting} 

\author{Shiqi Liu} 
\address{UNSW Sydney, Kensington, NSW, 2031, Australia}

\author{Edward McDonald}
\address{Universit\'e Paris-Est Cr\'eteil, Cr\'eteil, 94010, France}

\author{Fedor Sukochev}
\address{UNSW Sydney, Kensington, NSW, 2031, Australia}

\author{Dmitriy Zanin} 
\address{Central South University, Changsha, China}

\begin{document}

\maketitle

\makeatother


\begin{abstract}
The asymptotic properties of negative order pseudo-differential operators have been an important part of the spectral theory since H.Weyl's classical results.
In this paper, we derive a spectral asymptotic formula for the negative fractional powers of hypoelliptic operators on graded Lie groups. Such operators have anisotropically homogeneous principal symbols; for these, our results generalize known results of Birman and Solomyak from 1977. Additionally, our work implies a version of Connes' integration formula for hypoelliptic operators on graded Lie groups.

Our methods allow us to extend results from constant-coefficient operators to those with smoothly varying coefficients. 
The principal technique is to adapt the singular value perturbation arguments of Birman and Solomyak to the setting of nilpotent Lie groups. 
The decomposing of graded Lie groups is inspired by Folland and Stein in their development of harmonic analysis on homogeneous groups.
\end{abstract}

\tableofcontents


\section{Introduction}\label{sec-intro}

This paper is a continuation of a series of works on spectral asymptotics on non-compact manifolds, originating with Birman and Solomyak \cite{BirmanSolomyakAnisotropic1977}, reemphasized in \cite{SZ22-connes-integration, Rozenblum22-eigenvalue-Connes}, {and} extended to sub-Laplacians on stratified  Lie groups in \cite{MSZ-stratified-23}. 
Here, we generalize this theme to any uniformly Rockland differential operator on a graded Lie group (See Definition \ref{definition of ellipticity}). Such operators are hypoelliptic and constitute a much larger class than elliptic operators on Lie groups in the classical sense. As a direct application, we prove Connes' integral formula for uniformly Rockland differential operators on {graded} Lie groups.


Weyl-type asymptotics have a long history dating back to H.Weyl \cite{Weyl1911}. For readers interested in this history, we recommend  \cite[Ch7, Notes and Historical Remarks]{Simon-course-IV}, \cite{Ivrii-100yearsWeyl}, and \cite{ArendtNittkaPeterSteiner2009}.
For non-compact manifolds, Birman and Solomyak considered a negative order compact pseudodifferential operator and derived a spectral asymptotic formula \cite{BirmanSolomyakAnisotropic1977}. 
Due to the non-ellipticity of our operator and the non-commutativity of the {graded} Lie group,  we use operator theory in place of the Fourier analysis of Birman and Solomyak. These methods were first introduced in \cite{MSZ-stratified-23} and are fully developed in this paper.
On the other hand, Weyl-type asymptotics describe the behaviour of the eigenvalues directly without taking a Dixmier trace. This implies Connes' integration formula \cite{SZ22-connes-integration, Rozenblum22-eigenvalue-Connes, Ponge23-Connes-Weyl}.

Since the work of Birman and Solomyak, many authors have studied the singular value asymptotics of negative order pseudodifferential operators from various perspective. This includes the recent work of Xiong \cite{MSX2, SXZ2023}, Frank \cite{Frank_Cwikel_CLR_2014, FSZ2023, FSZ2024} and Rozenblum \cite{RozenblumShargorodsky-critical-case, Rozenblum-LiebThirring,RozenblumTashchiyan, RozenblumTashchiyan-noncritical-case}.


In the setting of {graded} Lie groups, we define a global version of the ``maximal hypoellipticity'' of Helffer and Nourrigat \cite[Definition 1.1]{HelfferNourrigat1985}. That is, a differential operator is said to be uniformly Rockland (Definition \ref{definition of ellipticity}) if its top degree part at each point, as a constant coefficient differential operator, obeys elliptic estimates with constant uniform over the group. This has an equivalent description in terms of the representation theory \cite[Lemma 4.8.]{LMSZ25_elliptic}, which recalls the Rockland condition \cite{Rockland1978}. Such operators are hypoelliptic but not always elliptic in the classical sense \cite[Theorem 2.1]{HelfferNourrigat1979}, also see \cite[Example 4.31]{LMSZ25_elliptic}. We recommend \cite[Chapter 4]{FischerRuzhansky2016} and \cite{LMSZ25_elliptic} for more detail.

The main result of this paper is an explicit formula for the spectral asymptotic associated with  uniformly Rockland differential operators on {graded} Lie group.

We continue the philosophy of our previous work \cite{LMSZ25_elliptic}: passing from pointwise information to global conclusions. Starting with results for constant-coefficient operators, we extend them to operators with smoothly varying coefficients. Inspired by work of Folland and Stein \cite{FollandStein1982}, we decompose the Lie group into translates of small balls, establish the local estimates in each ball, and assemble the local estimates into one that holds globally. Techniques from noncommutative analysis \cite{LSZ2021-singular-trace-v1} play a central role throughout.
Although we do not employ pseudodifferential calculus, we expect that our results could be recovered using those methods. For reference, see \cite{BealsGreiner1988}, \cite{Christ-Geller-Glowacki-psudo-on-group-w-dila}, \cite{PongeAMS2008}, and the more recent contributions \cite{vanErpAnn1, vanErpYuncken2019} and \cite{FischerM_Semiclassical2025}.
Same spirit of using operator algebra methods instead of pseudodifferential calculus also appears in Voiculescu's early work on the Heisenberg group \cite{Voicu_extension_Heisenberg1981}.

To illustrate our result, we show our formula resembles a version of Connes' integral formula for uniformly Rockland differential operators on {graded} Lie group (see Theorem \ref{thm-Connes-trace-constant}). In 1984, Wodzicki studied the noncommutative residue \cite{Wodzicki1984}. Shortly thereafter, Connes revealed that, for classical pseudodifferential operators with order $-d$ on a compact $d-$manifold, the Wodzicki residue is the same as the Dixmier trace. Recently, within the framework of \cite{vanErpYuncken2019},  Couchet and  Yuncken extended the definition of the Wodzicki residue to filtered manifolds \cite{CouYuncken24}. In Appendix \ref{sec-Connes-Trace}, we express some of our results in terms of Couchet and Yuncken's residue.

\subsection{History of  Weyl-type asymptotic}
On a closed Riemannian manifold $X$, Weyl's law describes the asymptotic behaviour of $N(\lambda)$, the number of eigenvalues of the positive Laplace-Beltrami operator $-\Delta_X$ on $X$ which are less than $\lambda$ as follows: 
\begin{equation*}
N(\lambda) = (2\pi)^{-d}\omega_d \vol(X) \lambda^{d/2} (1 + o(1)) \quad \text{as} \ \lambda \to \infty.
\end{equation*}
Here, $\omega_d = \frac{\pi^{\frac{d}{2}}}{\Gamma(\frac{d}{2}+1)}$ is the volume of the unit ball in $\R^d$.
The analogous statement for planar domains with Dirichlet boundary conditions was originally proved by H.~Weyl in 1912 \cite{Weyl1912}. Several modern expositions are available, for the Euclidean domain case see e.g. \cite[Theorem 7.5.29]{Simon-course-IV} and for the compact manifold or domain with compact closure case see e.g. \cite{Chavel-eigenvalues-1984}.

The inverse of $1-\Delta_X$ is compact, and in fact belongs to a weak Schatten class. Weyl's law can equivalently be stated in terms of the eigenvalues of the inverse $(1-\Delta_X)^{-1},$ as
\begin{equation}\label{laplacian_weyl_law_intro}
\mu(k, (1-\Delta_{X})^{-1}) \sim \frac{(\omega_d \vol(X))^{2/d}}{(2\pi)^2} k^{-2/d} \quad \text{as} \ k \to \infty.
\end{equation}
Here,   $\sim$ is asymptotic equivalence (the limit of the ratio is unity) and $\mu(k, (1-\Delta_{X})^{-1})$ denotes the $k$-th largest singular value of $(1-\Delta_{X})^{-1}$. Since $(1-\Delta_X)^{-1}$ is positive, these are also the eigenvalues.

One of the standard methods to prove Weyl's law \eqref{laplacian_weyl_law_intro} and its generalisations involves analysis of the zeta function
\[
\zeta_X(s) := \mathrm{Tr}((1-\Delta_X)^{-s}),\quad \Re(s)>\frac{\mathrm{dim}(X)}{2}.
\]
It can be proved that $\zeta_X$ has meromorphic continuation to the complex plane, with a finite number of poles. The presence of a simple pole at $s=\frac{d}{2}$ of the meromorphic continuation of $\zeta_X$ implies the Weyl law via the Wiener-Ikehara Tauberian theorem \cite[Chapter II]{Shubin-psido-2001}. 

This restatement of Weyl's law in \eqref{laplacian_weyl_law_intro} is a very special case of a general result concerning the asymptotic behaviour of the eigenvalues of negative order pseudodifferential operators. The first such result was proved by Birman and Solomyak \cite{BirmanSolomyakAnisotropic1977}. Later proofs and generalisations were obtained by {many others, including Birman and Solomyak, Dauge and Robert \cite{Dauge86}, Ponge \cite{PongeAMS2008, Ponge23-Connes-Weyl}, and Ivrii \cite[Section 11.8]{Ivrii_microlocal_II}}.

Weyl-type asymptotics for operators on $\R^d,$ and on other non-compact manifolds, have been investigated by many authors.
Since a negative order pseudodifferential operator $P$ on $\R^d$ is not necessarily compact, one option is to consider the product $M_fP,$ where $f$ is a smooth compactly supported function and $M_f$ is the operator of pointwise multiplication. Already Birman and Solomyak derived spectral asymptotics for operators of the form $M_fPM_g,$ where $f$ and $g$ are compactly supported and $P$ is a pseudodifferential operator on $\R^d$ of negative order. 

A typical result is for $P = (1-\Delta_{\R^d})^{-\frac{\gamma}{2}},$ where $\Delta_{\R^d} = \sum_{j=1}^d \partial_j^2$ is the standard Euclidean Laplacian and $\gamma>0.$ Then Birman and Solomyak's work implies that if $0\leqslant f\in C^\infty_c(\R^d)$ then
\begin{equation}\label{strong_weyl_law}
\lim_{k\to\infty} k^{\frac{\gamma}{d}}\mu(k,M_f(1-\Delta_{\R^d})^{-\frac{\gamma}{2}}) = \left(\frac{\omega_d}{(2\pi)^{\frac{d}{2}}}\int_{\R^d}{f(x)}^{\frac{d}{\gamma}}\,dx\right)^{\frac{\gamma}{d}}.
\end{equation}
See \cite[Theorem 2]{BirmanSolomyakAnisotropic1977}.
The conditions on $f$ can be weakened, and when $2\gamma<d$ we can assume that merely $f \in L_{\frac{d}{\gamma}}(\R^d).$ In that case we can also replace $1-\Delta_{\R^d}$ by $-\Delta_{\R^d}.$

{Some of the interest in expressions similar to \eqref{strong_weyl_law} has been due to the connection to counting bound states of Schr\"odinger operators. 
}

Special attention has been paid to the case $\gamma=1$ and $d>2$. {In this case,  a  slight modification of \eqref{strong_weyl_law} yields} an asymptotic
formula for the number of bound states of the Schr\"odinger operator $-\hbar^2\Delta_{\R^d}-M_f^2$ as $\hbar\to 0.$ This is due to the well-known Birman-Schwinger principle \cite{Birman-birman-schwinger-prip, Schwinger-birman-schwinger-prip}, more details can be found in \cite[Section 5]{Simon-bound-states-1976}.
{
The Cwikel-Lieb-Rozenblum inequality \cite{Cwikel-estimate-1977, Lieb-CLR-1976, Rozenbljum-CLR-1972} allows this result to be extended to less regular functions by giving an upper bound for the number of negative eigenvalues of $-\Delta+M_V,$ on $\mathbb{R}^d$ in $d>2.$  The strongest results in this direction, by Frank \cite{Frank2023}, make only very weak assumptions  of $V_+$, that is, $V_- \in L_{\frac{d}{2}}(\R^d)$, while $V_{+}\in L_{1}^{loc}(\R^d)$.  For $V \leqslant 0$, passing from a limit of the form \eqref{strong_weyl_law} to an asymptotic formula for the number of bound states of $-\hbar^2 \Delta + M_V$ is straightforward, but this is no longer the case for general $V$. The extension to general $V$ involves significant work in the $\mathbb{R}^d$ case. In the general graded group case, we expect it to be similarly difficult and hence we refrain from doing so in this text.
}

Taking $\gamma=\frac{d}{2},$ \eqref{strong_weyl_law} can be understood as a strong version of Connes' integration formula, as we discuss below in Section \ref{sec-main-results}. {Remarkably, Birman and Solomyak's pioneering work was done decades before Connes' integration formula.}.

It is instructive to see how \eqref{strong_weyl_law} can be proved by following Birman and Solomyak's argument. The idea of their proof, applied to this special situation, is to suppose initially that $f$ is constant on the $d$-cubes $Q_{n,N} = \frac{n+[0,1)^d}{N}$ for some $N\geqslant 1,$ that is
\[
f = \sum_{n\in\mathbb{Z}^d} a_n \chi_{Q_{n,N}}
\]
for some finitely supported scalar sequence $a_n.$ Then
\[
M_f(1-\Delta_{\R^d})^{-\gamma}M_f = \sum_{n,m\in \mathbb{Z}^d} a_{n}a_mM_{\chi_{Q_{n,N}}}(1-\Delta_{\R^d})^{-\gamma}M_{\chi_{Q_{m,N}}}.
\]
The off-diagonal terms with $n\neq m$ have smooth kernel, 
which make their contribution to the singular value asymptotics negligible. The diagonal terms with $n=m$ are pairwise orthogonal operators, and for the purposes of computing spectral asymptotics they can be dealt with individually. This reduces the problem to analysing the singular value asymptotics of the operator
\[
M_{Q_{n,N}}(1-\Delta_{\R^d})^{-\gamma}M_{Q_{n,N}}.
\]
This can be achieved by an explicit computation involving Fourier decomposition. The final step is to remove the assumption that $f$ is constant on cubes of side length $N$ using certain \emph{a priori} estimates on the magnitude of the singular values of operators of the form $M_f(1-\Delta_{\R^d})^{-\gamma}M_f.$

At least from this high-level sketch, we can see that {in Birman and Solomyak's original proof}, the basic strategy of  is to analyse the behaviour of $(1-\Delta_{\R^d})^{-\frac{\gamma}{2}}$ on small cubes using an explicit Fourier decomposition. {Later proofs used other methods to prove the asymptotics for sufficiently regular $f$.}

This raises the question of how we can prove similar results in settings where analogies of the Fourier transform may be unavailable or inconvenient. 

One setting where analogies to \eqref{strong_weyl_law} make sense but where Fourier analytic methods are much more difficult to apply is that of {graded} Lie groups.


\subsection{Graded Lie groups}\label{sec-intro-graded-group}

For definitions and preliminaries on graded Lie group, see Section \ref{subsec-graded-Lie}.  We fix a choice of a preferred generating set $\{X_j\}_{j=1}^{n'}$ (see, Definition \ref{def-generator-basis}) for the remainder of this paper. We also denote $\{v_j\}_{j=1}^{n'}$ for their degrees, and the least common multiple $v = \lcm(\{v_j\}_{j=1}^{n'})$.
    We define an associated operator $\Delta_{G}$:
    \begin{equation}\label{eq-def-new-Laplacian}
        \Delta_{G} = - \sum_{j=1}^{n'} (-1)^{\frac{v}{v_j}}X_j^{\frac{2v}{v_j}}.
    \end{equation}

Mostly we will write $\Delta$ for $\Delta_G$ without ambiguity. 
{The operator $\Delta_G$ is a homogeneous differential operator of order $2v$.} It is typically not elliptic, but hypoelliptic in the classical sense, as in e.g. \cite[Section 11.1]{Hormander-II}.
{A graded Lie algebra $\gf$ is said to be stratified if $\gf_1$ generates $\gf$.
If $G$ is stratified, then the set of preferred generators can be selected to consist entirely of homogeneous elements of order 1. Thus, $v=1$ and $\Delta_G$ becomes the sub-Laplacian.  The paper \cite{MSZ-stratified-23} is based on this setting, where we established an analogous result to \eqref{strong_weyl_law} for  the sub-Laplacian on a stratified Lie group $G$.} Theorem 1.4 there implies that there is a constant $c_G$ such that if {$f \in C^\infty_c(G)$} then for all $\gamma>0$ we have
\begin{equation}\label{our_strong_weyl_law}
\lim_{k\to\infty} k^{\frac{\gamma}{d_{\hom}}}\mu(k,M_f(1-\Delta_{G})^{-\frac{\gamma}{2}}) = \left(c_G\int_{\R^d}|f(x)|^{\frac{d_{\hom}}{\gamma}}\,dx\right)^{\frac{\gamma}{d_{\hom}}}.
\end{equation}
Strictly speaking, in \cite[{Theorem 1.4}]{MSZ-stratified-23}, the above statement was proved assuming that $\gamma$ is a positive integer or half-integer.  One of the results of the present paper is that this restriction is unnecessary.
Fourier analysis on $G$ involves the unitary dual, which is computed by the Kirillov co-adjoint orbit theory. In each irreducible unitary representation, $\Delta_G$ acts a differential operator on some Euclidean space and there may not be any explicit form of its spectrum. This makes methods based on a Fourier decomposition impractical, and for this reason \eqref{our_strong_weyl_law} was proved in \cite[Theorem 1.4]{MSZ-stratified-23} using methods totally different to that of Birman and Solomyak. As the Fourier transform on $G$ is much less useful than on $\R^d,$ the method in \cite{MSZ-stratified-23} was based on the $\zeta$-function proofs of Weyl's law for compact manifolds. There, {for a function $0 \leqslant f \in C_c^{\infty}(G)$}, we defined a $\zeta$-function
\[
\zeta_{f,G}(s) = \mathrm{Tr}((M_{f}(1-\Delta_G)^{-\gamma}M_f)^s),\quad \Re(s)>\frac{d_{\hom}}{2\gamma}
\]
and proved that $\zeta_{f,G}$ has a pole with explicitly computable residue at $s = \frac{d_{\hom}}{2\gamma}.$ This is not a straightforward analogy of the $\zeta$-function proof of the Weyl law on compact manifolds, because $M_f(1-\Delta_G)^{-\gamma}M_f$ is not invertible or elliptic.

Identifying $G$ with $\gf,$ the derivations $\{X_1,\ldots,X_{n'}\}$ are first order differential operators with polynomial coefficients. Therefore, \eqref{our_strong_weyl_law} can be stated in a form making no explicit reference to Lie groups. For example, there exists a stratified Lie group structure $G$ on $\R^{2+N}$ where $N\geqslant 1$ such that
\[
\Delta_G = \partial_{x_0}^2+(\partial_{x_{N+1}}+M_{x_0} \partial_{x_1}+M_{x_0^2}\partial_{x_2}+M_{x_0^3}\partial_{x_3}+\cdots+M_{x_0^N}\partial_{x_N})^2.
\]
Here, $(x_0,\ldots,x_{N+1})$ denote the coordinates of $\R^{2+N}.$ The homogeneous dimension in this case is $d_{\hom} = 1+\frac12(N+1)(N+2).$ This operator is clearly not elliptic in the traditional sense, and is outside the class of operators considered by Birman and Solomyak. In each infinite-dimensional irreducible representation, $\Delta_G$ acts as an anharmonic oscillator on $L_2(\R),$ and its spectrum is not explicitly computable. For details on the Lie group behind this example see \cite[Section 4.3.3]{BLU-stratified-2007}.

%

\subsection{Differential operators}\label{sec-intro-differential-graded}
Birman and Solomyak's version of \eqref{strong_weyl_law} also applies to the situation where $\Delta_{\R^d}$ is replaced by a uniformly elliptic differential operator. We write differential operators on $\R^d$ as
\[
P = \sum_{|\alpha|\leqslant m} M_{a_{\alpha}}\partial^{\alpha}.
\]
Here, $\alpha \in \mathbb{Z}_{\geqslant 0}^d$ is a multi-index, and $|\alpha| = \alpha_1+\cdots+\alpha_d.$ The functions $a_{\alpha}$ are smooth functions on $\R^d,$ and $\partial^{\alpha} = \partial_1^{\alpha_1}\cdots \partial_d^{\alpha_d}.$ We say that $P$ has order $m$ if at least one of $a_{\alpha}$ for $|\alpha|=m$ is non-zero.

The principal symbol $\sigma_P(x,\xi)$ of $P$ is defined as
\[
\sigma_P(x,\xi) := i^m\sum_{|\alpha|=m} a_{\alpha}(x)\xi^{\alpha}.
\]
Here, $\xi = (\xi_1,\ldots,\xi_d)\in \R^n$ and $\xi^{\alpha} = \xi_1^{\alpha_1}\cdots \xi_d^{\alpha_d}.$ These definitions and notations are standard.

The differential operator $P$ is called uniformly elliptic if the derivatives of every $a_{\alpha}$ to all orders are uniformly bounded, and there exists a constant $c_P>0$ such that
\begin{equation}\label{eq-def-uniformly-elliptic}
	|\sigma_P(x,\xi)| \geqslant c_P|\xi|^m,\quad (x,\xi) \in \R^d\times \R^d.
\end{equation}
Here, $|\xi| = (\sum_{j=1}^d |\xi_j|^2)^{\frac12}.$

Birman and Solomyak's result in \cite{BirmanSolomyakAnisotropic1977} implies that if $P$ is uniformly elliptic and $P\geqslant 0$, then for every $\gamma>0$ we have
\begin{equation}\label{strong_weyl_law_variable_coefficients}
\lim_{k\to\infty} k^{\frac{\gamma}{d}}\mu(k,M_f(1+P)^{-\frac{\gamma}{m}}) = \Big(\frac{m}{d(2\pi)^d}\int_{\mathbb{R}^d\times\mathbb{S}^{d-1}}|f(x)|^{\frac{d}{\gamma}} \sigma_P(x,\xi)^{-\frac{d}{m}}dx d\xi\Big)^{\frac{\gamma}{d}}.
\end{equation}
The same method of proof outlined above for \eqref{strong_weyl_law} works here, because $P$ is, near the point $x\in \R^d,$ approximated by the constant-coefficient differential operator
\begin{equation}\label{principal_cosymbol}
\sum_{|\alpha|= m} a_{\alpha}(x)\partial^{\alpha}
\end{equation}
which is easily understandable as a Fourier multiplier.

If we attempt to pose the same problem for operators on a {graded} Lie group $G,$ we need to find analogies of differential operators. {The stratified case was long-established, for example, in \cite{Folland1975}. For the general graded case, see \cite{FischerRuzhansky2016}.}  Given a word $\alpha = \alpha_1\ldots\alpha_k$
in the alphabet $\alpha_i \in \{1,2,\ldots,n'\},$ write
\begin{equation}\label{eq-def-diff-op}
    X^{\alpha} = X_{\alpha_1}X_{\alpha_2}\cdots X_{\alpha_k}.
\end{equation}
{
Recall that throughout $\{X_i\}_{j=1}^{n'}$, with degrees $\{v_j\}_{j=1}^{n'}$, is a choice of preferred generators fixed after Definition \ref{def-generator-basis}.
Additionally, we define
\begin{equation}\label{eq-def-len}
    \mathrm{len}(\alpha)= \sum_{j=1}^{k}\deg(X_{\alpha_j})=\sum_{j=1}^{k} v_{\alpha_j}.
\end{equation}
Note that $\mathrm{len}(\alpha)$ is not the usual length $k$ of the word $\alpha,$ but is weighted by the degrees $v_j.$ In the literature, $\len(\alpha)$ has been called a \textit{weighted length} \cite[Page 2]{terElstRobinson97}.
}

\begin{definition}\label{def-differential-op}
A differential operator of order at most $m$ on $G$ is a linear operator given by
    \begin{equation}\label{eq-P-expression}
        P = \sum_{\len(\alpha)\leqslant m} M_{a_{\alpha}}X^{\alpha}:\Sc(G)\to \Sc(G).
    \end{equation}
    Here, every $a_{\alpha}$ is a smooth function on $G$, and $\Sc(G)$ is the set of Schwartz class functions on $G$. In this paper, we ask all coefficient functions $a_{\alpha}$, together with all its derivatives, to be bounded. See Section \ref{subsec-3.2} below.
\end{definition}

If $P$ has a representation where {all of the coefficient functions $a_{\alpha}$ are constants}, we say that $P$ is a constant coefficient differential operator. In this case, $P$ may be identified with an element of the universal enveloping algebra $\Uc(\gf)$. For more background on the universal enveloping algebra, see  \cite[Section 2.2]{LMSZ25_elliptic}.

As a replacement for the notion of principal symbol, we consider the constant coefficient operator obtained from $P$ as follows.
\begin{definition}\label{def-P-top-g} 
Let $P$ be a differential operator as in \eqref{eq-P-expression}. For $g \in G,$ denote
    \[
    P_g^{{\rm top}}=\sum_{\len(\alpha)=m}a_{\alpha}(g)X^{\alpha}.
    \]
{Since each $a_{\alpha}(g)$ is constant, thus $P_g^{\rm top} \in \Uc(\gf)$.}
\end{definition}

The operator $P_g^{{\rm top}}$ is the analogy for $G$ of the operator \eqref{principal_cosymbol}, the Fourier transform of the principal symbol.

In the nilpotent group setting, several different notions of ellipticity have been proposed, such as the maximal hypoellipticity of Helffer and Nourrigat \cite[Definition 1.1]{HelfferNourrigat1985} and the Rockland condition \cite{Rockland1978}. 
For our purposes, the following notion is most useful, recall \cite[Definition 1.10]{LMSZ25_elliptic}. 
\begin{definition}\label{definition of ellipticity} 
Let $P$ be a differential operator of order $m.$ We say that $P$ is uniformly Rockland, if there exists a constant $c_P>0$ such that
$$\|P_g^{{\rm top}}u\|_{L_2(G)}\geqslant c_P\|(-\Delta_G)^{\frac{m}{2{v}}} u\|_{L_2(G)},\quad u\in\Sc(G),\quad g\in G.$$
Recall that $\Delta_G$ is defined by \eqref{eq-def-new-Laplacian}, and $v$ is the least common multiple of degrees of the preferred generators, see Section \ref{sec-intro-graded-group} and Definition \ref{def-generator-basis}.
\end{definition}
This definition extends the classical notion of \emph{uniformly elliptic} to the realm of graded Lie groups. In Euclidean case, Definition \ref{definition of ellipticity} reduces to \eqref{eq-def-uniformly-elliptic}.

\subsection{Main Results}\label{sec-main-results}
The main result of this paper generalises \eqref{our_strong_weyl_law} to the case of operators with variable coefficients.
Our results focus on operators that are positive definite, in the sense that
\begin{equation*}
	\langle Pu,u\rangle_{L_2(G)} \geqslant 0, \quad u\in \Sc(G).
\end{equation*}
This implies $P$ is formally symmetric, that is $P^{\dagger}=P$, see Definition \ref{def-P-dagg-ext} for $P^{\dagger}$. If $P$ is also uniformly Rockland, then $P$ is self-adjoint, see Theorem \ref{gee_introduction}.
\begin{theorem}\label{main theorem general case}
Let $\gamma > 0$, and let $P\geqslant 0$ be an uniformly Rockland operator of order $m.$ For all {$f \in C^\infty_c(G)$} we have
\[
\lim_{k\to\infty}k^{\frac{\gamma}{d_{\hom}}}\mu(k,M_f(P+1)^{-\frac{\gamma}{m}})= \Big(\frac1{\Gamma(\frac{d_{\hom}}{m}+1)}\int_G |f(g)|^{\frac{d_{\hom}}{\gamma}} \tau(\exp(-P^{\mathrm{top}}_g))\,dg\Big)^{\frac{\gamma}{d_{\hom}}}
\]
where $\tau$ is the standard trace on the group von Neumann algebra $\mathrm{VN}(G)$ (see Section \ref{sec_vna} below).
\end{theorem}
We emphasise that Theorem \ref{main theorem general case} is not a trivial generalisation. The nature of Birman and Solomyak's proof is that \eqref{strong_weyl_law_variable_coefficients} is not substantially more difficult to prove than \eqref{strong_weyl_law}, but in our setting allowing the coefficients to vary poses substantial difficulties.

The quantity $\tau(\exp(-P^{\mathrm{top}}_g))$ appearing in Theorem \ref{main theorem general case} is a replacement for the function
\[
x\mapsto \frac{\Gamma(\frac{d}{m})}{(2\pi)^d}\int_{\mathbb{S}^{d-1}} \sigma_P(x,\xi)^{-\frac{d}{m}}\,d\xi.
\]
appearing in \eqref{strong_weyl_law_variable_coefficients}.
Observe that since $\sigma_P$ is homogeneous of order $m$, by changing variable to polar coordinates and definition of Gamma function, the function above is the same as
\[
x\mapsto (2\pi)^{-d}\int_{\R^d} \exp(-\sigma_P(x,\xi))\,d\xi.
\]
More detailed computation can be found in Appendix \ref{appdx-alg-trace}, such as Example \ref{example-Rd-alg-trace}.
Hence, \eqref{strong_weyl_law_variable_coefficients} can be restated as
\begin{align*}
&\lim_{k\to\infty} k^{\frac{\gamma}{d}}\mu(k,M_f(1+P)^{-\frac{\gamma}{m}})\\
&= \Big(\frac{1}{\Gamma(\frac{d}{m}+1)}\int_{\R^d\times \R^d} |f(x)|^{\frac{d}{\gamma}} \exp(-\sigma_P(x,\xi)) (2\pi)^{-d}d\xi dx\Big)^{\frac{\gamma}{d}}.
\end{align*}
We can put Theorem \ref{main theorem general case} into a similar form using the Plancherel theorem, which in this case states
\[
\tau(A) = \int_{\widehat{G}} \mathrm{Tr}_{H_{\pi}}(\pi(A))\,d\pi,\quad A \in L_1(\mathrm{VN}(G),\tau).
\]
Here, $\widehat{G}$ is the space of unitary irreducible unitary representations $(\pi,H_{\pi})$ of $G,$
and $d\pi$ denotes the correctly normalised Plancherel measure. This is a straightforward consequence of the corresponding assertion for $L_2(G),$ as in \cite[Section 18.8]{DixmierCStar1977}.
Hence, the result of Theorem \ref{main theorem general case} is that
\begin{align*}
&\lim_{k\to\infty}k^{\frac{\gamma}{d_{\hom}}}\mu(k,M_f(P+1)^{-\frac{\gamma}{m}})\\
&= \Big(\frac1{\Gamma(\frac{d_{\hom}}{m}+1)}\int_{\widehat{G}\times G} |f(g)|^{\frac{d_{\hom}}{\gamma}} \mathrm{Tr}_{H_{\pi}}(\exp(-\pi(P^{\mathrm{top}}_g)))\,d\pi dg\Big)^{\frac{\gamma}{d_{\hom}}}
\end{align*}
The Plancherel measure on $\widehat{\R^d} = \R^d$ is $(2\pi)^{-d}d\xi,$ and we can see that Theorem \ref{main theorem general case} recovers \eqref{strong_weyl_law_variable_coefficients}.

Using the Cwikel-type estimates proved in {Appendix \ref{cwikel_appendix}}, we can easily deduce the following strengthening of Theorem \ref{main theorem general case}.
\begin{corollary}\label{weaked_smoothness_corollary}
Let $P,\gamma$ and $m$ be as in Theorem \ref{main theorem general case}.
\begin{enumerate}[{\rm (i)}]
\item{} If $\frac{d_{\hom}}{\gamma}>2,$ then the result of Theorem \ref{main theorem general case} holds for all {$f \in L_{\frac{d_{\hom}}{\gamma}}(G).$}
\item{} If $\frac{d_{\hom}}{\gamma}\leqslant 2,$ then the result of Theorem \ref{main theorem general case} holds if {$f \in \ell_{\frac{d_{\hom}}{\gamma}}(L_q)(G)$} for some $q>2.$
\end{enumerate}
\end{corollary}
The space $\ell_{p}(L_q)(G)$ was defined in \cite[{ Definition 6.3}]{MSZ-stratified-23} and is recalled in Definition \ref{decomposition_spaces_definition} below. {
In the abelian case, $G=\mathbb{R}^d,$ when $\frac{d_{\hom}}{\gamma}=2$,
the condition that $f \in \ell_{2}(L_q)(\mathbb{R}^d)$ for some $q>2$ can be weakened to an Orlicz condition \cite{Solomyak1994,SZ2022-eigenvalue-estimate}. We still do not know if the same can be done in general.}


In 1966, J.~Dixmier found a non-zero trace on an ideal of compact operators differing from the classical operator trace \cite{Dixmier}. Originally, he defined a trace on the ideal of compact operators $A\in \Bc(H)$ satisfying:
\begin{equation*}
\sup_{n \geqslant 0 }\frac{1}{\log(n+2)}\sum_{k=0}^{n}\mu(k, A) < \infty.
\end{equation*}

One of the consequences of Connes' trace formula is Connes' integration formula, which says that if $\varphi$ is a continuous normalised trace on the weak-trace class ideal then
\[
\varphi(M_f(1-\Delta)^{-\frac{d}{2}}) = \frac{\mathrm{Vol}(\mathbb{S}^{d-1})}{d(2\pi)^d}\int_{\R^d} f(x)\,dx.
\]
We can restate this as
\[
\varphi((1-\Delta)^{-\frac{d}{4}}M_f(1-\Delta)^{-\frac{d}{4}}) = \frac{\mathrm{Vol}(\mathbb{S}^{d-1})}{d(2\pi)^d}\int_{\R^d} f(x)\,dx.
\]
{Notice that the right hand side does not depend on $\varphi$}.
The connection between Birman--Solomyak's asymptotic formulae and Connes' integration formula was first publicised by
Rozenblum \cite{Rozenblum22-eigenvalue-Connes}, and later in \cite{SZ22-connes-integration}.

Theorem \ref{main theorem general case} together with Corollary \ref{weaked_smoothness_corollary} implies that if $0\leqslant f \in \ell_1(L_q)(G)$ for some $q>1,$ then
\begin{multline*}
\lim_{k\to\infty}k\mu(k,(1+P)^{-\frac{d_{\hom}}{2m}}M_f(P+1)^{-\frac{d_{\hom}}{2m}})\\
    = \frac1{\Gamma(\frac{d_{\hom}}{m}+1)}\int_G f(g) \tau(\exp(-P^{\mathrm{top}}_g))\,dg.
\end{multline*}

{
This represents a strong form of Connes' integration formula for positive functions $f.$ To extend  this to real-valued functions, we need to extract the positive and negative parts. This is the essence of the following theorem:
\begin{theorem}\label{positive_negative_parts_theorem}
    Let $f$ be a real-valued locally integrable function on $G.$ Let $\gamma>0.$ 
    \begin{enumerate}[{\rm (i)}]
        \item{} If $\gamma<d_{\hom},$ assume that $f \in L_{\frac{d_{\hom}}{\gamma}}(G).$ 
        \item{} If $\gamma\geqslant d_{\hom},$ assume that $f \in \ell_{\frac{d_{\hom}}{\gamma}}(L_q)(G)$ for some $q>1.$
    \end{enumerate}
    Let $P\geqslant 0$ be an uniformly Rockland operator of order $m.$ We have
    \begin{multline*}
    \lim_{k\to\infty}k^{\frac{\gamma}{d_{\hom}}}\mu(k,((1+P)^{-\frac{\gamma}{2m}}M_{f}(P+1)^{-\frac{\gamma}{2m}})_{\pm}) \\
    = \Big(\frac1{\Gamma(\frac{d_{\hom}}{m}+1)}\int_G f(g)_{\pm}^{\frac{d_{\hom}}{\gamma}} \tau(\exp(-P^{\mathrm{top}}_g))\,dg\Big)^{\frac{\gamma}{d_{\hom}}}.
    \end{multline*}
    Here, $x_{\pm} = \frac12(|x|\pm x).$
\end{theorem}
}

This can be thought of as a version of Connes' integration formula, in the sense that it implies if $\varphi$ is any continuous normalised trace on the ideal $\mathcal{L}_{1,\infty}$ then for all $f \in \ell_1(L_q)(G)$ for some $q>1$ we have
\[
\varphi((1+P)^{-\frac{d_{\hom}}{2m}}M_f(1+P)^{-\frac{d_{\hom}}{2m}}) = \frac{1}{\Gamma(\frac{d_{\hom}}{m}+1)}\int_{G} f(g)\tau(\exp(-P^{\mathrm{top}}_g))\,dg.
\]
{
\begin{remark}
    We have stated our results so far for operators on $L_2(G).$ Much the same holds in the vector-valued case for operators on $L_2(G,\mathbb{C}^N),$ provided that the coefficients of the operator $P$ are scalar. 
    
    If $f$ is an locally integrable function on $G,$ valued in the self-adjoint $N$ by $N$ complex  matrices, then the matrix-valued analogy of Theorem \ref{positive_negative_parts_theorem} asserts that for $f$ belonging to an appropriate $L_p$ or $\ell_p(L_q)$ space depending on the value of $\frac{\gamma}{d_{\hom}},$ we have
    \begin{multline*}
        \lim_{k\to\infty}k^{\frac{\gamma}{d_{\hom}}}\mu(k,((1+1\otimes P)^{-\frac{\gamma}{2m}}M_{f}(1+1\otimes P)^{-\frac{\gamma}{2m}})_{\pm}) \\
        = \Big(\frac1{\Gamma(\frac{d_{\hom}}{m}+1)}\int_G \mathrm{Tr}(f(g)_{\pm}^{\frac{d_{\hom}}{\gamma}}) \tau(\exp(-P^{\mathrm{top}}_g))\,dg\Big)^{\frac{\gamma}{d_{\hom}}}.
    \end{multline*}
    The proof is the same. We restrict attention to $N=1$ for notational simplicity.
\end{remark}
}

In Appendix \ref{sec-Connes-Trace}, we relate the quantity in the limit formula in Theorem \ref{positive_negative_parts_theorem} with the residue on a filtered manifold recently defined by Couchet and Yuncken.
With this language, the formula is
\[
    \varphi((1+P)^{-\frac{d_{\hom}}{2m}}M_f(1+P)^{-\frac{d_{\hom}}{2m}}) = \frac{1}{d_{\hom}}\int_{G} f(g)\mathrm{Res}((1+P^{\mathrm{top}}_g)^{-\frac{d_{\hom}}{m}})\,dg.
\]

This brings the theorem into a form more closely resembling Connes' trace formula.

\subsection{Examples}\label{subsec-example}
We give a few examples where the quantities appearing in our main Theorem \ref{main theorem constant} are computable. For simplicity, we restrict ourselves to $\gamma=m,$ i.e. to the following result.
\begin{theorem}\label{main theorem special case}
Let $P\geqslant 0$ be an uniformly Rockland order $m$ differential operator on $G.$ For all $f \in C^\infty_c(G),$ we have
\begin{align*}
&\lim_{t\to\infty}t^{\frac{m}{d_{\hom}}}\mu(t,M_f(P+1)^{-1})\\
&=\Big(\frac1{\Gamma(\frac{d_{\hom}}{m}+1)}\int_{\widehat{G}\times G} |f(g)|^{\frac{d_{\hom}}{m}} {\rm Tr}_{H_{\pi}}(\exp(-\pi(P^{\mathrm{top}}_g)))\,d\pi dg\Big)^{\frac{m}{d_{\hom}}}.
\end{align*}
\end{theorem}

The irreducible unitary representations of a nilpotent Lie group are classified by the Kirillov theory of coadjoint orbits. There are a few cases where the classification is simple enough that $\tau$ can be expressed in less abstract terms. Here we present the formulas, the detailed computations are in the Appendix.

\begin{example}\label{example-R-n}
{We consider the abelian case, where $G = \mathbb{R}^d$
is equipped with the trivial grading $\mathfrak{g}=\mathfrak{g}_1 =\mathbb{R}^d.$} In this case, $\widehat{G} = \mathbb{R}^d,$ where the irreducible representations $(\pi_{\xi},\mathbb{C})$ are labelled by $\xi \in \mathbb{R}^d$ with the Plancherel measure
\[
d\pi_{\xi} = (2\pi)^{-d}d\xi.
\]
Let
\[
P = \sum_{|\alpha|\leqslant m} M_{a_{\alpha}}\partial^{\alpha}
\]
be a differential operator with smooth and suitably uniformly bounded coefficients $a_{\alpha}\in C^\infty(\mathbb{R}^d).$
The top degree component, with coefficients frozen at $g \in \mathbb{R}^d$ is the constant-coefficient operator
\[
P^{\mathrm{top}}_g = \sum_{|\alpha|=m} a_{\alpha}(g)\partial^{\alpha},\quad g \in \mathbb{R}^d.
\]
We have
\[
\pi_{\xi}(P^{\mathrm{top}}_g) = \sigma_P(x,\xi) = \sum_{|\alpha|=m} i^{m}a_{\alpha}(g)\xi^{\alpha},\quad g \in \R^d.
\]
Hence, Theorem \ref{main theorem general case} in this case says that
\[
\lim_{t\to\infty}t^{\frac{m}{d}}\mu(t,M_f(P+1)^{-1})= \Big(\frac{1}{d(2\pi)^d}\int_{\mathbb{R}^d\times \mathbb{S}^{d-1}} |f(g)|^{\frac{d}{m}}(i^m\sum_{|\alpha|=m} a_{\alpha}(g)\theta^{\alpha})^{-\frac{d}{m}}\,d\theta dg\Big)^{\frac{m}{d}},
\]
where $d\theta$ is the standard measure on the unit sphere.

If $m=2,$ then the expression for $\tau(\exp(-P^{\mathrm{top}}_g))$ is a Gaussian integral. Writing $a(g) = \{a_{j,k}(g)\}_{j,k=1}^d,$ we evaluate
\[
\tau(\exp(-P^{\mathrm{top}}_g)) = (4\pi)^{-\frac{d}{2}}\det(a(g))^{-\frac12}.
\]
Hence in this case, Theorem \ref{main theorem general case} reduces to
    \begin{multline*}
        \lim_{t\to\infty} t^{\frac{2}{d}}\mu(t,M_f(1+\sum_{|\alpha|\leqslant 2} M_{a_{\alpha}}\partial^{\alpha})^{-1}) \\
        =(4\pi)^{-1}\left(\frac{1}{\Gamma(\frac{d}{2}+1)}\int_{\mathbb{R}^d} |f(g)|^{\frac{d}{2}}\det(a(g))^{-\frac12}\,dg\right)^{\frac{2}{d}}.
    \end{multline*}
\end{example}

\begin{remark}
    {Example \ref{example-R-n} concerns $\mathbb{R}^d$ with a trivial grading. If we instead chose a nontrivial grading, we are in a setting resembling Birman-Solomyak's anisotropically homogeneous operators \cite{BirmanSolomyakAnisotropic1977}. A similar computation in this case is given in Appendix \ref{appdx-alg-trace}.}
\end{remark}

\begin{example}
If $G = \mathbb{H}^n$ is the $2n+1$-dimensional Heisenberg group, then the dual of $G$ (up to a Plancherel null set) consists of representations on $L_2(\mathbb{R}^n).$ By a homogeneity argument, the integral over $\widehat{G}$ reduces to an evaluation at two particular representations, denoted $\pi_{\pm 1}.$ Explicit computation yields
\begin{align*}
&\tau(\exp(-P^{\mathrm{top}}_g))\\
&= (2\pi)^{-3n-1}\frac2m\Gamma(\frac{2n+2}{m})(\mathrm{Tr}_{L_2(\mathbb{R}^n)}(\pi_{+1}(P^{\mathrm{top}}_g)^{-\frac{2n+2}{m}})+\mathrm{Tr}_{L_2(\mathbb{R}^n)}(\pi_{-1}(P^{\mathrm{top}}_g)^{-\frac{2n+2}{m}})).  
\end{align*}

If $m=2,$ we can go further. Assuming that $P^{\mathrm{top}}_g$ has the form
\[
P^{\mathrm{top}}_g = \sum_{j,k=1}^{2n} B_{j,k}(g)X_{j}X_k
\]
where $B_{j,k}$ are smooth functions on $\mathbb{H}^n$ and $X_1,\ldots,X_{2n}$ are the generators of the Heisenberg Lie algebra, we can evaluate $\tau(\exp(-P^{\mathrm{top}}_g))$ as
\[
\tau(\exp(-P^{\mathrm{top}}_g)) = 2^{-n}(2\pi)^{-(3n+1)}\det(B(g))^{-\frac12}\int_{-\infty}^\infty \det(\frac{i\Omega B(g)s}{\sinh(i\Omega B(g)s)})^{\frac12}\,ds. 
\]
Here, $\Omega$ is the standard $2n\times 2n$ symplectic matrix.
\end{example}

\subsection{Structure of the paper}
\begin{itemize}
    \item In Section \ref{sec-Graded-Lie-Group},  We recall background material on analysis on graded Lie groups. In particular, we review results from \cite{LMSZ25_elliptic}, including a special covering lemma and global estimates for uniformly Rockland operators. These results play an important role in the proofs of the main theorems. 
    \item In Section \ref{sec-trace-ideal}, we exhibits  properties of singular values, trace ideals, and von Neumann algebras.
    \item In Section \ref{sec-5}, we prove a special case of Theorem \ref{main theorem general case}. Namely,  when the operator $P$ has constant coefficients and $\gamma = m$.
    \item In Section \ref{sec_abstract_asymptotics}, we prove the general case of Theorem \ref{main theorem general case}.
    \item In Appendix \ref{appdx-alg-trace}, we compute formulas for the von Neumann trace on the group algebra.
    \item In Appendix \ref{cwikel_appendix}, we extend the Cwikel-type estimates in \cite[Section 6]{MSZ-stratified-23} to the case of graded Lie groups.
    \item In Appendix \ref{psido_commutator_appendix}, we prove a generalised version of \cite[Section 5]{MSZ-stratified-23}.
    \item In Appendix \ref{sec-Connes-Trace}, we discuss the relation between our spectral asymptotic formula and the Wodzicki residue defined by Couchet and Yuncken.
\end{itemize}


\section{Analysis on Graded Lie Group}\label{sec-Graded-Lie-Group}

 In this section, we begin by reviewing preliminary material on graded Lie groups, as well as the function spaces and differential operators defined on them. See, subsection \ref{subsec-graded-Lie} and \ref{subsec-3.1}. Our main sources here are \cite{FischerRuzhansky2016,Folland1975,FollandStein1982,RothschildStein1976}.

Next, we recall a special covering Lemma from \cite{LMSZ25_elliptic}, see Lemma \ref{quasi_metric_space_construction}. It will play an essential role in later chapters, such as, Theorem \ref{abstract asymptotics theorem}, especially, Lemma \ref{function system lemma}.

At the end, we collect results from our previous work \cite{LMSZ25_elliptic} (see, subsection \ref{subsec-LMSZ25-result}), which will be essential for the developments in Section \ref{sec_abstract_asymptotics}.

\subsection{Graded Lie group}\label{subsec-graded-Lie}

\begin{definition}\label{def-stratfied-g}
    A Lie algebra $\gf$ is said to be graded if equipped with a direct sum decomposition
    \begin{equation}\label{eq-stratfied-decomp}
        \gf  = \bigoplus_{k=1}^\infty V_k
    \end{equation}
    such that
    \begin{equation}\label{eq-stratified-relation}
        [V_j, V_k] \subseteq V_{j+k}, \quad j,  k \geqslant 1,
    \end{equation}
    In this paper, we only consider finite dimensional Lie algebras, thus there exists $s\geqslant 1$ such that $V_{s+1}=V_{s+2}=\cdots=\{0\}.$ 
    
    A Lie group $G$ is called graded, if it is connected, simply connected, and its Lie algebra $\gf$ is graded.
\end{definition}

\begin{definition}
    A Lie algebra $\gf$ is nilpotent if for
    \begin{equation*}
        \gf_1 \coloneqq \gf, \quad \gf_{j} \coloneqq [\gf, \gf_{j-1}], \ j\geqslant 1,
    \end{equation*}
    there exists a $s\in \Z^{+}$, such that $\gf_j = 0$, $\forall j\geqslant s$. 
\end{definition}

Apparently, by definition, finite dimensional graded Lie algebra is nilpotent. While the reverse is not always true.
\begin{remark}
    Not all nilpotent Lie algebras can be equipped with grading structure. On the other hand, the grading structure is not always unique. \cite[Remark 3.1.6.]{FischerRuzhansky2016}
\end{remark}

\begin{remark}
    We use the term "graded Lie algebra" to maintain consistency with the literature, such as, \cite[Definition 3.1.1.]{FischerRuzhansky2016}. One can also call this algebra  $\Z_s$ graded, $\mathbb{N}$ graded, or $\Z$ graded with additional conditions. However, we caution the reader not to confuse this with the "super graded Lie algebra" commonly encountered in differential geometry, where the prefix "super" is sometimes omitted.
\end{remark}

The homogeneous dimension $d_{\hom}$ is defined as
\[
d_{\hom} = \sum_{k=1}^\infty k\cdot \mathrm{dim}(V_k).
\]
Notice that $d_{\hom}$ highly depends on the grading.
Since $G$ is connected, simply connected and nilpotent, it follows that the exponential map $\exp:\gf\to G$ is a diffeomorphism, and so we may assume that $G$ and $\gf$ coincide as sets. With this identification, the Lebesgue measure on $\gf$ is a bivariant Haar measure for $G$ \cite[Proposition 1.2]{FollandStein1982}. The Haar measure on $G$ is therefore written as $dx,$ so that $\int_{G} f(x)\,dx$ means the integral of a measurable function $f$ on $G$ with respect to the Haar measure of $G.$ The elements of $\gf$ can be identified with right-invariant derivations of compactly supported smooth functions on $G$ ({denote such function class as $C^\infty_c(G)$}), if we identify $X\in \gf$ as the generator of the group of \emph{left} translations
    \[
        Xu(g) := \frac{d}{dt}u(\exp(-tX)g)|_{t=0},\quad g \in G,\, u\in C^\infty_c(G).
    \]

Furthermore, $\gf$ can be equipped with an action of $\mathbb{R}^{\times}$ which we denote by $\delta$, with formula 
\begin{equation}\label{eq-dilation}
    \delta_{t}(y_j) = t^jy_j,\quad y_j \in V_j,\quad t>0.
\end{equation}

These are Lie algebra isomorphisms. Since $G$ is simply connected, by Lie theory, each $\delta_{t}$ will induce an unique Lie group isomorphism which we also denote by $\delta_{t}$. It satisfies  
\[
\delta_{t}\exp(y_1+\cdots+y_s) = \exp(t y_1+t^2 y_2+\cdots+t^s y_s),
\]
where $y_j \in V_j$ ({$y_j=0$, if $V_j = \{0\}$}), and $\exp(y_1+\ldots+y_s)$ is a generic element of $G$, {due to \eqref{eq-stratfied-decomp}}. 
Since $G$ is nilpotent, exponential map is isomorphic between Lie algebra and Lie group, thus every element in $G$ can be expressed in the form $\exp(y_1+\ldots+y_s)$ .
Given a measurable function $f$ on $G,$ we write
\begin{equation}\label{def-op-delta-lambda}
\delta_t f(x) := f(\delta_{t}x),\quad x \in G.
\end{equation}

The space $\Sc(G)$ of Schwartz class functions on $G$ is defined via the identification of $G$ with $\R^d$, and $\Sc(G)$ is the usual Schwartz space $\Sc(\R^d)$ with its canonical Fr\'echet topology. The space of tempered distributions $\Sc'(G)$ is the topological dual of $\Sc(G)$.
We define the dilation of a distribution $\omega$ on $G$ by
\[
(\delta_{t}\omega,f) = t^{-d_{\hom}}(\omega,\delta_{t^{-1}}f),
\]
which is consistent with the definition of $\delta_{t}$ on functions if we identify functions with distributions in the canonical way. A distribution $\omega\in \Sc'(G)$ is called homogeneous of degree $\alpha$ if 
\[
\delta_{t}\omega = t^{\alpha}\omega,\quad t>0.
\]

In this paper, we fix a choice of a set of preferred generators.
\begin{definition}\label{def-generator-basis}
    For a graded Lie algebra $\gf,$ we say that a set $\{X_j\}_{j=1}^{n'}$,$X_i \in \gf$
    is a \emph{set of preferred generators} if
    \begin{enumerate}[{\rm (i)}]
        \item\label{enu-generator-basis-1} $\{X_j\}_{j=1}^{n'}$ are linearly independent,
        \item\label{enu-generator-basis-2} $\{X_j\}_{j=1}^{n'}$ generates $\gf$,
        \item\label{enu-generator-basis-3} Each $X_j$ is homogeneous, that is, $X_j \in V_{v_j}$, for some $v_j \geqslant 1$. 
    \end{enumerate}
    We call $v_j$ the degree of $X_j$, denoted by $\deg(X_j) = v_j$. We also denote the least common multiple $v = \lcm (\{v_j\}_{j=1}^{n'})$.
\end{definition}
\begin{remark}
    By \cite[Lemma 2.2]{terElstRobinson97}, there always exists a preferred generating set for any graded Lie algebra.
\end{remark}

\subsection{Sobolev Spaces}\label{subsec-3.1}

In this Section, we define Sobolev spaces on graded Lie groups, and exhibits some essential properties. For further details, see \cite[Chapter 4]{FischerRuzhansky2016} or \cite{Folland1975} for the special case when $\gf$ is stratified. 

As in Definition \ref{def-generator-basis}, a set of preferred generators $\{X_j\}_{j=1}^{n'}$ is fixed throughout the paper, $\{v_j\}_{j=1}^{n'}$ are their degrees, and $v = \lcm(\{v_j\}_{j=1}^{n'})$. Recall \eqref{eq-def-new-Laplacian}, 
    \begin{equation*}
        \Delta_{G} := - \sum_{j=1}^{n'} (-1)^{\frac{v}{v_j}}X_j^{\frac{2v}{v_j}}.
    \end{equation*}
    This operator is Rockland, that is, for every non-trivial unitary irreducible representation $\pi$ of $G$, $\pi(\Delta)$ is injective on smooth vectors \cite[Definition 4.1.1. and Corollary 4.1.10.]{FischerRuzhansky2016}. The operator $-\Delta_G$ is positive definite and self-adjoint on $L_2(G)$ \cite[Proposition 4.1.15.]{FischerRuzhansky2016}.
    Knowing the self-adjointness of $\Delta_G,$ we define a scale of Sobolev spaces $\{W^s_2(G)\}_{s\in \mathbb{R}}$ with the norms
    \begin{equation}\label{eq-sobolev-norm}
        \|u\|_{W^s_2(G)} := \|(1-\Delta_G)^{\frac{s}{2v}}u\|_{L_2(G)}, \quad u \in C_c^{\infty}(G).
    \end{equation}
    Here, the power $(1-\Delta_G)^{\frac{s}{2v}}$ is understood in the sense of spectral theory.
    The Sobolev space is defined as the closure of $C^\infty_c(G)$ in $\Sc'(G)$ with the norm $\|\cdot\|_{W^s_2(G)}.$
    
    We will mostly abbreviate $\Delta_G$ as $\Delta.$

We list a few useful properties of these Sobolev spaces:
\begin{enumerate}[{\rm (i)}]
    \item Sobolev norms $W^s_2(G)$ are independent of the choice of preferred generating set $\{X_j\}_{j=1}^{n'}$ \cite[Theorem 4.4.20]{FischerRuzhansky2016}.
    \item The Schwartz space $\Sc(G)$ is dense in $W^s_2(G)$ for all $s\in\mathbb{R}$ \cite[Lemma 4.4.1]{FischerRuzhansky2016}.
    \item We have $W^{s_1}_2(G) \subseteq W^{s_2}_2(G)$ for $s_1 \geqslant s_2 \in \R$ \cite[Theorem 4.4.3]{FischerRuzhansky2016}.
    \item For all $s_0,s_1\in \mathbb{R}$ and $0 < \theta < 1$ we have (up to equivalence of norms)
        \begin{equation}\label{interpolation}
            (W^{s_0}_2(G),W^{s_1}_2(G))_{\theta} = W^{s_\theta}_2(G),\quad s_\theta=(1-\theta)s_0+\theta s_1,
        \end{equation}
        where $(\cdot,\cdot)_{\theta}$ is the functor of complex interpolation. 
        In particular,
        \begin{equation}\label{eq-interpolation-norm}
            \|u\|_{W^{s_\theta}_2(G)} \leqslant \|u\|_{W^{s_0}_2(G)}^{1-\theta}\|u\|_{W^{s_1}_2(G)}^{\theta},\quad u \in W^{\max\{s_0,s_1\}}_2(G).
        \end{equation}
        See \cite[Theorem 4.4.28]{FischerRuzhansky2016}.
    \item Recall that $v$ is the least common multiple of $\{v_j\}_{j=1}^{n'}$, the degrees of preferred generators. For $s \in 2v\cdot \Z_{+}$, the Sobolev norm $\Vert \cdot \Vert_{W^{s}_2(G)}$ is equivalent to the following norm:
        \begin{equation}\label{integer sobolev norm}
            u \mapsto \big(\sum_{\len(\alpha)\leqslant s}\|X^{\alpha}u\|_{L_2(G)}^2\big)^{\frac12},
            \quad u \in W^s_2(G).
        \end{equation}
        For $s \in 2v\cdot\Z^+$, we have
        \begin{equation}\label{eq-Sobolev-alt-def}
            W^s_2(G) = \left\{u\in L_2(G) \;:\; \ X^{\alpha}u\in L_2(G), \ \len(\alpha)\leqslant s \right\}.
        \end{equation}
        The equivalence of \eqref{eq-sobolev-norm} and \eqref{integer sobolev norm} was originally proved by Helffer and Nourrigat  \cite[Estimate (6.1)]{HelfferNourrigat1979}, and restated in \cite[Corollary 4.1.14.]{FischerRuzhansky2016}.
    \item For all $s \in 2v \cdot \mathbb{Z}_+,$ $W^{-s}_2(G)$ coincides with the Banach dual of $W^s_2(G).$ Concretely, $W^{-s}_2(G)$ is identified with the space of distributions $u\in \Sc'(G)$ such that there is a constant $C$ such that for all $\phi \in \Sc(G)$ we have
    \[
        |(u,\phi)|\leqslant C\|\phi\|_{W^s_2(G)}.
    \]
    The least constant $C$ is a norm equivalent to $\|\cdot\|_{W^{-s}_2(G)}.$ 
    For a stratified Lie group $G$, this follows from \cite[Theorem 3.15(v)]{Folland1975}, in particular see the Remark below the proof of \cite[Proposition 4.1]{Folland1975}.
\end{enumerate}


It is immediate from the above definitions that the order of a differential operator in $\Uc(\gf)$ coincides with its order as a mapping between Sobolev spaces.
\begin{lemma}\label{lemma-X_alpha-bdd}
For every $s \in \mathbb{R},$ an element $D \in \Uc_m(\gf)$ extends by continuity to a bounded linear map
\begin{equation*}
D:W^s_2(G)\to W^{s-m}_2(G).
\end{equation*}
\end{lemma}


\begin{definition}\label{C_infty_definition}
Let $C_b(G)$ denote the space of bounded continuous functions on $G.$
For $k\geqslant 0,$ let $C^k_b(G)$ denote the space of $f \in C_b(G)$ such that for all words $\alpha$ with $\len(\alpha)\leqslant k$ we have
\[
X^{\alpha}f \in C_b(G).
\]
Define $\|f\|_{k,b} = \sup_{\len(\alpha) \leqslant k} \|X^{\alpha}f\|_{L_{\infty}(G)}$ and let $C^\infty_b(G) = \bigcap_{k\geqslant 0} C^k_b(G).$
\end{definition}

\begin{lemma}\label{multiplication_lemma}
If $\phi \in C^\infty_b(G),$ then the multiplier operator $M_{\phi}$ is bounded from $W^s_2(G)$ to $W^s_2(G)$ for all $s\in\mathbb{R}$ with norm no greater than $C_s \|\phi\|_{\lceil|s|\rceil,b}.$
\end{lemma}

\subsection{Covering Lemma}\label{subsec-3.2}

The following is proved in \cite{LMSZ25_elliptic}. The idea was used by Folland and Stein \cite[Lemma 7.14]{FollandStein1982}, see also \cite[Lemma 5.7.5]{FischerRuzhansky2016} and \cite[Lemma 6.2]{MSZ-stratified-23}.
\begin{lemma}{\cite[Lemma 3.11]{LMSZ25_elliptic}}\label{quasi_metric_space_construction}
{Let $(X,d)$ be a {separable} metric space, thus $X$ possesses a Borel measure $\mu$. If there exists a constant $\delta>0$ such that}
$$\mu(B(x,r))=r^{\delta},\quad x\in X,\quad r > 0.$$
{Then,} for every $\epsilon>0,$ there exists a set $\{x_i\}_{i\in I}\subset X$ such that
\begin{enumerate}[{\rm (i)}]
\item{} $\{B(x_i,\epsilon)\}_{i\in I}$ covers $X,$ and 
\item{} A fixed ball $B(x_i,\epsilon)$  {intersects} at most $5^{\delta}$ of balls $\{B(x_i,\epsilon)\}_{i\in I}.$
\item{} Moreover, for every $N\in\mathbb{N}$, a fixed ball $B(x_i, N\epsilon)$ {intersects} at most $(4N+1)^{\delta}$ of balls $\{B(x_i,N\epsilon)\}_{i\in I}.$
\end{enumerate}
\end{lemma}


\subsection{Uniformly Rockland operators}\label{subsec-LMSZ25-result}

In this subsection, we recall some definitions and main theorems from \cite{LMSZ25_elliptic}.
We start with definition of formal adjoint and extension of a differential operator.
\begin{definition}\label{def-P-dagg-ext}
For differential operator $P$ as in Definition \ref{def-differential-op} with expression \eqref{eq-P-expression}, its formal adjoint is denoted by $P^\dagger$, with formula:
    \begin{equation}\label{eq-P-dagg-expre}
        P^\dagger = \sum_{\len(\alpha)\leqslant m} (X^{\alpha})^\dagger M_{\overline{a_{\alpha}}} :\Sc(G)\to \Sc(G) ,
    \end{equation}
    where if $\alpha=\alpha_1\cdots\alpha_k,$ we have $(X^{\alpha})^\dagger = (-1)^kX_{\alpha_k}\cdots X_{\alpha_1}.$ \\
    Let $\widetilde{P}:\Sc'(G)\to\Sc'(G)$ denote the extension of $P$ to distributions, defined on $\omega \in \Sc'(G)$ by
    \begin{equation}\label{eq-P-tilde-expre}
        \langle \widetilde{P}\omega,\phi\rangle = \langle \omega,P^\dagger\phi\rangle,\quad \phi\in \Sc(G).
    \end{equation}
\end{definition}

\begin{remark}
    Since the monomials $X^{\alpha}$ are not linearly independent, the expression \eqref{eq-P-expression} for $P$ might not be unique, so it is not immediate that $P^{\dagger}$ is well-defined by \eqref{eq-P-dagg-expre}. However $P^\dagger$ is related to $P$ by the adjoint relation
    \begin{equation*}
        \langle P\omega,\phi\rangle = \langle \omega,P^\dagger\phi\rangle,\qquad \omega, \phi\in \Sc(G).
    \end{equation*}
    Hence, as an operator $P^\dagger$ is uniquely determined by the operator $P.$
\end{remark}

We present several properties of differential operators. By combining Lemmas \ref{lemma-X_alpha-bdd} and \ref{multiplication_lemma}, we obtain the boundedness of differential operators between Sobolev spaces.
\begin{lemma}\label{differential_operators_are_bounded} 
The distributional extension $\widetilde{P}$ of a differential operator $P$ of order at most $m$ restricts to a bounded linear operator from $W^{m+s}_2(G)$ to $W^s_2(G)$ for every $s\in \mathbb{R}.$
\end{lemma}

The following lemma shows two new operations $(\cdot)^{\dagger}$ (see, Definition \ref{def-P-dagg-ext}) and $(\cdot)^{\rm top}_g$ (see, Definition \ref{def-P-top-g}) commute.
\begin{lemma}{\cite[Lemma 4.7]{LMSZ25_elliptic}} \label{top dag commute}
$(P_g^{{\rm top}})^\dag = (P^{\dag})_g^{\rm top}$.
\end{lemma}
 
Furthermore, the product of uniformly Rockland operators are also uniformly Rockland.
\begin{lemma}{\cite[Lemma 4.12]{LMSZ25_elliptic}} \label{ellipticity of the product}
Let $P, Q$ be defined as in Definition \ref{def-differential-op}.
If $P$ and $Q$ are both uniformly Rockland, then so is $PQ$.
\end{lemma}
Now, we state the main theorems in \cite{LMSZ25_elliptic}.
\begin{theorem}{\cite[Theorem 1.12]{LMSZ25_elliptic}}\label{elliptic regularity theorem}
Let $P=P^{\dagger}$ be an uniformly Rockland order $m$ differential operator on $G.$
\begin{enumerate}[{\rm (i)}]
\item\label{erta} (Elliptic regularity) If $u\in L_2(G)$ is such that $\widetilde{P}u\in L_2(G),$ then $u\in W^m_2(G).$
\item\label{ertb} (Self-adjointness) {$\widetilde{P}$ is a self-adjoint operator on $L_2(G)$ with domain $W^m_2(G).$}
\end{enumerate}
\end{theorem}
Further, we have a quantitative result.
\begin{theorem}{\cite[Theorem 1.13]{LMSZ25_elliptic}}\label{gee_introduction}
    Let $P=P^\dagger$ be an  uniformly Rockland  order $m$ differential operator on $G.$ Then for every $s \in \mathbb{R},$ there exists $c_{P,s}>0$ such that for all real $c$ with $|c|$ sufficiently large, we have
    \[
        \|u\|_{W^{s+m}_2(G)} \leqslant c_{P,s}\|(P+ic)u\|_{W^s_2(G)},\quad u \in \Sc(G).
    \]
\end{theorem}

Further, the condition $P=P^{\dagger}$ could be removed.
\begin{theorem}{\cite[Theorem 1.14]{LMSZ25_elliptic}}\label{general elliptic estimate theorem 2} 
Let $P$ be a differential operator of order $m$. If $P$ is  uniformly Rockland, then for every $s\in \mathbb{R}$, there exist constants $c_{P,s,1}, c_{P,s,2}>0$ such that 
\begin{equation*}
	c_{P,s,1}\|u\|_{W^{s+m}_2(G)}\leqslant \|Pu\|_{W^s_2(G)} + c_{P,s,2}\|u\|_{L_2(G)} ,\quad u\in \Sc(G).
\end{equation*}
\end{theorem}

When $P\geqslant 0$ is positive definite, $i\cdot c$ in Theorem \ref{gee_introduction} can be replaced by $1$, since $-1$ is not in the spectrum of $P$. Thus, we have an adapted version 
\begin{theorem}\label{general elliptic estimate theorem}
	Let $P$ be a differential operator of order $m$. If $P$ is uniformly Rockland and $P\geqslant 0$, then for every $s\in \mathbb{R}$, there exist constants $c_{P,s}>0$ such that 
	\begin{equation*}
		\|u\|_{W^{s+m}_2(G)}\leqslant c_{P,s} \|(1+P)u\|_{W^s_2(G)}, \quad u\in \Sc(G).
	\end{equation*}
\end{theorem}

\begin{proof}
	Let $s = k m, k\in \Z$. 
	By combining Lemma \ref{differential_operators_are_bounded} with Theorem \ref{general elliptic estimate theorem}, for all positive integer $k,$ if $|c|$ is sufficiently large (depending on $k$), then
	\begin{equation}\label{eq-P-norm-equiv-Sobolev}
            \Vert u \Vert_{W^{km}_2(G)}
			\lesssim \Vert (P+ic)^k u \Vert_{L_2(G)}
			\lesssim \Vert u \Vert_{W^{km}_2(G)}, \quad u\in \Sc(G).
	\end{equation}
	where the implicit constants depend on $k.$
	By a duality argument we conclude the same for integer $k<0.$
	Since $P$ is positive definite, we have
	\begin{multline}\label{eq-P-pd-ineq}
		\Vert (P+1)u \Vert^2_{L_2(G)} 
		= \langle (P+1)^2 u, u \rangle \\
		= \langle P^2 u, u \rangle + 2\langle P u, u \rangle + \langle u, u \rangle
		\geqslant \Vert Pu \Vert^2_{L_2(G)} + \Vert u \Vert^2_{L_2(G)}, \quad u\in \Sc(G).
	\end{multline}
	Combining \eqref{eq-P-norm-equiv-Sobolev} and \eqref{eq-P-pd-ineq}, we obtain the following inequality
	\begin{align*}
		\|(1+P)u\|_{W^{km}_2(G)} 
		& \underset{\eqref{eq-P-norm-equiv-Sobolev}}{\gtrsim}  
			\Vert (P+ic)^k (P+1) u \Vert_{L_2(G)}
			=  \Vert (P+1)(P+ic)^k  u \Vert_{L_2(G)}\\
		& \underset{\eqref{eq-P-pd-ineq}}{\gtrsim}  
			\Vert P (P+ic)^k u \Vert_{L_2(G)} +  \Vert (P+ic)^k u \Vert_{L_2(G)}\\
		& \quad = \Vert (P+ic)^k P u \Vert_{L_2(G)} +  \Vert (P+ic)^k u \Vert_{L_2(G)}\\
		& \underset{\eqref{eq-P-norm-equiv-Sobolev}}{\gtrsim} \Vert  P u \Vert_{W^{km}_2(G)} +  \Vert u \Vert_{W^{km}_2(G)}, \qquad u\in \Sc(G).
	\end{align*}
	Applying Theorem \ref{general elliptic estimate theorem 2} yields
	\[
        \|(1+P)u\|_{W^{km}_2(G)} \gtrsim \|u\|_{W^{km+m}},\quad u\in \Sc(G).
	\]
	On the other hand since $P$ has order $m,$ we have
	\[
        \|u\|_{W^{km+m}} \gtrsim \|(1+P)u\|_{W^{km+m}_2(G)},\quad u\in \Sc(G).
	\]
	This proves the assertion for $s \in m\mathbb{Z}.$ The general case of $s\in \R$ then follows by an easy interpolation argument.
\end{proof}

\section{Trace Ideals}\label{sec-trace-ideal}

This section collects properties of trace ideals, trace norms, and noncommutative analysis, which will be useful in Section \ref{sec-5} and \ref{sec_abstract_asymptotics}. 

\subsection{Singular values and ideals}
The following terminology can be found in \cite[Chapter 1]{Simon-trace-ideals-2005}, \cite[Chapter 2 and 3]{GohbergKrein}, \cite[Chapter 1]{LSZ2021-singular-trace-v1} and \cite[Chapter 1]{LMSZ-SingularTrace-2023}.

For a Hilbert space $H$, denote by $\Bc(H)$ the algebra of all bounded linear operators on $H.$ The ideal of all compact
linear {operators} is denoted by $\Kc(H).$ The operator norm on $\Bc(H)$ is denoted by $\|\cdot\|_{\infty}.$ Given $T \in \Kc(H),$ the singular value sequence $\mu(T):=\{\mu(n,T)\}_{n=0}^\infty$ is defined in terms of the singular value function
\[
\mu(t,T) := \inf\{\|T-R\|_{\infty}\;:\; \mathrm{rank}(R)\leqslant t\},\quad t\geqslant 0.
\]
Note that $\mu(t, T)$ is a step function in the variable $t$. Equivalently, $\mu(n,T)$ is the $(n+1)$-st eigenvalue of $|T|$, arranged in decreasing order with multiplicities.

For any compact $A$ and bounded $B$ in $\Bc(H)$,
\begin{equation}\label{eq-singular-value-multi}
\mu(t,AB) \leqslant  \Vert B \Vert_{\infty} \mu(t,A),
\quad \mu(t,BA) \leqslant  \Vert B \Vert_{\infty} \mu(t,A).
\end{equation}

For $0 < p< \infty,$ the Schatten ideal $\Lc_p(H)$ consists of all compact linear operators $T$ such that
\begin{equation}\label{eq-def-p-norm}
\|T\|_p \coloneqq \left(\int_0^\infty \mu(t,T)^p dt\right)^{\frac1p}<\infty.
\end{equation}

We also make extensive use of the weak Schatten ideals, defined with the weak Schatten (quasi)-norms
\begin{equation}\label{eq-quasi-norm-L-p-infity}
\|T\|_{p,\infty} \coloneqq \sup_{t\geqslant 0} \ t^{\frac{1}{p}}\mu(t,T).
\end{equation}

We have the H\"older-type inequalities
\begin{equation}\label{eq-holder}
\|TS\|_r \leqslant \|T\|_p\|S\|_q,\quad \|TS\|_{r,\infty} \leqslant c_{p,q,r}\|T\|_{p,\infty}\|S\|_{q,\infty},\quad \frac{1}{r} = \frac{1}{p}+\frac{1}{q}.
\end{equation}
In this paper, the choice of Hilbert space is always clear from context; hence we abbreviate $\Lc_p(H)$ as $\Lc_p,$ similarly $\Lc_{p,\infty}(H)$ as $\Lc_{p,\infty}$ and so on.


For $p>0,$ we write $(\Lc_{p,\infty})_0$ for the closure of the finite rank operators in the $\Lc_{p,\infty}$ quasinorm. Equivalently, $T \in (\Lc_{p,\infty})_0$ if and only if
\begin{equation}\label{eq-singular-value-0-finite-rank-ideal}
\lim_{t\to\infty} t^{\frac1p}\mu(t,T) = 0.
\end{equation}
Combining with \eqref{eq-singular-value-multi}, we can see $(\Lc_{p,\infty})_0$ is also an ideal.
Note that if $\frac1p+\frac1q=\frac1r,$ then $(\Lc_{p,\infty})_0\cdot \Lc_{q,\infty} \subseteq (\Lc_{r,\infty})_0.$

\begin{remark}\label{remark-inclusion-Lpq}
\begin{equation*}
\Lc_{q, \infty} \subset \Lc_p \subset (\Lc_{p,\infty})_0 \subset \Lc_{p, \infty}, \quad \text{for} \ q<p. 
\end{equation*}
This comes from combining  \eqref{eq-def-p-norm}, \eqref{eq-quasi-norm-L-p-infity}, and \eqref{eq-singular-value-0-finite-rank-ideal}.
\end{remark}

\subsection{Auxiliary lemmas}

{This section was original 6.2, moved here.  We lists a few inequalities and formulas in weak Schatten ideals, which will be useful in our later proofs.}

Recall that two operators $A, B$ in $\Bc(H)$ are orthogonal if $AB = BA =0$.

\begin{lemma}\label{first disjoint lemma} If $(A_n)_{n\geqslant0}$ are pairwise orthogonal operators, then for $0<p<\infty$,
$$\|\sum_{n\geqslant0}A_n\|_{p,\infty}^p\leqslant\sum_{n\geqslant0}\|A_n\|_{p,\infty}^p.$$
\end{lemma}
\begin{proof} {Notice that}
$$\|\sum_{n\geqslant0}A_n\|_{p,\infty}=\|\sum_{n\geqslant0}\mu(A_n)\otimes \chi_{(n,n+1)}\|_{L_{p,\infty}(\mathbb{R}^2_+)}.$$
By definition, we have
$$\mu(A_n)\leqslant\|A_n\|_{p,\infty}z_p,\quad z_p(t)=t^{-\frac1p},\quad t>0.$$
Thus,
$$0\leqslant \sum_{n\geqslant0}\mu(A_n)\otimes \chi_{(n,n+1)}
    \leqslant \sum_{n\geqslant0}\|A_n\|_{p,\infty}z_p\otimes\chi_{(n,n+1)}
    =z_p\otimes(\sum_{n\geqslant0}\|A_n\|_{p,\infty}\chi_{(n,n+1)}).$$
Thus,
$$\|\sum_{n\geqslant0}A_n\|_{p,\infty}\leqslant \|z_p\otimes(\sum_{n\geqslant0}\|A_n\|_{p,\infty}\chi_{(n,n+1)})\|_{L_{p,\infty}(\mathbb{R}^2_+)}.$$
Recall that
$$\mu(z_p\otimes f)=\|f\|_pz_p,\quad f\in L_p(0,\infty).$$
The assertion follows immediately.
\end{proof}

\begin{lemma}\label{lemma-orthogonal-Pythagoras-2} 
If $(A_n)_{n=1}^{N}$ are pairwise orthogonal operators, then for $0<p<\infty$,
$${\rm dist}_{\mathcal{L}_{p,\infty}}^p(\sum_{n=1}^{N}A_n,(\mathcal{L}_{p,\infty})_0)\leqslant\sum_{n=1}^{N}{\rm dist}_{\mathcal{L}_{p,\infty}}^p(A_n,(\mathcal{L}_{p,\infty})_0).$$
\end{lemma}
\begin{proof} Fix $\epsilon>0$ and choose $\delta>0$ so small that
$$\sum_{n=1}^{N}({\rm dist}_{\mathcal{L}_{p,\infty}}(A_n,(\mathcal{L}_{p,\infty})_0)+\delta)^p\leqslant \sum_{n=1}^{N}{\rm dist}_{\mathcal{L}_{p,\infty}}^p(A_n,(\mathcal{L}_{p,\infty})_0)+\epsilon.$$
For every $1\leqslant  n \leqslant  N,$ choose $B_n\in (\mathcal{L}_{p,\infty})_0$ such that
$$\|A_n-B_n\|_{p,\infty}\leqslant {\rm dist}_{\mathcal{L}_{p,\infty}}(A_n,(\mathcal{L}_{p,\infty})_0)+\delta.$$
Set
$$C_n={\rm supp}_l(A_n)\cdot B_n\cdot{\rm supp}_r(A_n)\in (\mathcal{L}_{p,\infty})_0.$$
{Here, ${\rm supp}_l(A_n)$ means the left support of operator $A_n$, namely, the smallest projection such that ${\rm supp}_l(A_n)\cdot A = A$; ${\rm supp}_r(A_n)$ means the right support, the smallest projection such that $A\cdot {\rm supp}_l(A_n) = A$.}
Clearly,
$$A_n-C_n={\rm supp}_l(A_n)\cdot (A_n-B_n)\cdot{\rm supp}_r(A_n).$$
Thus,
$$\|A_n-C_n\|_{p,\infty}\leqslant {\rm dist}_{\mathcal{L}_{p,\infty}}(A_n,(\mathcal{L}_{p,\infty})_0)+\delta,\quad 1\leqslant  n \leqslant  N.$$
It is immediate that $\sum_{n=1}^{N}C_n\in (\mathcal{L}_{p,\infty})_0.$ We, therefore, have
$${\rm dist}_{\mathcal{L}_{p,\infty}}^p(\sum_{n=1}^{N}A_n,(\mathcal{L}_{p,\infty})_0)\leqslant \|\sum_{n=1}^{N}(A_n-C_n)\|_{p,\infty}^p.$$
Since the operators $\{A_n-C_n\}_{n=1}^N$ are pairwise orthogonal, it follows from Lemma \ref{first disjoint lemma} that
\begin{multline*}
\|\sum_{n=1}^{N}(A_n-C_n)\|_{p,\infty}^p\leqslant \sum_{n=1}^{N}\|A_n-C_n\|_{p,\infty}^p\\
\leqslant \sum_{n=1}^{N}({\rm dist}_{\mathcal{L}_{p,\infty}}(A_n,(\mathcal{L}_{p,\infty})_0)+\delta)^p\leqslant \sum_{n=1}^{N}{\rm dist}_{\mathcal{L}_{p,\infty}}^p(A_n,(\mathcal{L}_{p,\infty})_0)+\epsilon.
\end{multline*}
Since $\epsilon>0$ is arbitrarily small, the assertion follows.
\end{proof}

The next three lemmas are available in \cite{SXZ2023} (see Lemmas 2.4 and 2.5 there, respectively).

\begin{lemma}\label{third disjoint lemma}
If $(A_n)_{n=0}^{N-1}$ are pairwise orthogonal operators are such that
$$\lim_{t\to\infty}t^{\frac1p}\mu(t,A_n)=c_n,\quad 1\leqslant  n \leqslant  N,$$
then for $0<p<\infty$,
$$\lim_{t\to\infty}t^{\frac1p}\mu(\sum_{n=1}^{N}A_n)=(\sum_{n=1}^{N}c_n^p)^{\frac1p}.$$
\end{lemma}

\begin{lemma}\label{bs convergence lemma} For $0<p<\infty$, let $(A_l)_{l\geqslant0}\subset\mathcal{L}_{p,\infty}$ be such that
$$\lim_{t\to\infty}t^{\frac1p}\mu(t,A_l)=c_l,\quad l\geqslant0.$$
If
$${\rm dist}_{\mathcal{L}_{p,\infty}}(A_l-A,(\mathcal{L}_{p,\infty})_0)\to0,\quad l\to\infty,$$
then the following limits exist and are equal
$$\lim_{t\to\infty}t^{\frac1p}\mu(t,A)=\lim_{l\to\infty}c_l.$$
\end{lemma}

\begin{lemma}\label{easy_bs_lemma}
If $A-B \in (\Lc_{p,\infty})_0$ and if
$$\lim_{t\to\infty} t^{\frac1p}\mu(t,A)=c,$$
then also
$$\lim_{t\to\infty} t^{\frac1p}\mu(t,B)= c.$$
\end{lemma}

\subsection{The group von Neumann algebra and noncommutative measure theory}\label{sec_vna}
{
We will use some of the notation of semifinite traces on von Neumann algebras. For further details, see \cite{Fack-Kosaki,PX2003, DdPS-vapour, LSZ}. The only von Neumann algebras that concern us will be the algebra of bounded operators $\Bc(H)$ for a Hilbert space $H$, and the group von Neumann algebra $\mathrm{VN}(G)$ of a graded group $G.$

Given a von Neumann algebra $\mathcal{M}\subset \Bc(H)$ a closed operator $T:\mathrm{dom}(T)\to H$ is said to be affiliated with $\mathcal{M}$ if $T$ commutes with every unitary in the commutant of $\mathcal{M}.$ For self-adjoint operators $T,$ we can equivalently say that every spectral projection of $T$ belongs to $\mathcal{M}.$ 

Given a semifinite normal trace $\tau$ on $\mathcal{M},$ an operator $T:\mathrm{dom}(T)\to H$ affiliated to $\mathcal{M}$ is said to be $\tau$-measurable if for any $\varepsilon>0$ there exists a projection $p\in \mathcal{M}$ such that $\tau(p)< \varepsilon$ and $(1-p)H\subseteq \mathrm{dom}(T).$ 

The singular value function $t\mapsto \mu(t,T)$ of a $\tau$-measurable operator $T$ is defined as
\[
    \mu(t,T) := \inf\{\|T(1-p)\|_{\infty}\;:\; \tau(p)\leqslant t\}.
\]
This is compatible with the previous definition of the singular value function if $\tau$ is taken to be the classical trace $\mathrm{Tr}.$ For $0<p<\infty,$ the noncommutative $L_p$ space $L_p(\mathcal{M},\tau)$ is defined as the set of $\tau$-measurable operators $T$ such that
\[
    \|T\|_{p} := \left(\int_{0}^\infty \mu(t,T)^p\,dt\right)^{\frac1p} < \infty.
\]
Similarly, the weak $L_p$ space $L_{p,\infty}(\mathcal{M},\tau)$ is defined as the set of $\tau$-measurable operators $T$ such that
\[
    \|T\|_{p,\infty} := \sup_{t>0} t^{\frac1p}\mu(t,T) < \infty.
\]
Again, this is compatible with the Schatten and weak-Schatten norms when $\mathcal{M}=\Bc(H)$ and $\tau=\mathrm{Tr}.$

We will apply these notions when $\mathcal{M}$ is a group von Neumann algebra. We give a brief description of the relevant definitions here. For further details see e.g. \cite[Chapter 13]{DixmierCStar1977} and more specifically \cite{Stinespring-tams-1959}.}
Recall that every locally compact unimodular group $G$ has a distinguished representation, that is the left regular representation $\lambda$ on $L_2(G),$ where $L_2(G)$ is taken with respect to the bivariant Haar measure. For $u\in L_2(G)$ and any $g\in G$,  
define
\begin{equation*}
    \lambda(g)u(x) := u(g^{-1}x),\quad x\in G.
\end{equation*}
Since the Haar measure is left-invariant, $\lambda(g)$ is a unitary operator.
The same symbol is used to denote the map $\lambda: L_1(G) \to \mathcal{B}(L_2(G))$
defined by
\[
\lambda(f)u(x) = \int_{G} f(h)\lambda(h)u(x)dh =  \int_{G} f(h)u(h^{-1}x)dh.\quad x \in G,\quad u\in L_2(G).
\]
The group von Neumann algebra $\mathrm{VN}(G)$ is defined as the closure of $\lambda(L_1(G))$ in the weak operator topology of $\mathcal{B}(L_2(G)).$

We also consider the right regular representation, defined for $g\in G$ and $u\in L_2(G)$ by
\[
    \rho(g)u(x) := u(xg),\quad x \in G.
\]
Since the Haar measure is bivariant, $\rho(g)$ is also unitary.
The group von Neumann algebra $\mathrm{VN}(G)$ is the commutant of the right regular representation. That is,
\[
    \mathrm{VN}(G) = \rho(G)' = \{T\in \Bc(L_2(G))\,:\,\forall g\in G\, \rho(g)T=T\rho(g)\}.
\]

Let $C_c(G)$ denote the set of continuous functions on $G$ with compact support. On the algebra $\lambda(C_c(G))$, the Plancherel weight $\tau$ is the functional
\begin{equation}\label{eq-tau-lambda-f}
\tau (\lambda(f)) = f(1_G), \quad f\in C_c(G).
\end{equation}
Here, $1_G$ is the identity in $G.$ The Plancherel weight 
extends to a faithful normal semifinite trace on $\mathrm{VN}(G)$ \cite[Theorem 9.2]{Stinespring-tams-1959}.

In the case that concerns us, $G$ is equipped with a family of dilations $\{\delta_t\}_{t>0}.$ In this case, there is a corresponding action of $\mathbb{R}^{\times}$ on $\mathrm{VN}(G)$ by conjugation:
\begin{equation}\label{conjugation_action}
    \alpha_t(T)u(x) = T(u\circ \delta_t)(\delta_{t}^{-1}x),\quad t>0,\, T\in \mathrm{VN}(G),\, u\in L_2(G).
\end{equation}
Since the Lebesgue measure is rescaled by $t^{d_{\hom}}$ under $\delta_t,$ an easy computation shows that
\[
    \lambda(f\circ \delta_t^{-1}) = t^{d_{\hom}}\alpha_t(\lambda(f)),\quad t>0,\; f\in L_1(G).
\]
Since $\delta_t$ fixes the identity, we conclude that
\begin{equation}\label{trace_scaling}
    \tau(\alpha_t(T)) = t^{-d_{\hom}}\tau(T),\quad T \in L_1(\mathrm{VN}(G),\tau).
\end{equation}

\section{Asymptotics for operators with constant coefficients}\label{sec-5}

In this section, we prove a special case of our main Theorem \ref{main theorem special case}, that is, when the operator $P$ has constant coefficients and $\gamma = m$ . We will prove the general case in subsection \ref{Section-general-case}.
All the notations used here can be found in Sections \ref{sec-intro}, \ref{sec-Graded-Lie-Group}, \ref{sec-trace-ideal}, and \cite{LMSZ25_elliptic}.

\begin{theorem}\label{main theorem constant} If $P\geqslant 0$ is an uniformly Rockland, order $m$, differential operator with constant coefficients, then for all $0\leqslant f\in C^{\infty}_c(G),$
$$\lim_{t\to\infty}t^{\frac{m}{d_{\hom}}}\mu(t,M_f(P+1)^{-1})=\|f\|_{\frac{d_{\hom}}{m}}\cdot \left(\frac1{\Gamma(\frac{d_{\hom}}{m}+1)} \tau(e^{-P^{{\rm top}}})\right)^{\frac{m}{d_{\hom}}}.$$
\end{theorem}

The idea behind the proof is essentially the same as the corresponding result in \cite{MSZ-stratified-23}, specifically Lemma 7.8 there. Essential for stating the theorem is that $M_f(1+P)^{-1}$ belongs to the correct operator ideal, namely $\mathcal{L}_{\frac{d_{\hom}}{m},\infty}.$ {This is proved below in Lemma \ref{asterisque pre-verification lemma} as a consequence of the following Cwikel-type estimate:
\begin{lemma}\label{specific_cwikel_estimates_from_jfa_lemma}
    If $\alpha+\beta>0,$ and $f \in C^\infty_c(G),$ then
\begin{equation}\label{specific_cwikel_estimates_from_jfa}
(1-\Delta)^{-\frac{\alpha}{2v}}M_f(1-\Delta)^{-\frac{\beta}{2v}} \in \mathcal{L}_{\frac{d_{\hom}}{\alpha+\beta},\infty}.
\end{equation}
If $\alpha+\beta=0,$ then 
\[
    (1-\Delta)^{-\frac{\alpha}{2v}}M_f(1-\Delta)^{-\frac{\beta}{2v}}\in \Bc(L_2(G)).
\]
\end{lemma}
} 
{
This is very similar to \cite[Theorem 5.1 (iii)]{MSZ-stratified-23}, although there the group was stratified and $v=1.$ The proof in the general case is not substantially different, an outline is given in Appendix \ref{cwikel_appendix}.
} 

The following is a special case of Theorem 5.4.2 in \cite{SZ-Asterisque}, combined with \cite[Corollary 5.4.11]{SZ-Asterisque}.
\begin{theorem}\label{asterisque key result} Let $p>2$ and $0\leqslant A,B\in \mathcal{L}_{\infty}  .$ Suppose that
\begin{enumerate}[{\rm (i)}]
\item\label{ast1} $B^pA\in\mathcal{L}_{1,\infty}.$
\item\label{ast2} $B^{q-2}[B,A]\in\mathcal{L}_1$ for every $q>p.$
\item\label{ast3} $A^{\frac12}BA^{\frac12}\in\mathcal{L}_{p,\infty}.$
\item\label{ast4} $[B,A^{\frac12}]\in\mathcal{L}_{\frac{p}{2},\infty}.$
\end{enumerate}
Under these assumptions, for $\Re(z)>p$, $(A^{\frac12}BA^{\frac12})^z$ and $A^zB^z$ are trace class, and the function
$$z\mapsto {\rm Tr}(A^zB^z)-{\rm Tr}((A^{\frac12}BA^{\frac12})^z)$$
admits an analytic continuation to the half-plane $\{\Re(z)>p-1\}.$
\end{theorem}

\begin{lemma}\label{easy complex powers lemma} Let $P\geqslant 0$ be an uniformly Rockland order $m$ differential operator with constant coefficients. For every $z$ with $\Re(z)\geqslant 0,$ the mapping $(P+1)^{-z}:L_2(G)\to W^{m\Re(z)}_2(G)$ is bounded.
\end{lemma}
\begin{proof} If $\Re(z)=0,$ then the assertion is a consequence of the self-adjointness of $P$ and the functional calculus. If $\Re(z)=n\in\mathbb{N},$ then 
$$(P+1)^{-z}=(P+1)^{-n}\circ (P+1)^{-i\Im(z)}.$$
Since $P$ is uniformly Rockland, it follows from 
 Lemma \ref{ellipticity of the product} 
that so is $(P+1)^n$. 
Hence, by Theorem \ref{elliptic regularity theorem},  $(P+1)^n$ is self-adjoint with domain $W^{mn}_2(G).$ Thus, $(P+1)^{-n}:L_2(G)\to W^{mn}_2(G)$ is bounded. Since $(P+1)^{-i\Im(z)}:L_2(G)\to L_2(G),$ is unitary, the assertion for $\Re(z)=n$ follows. The assertion in full generality follows by interpolation of analytic families, see \cite[Theorem 4.2]{CCRSW_interpolation_1982}.
\end{proof}

\begin{lemma}\label{asterisque pre-verification lemma} Let $P\geqslant 0$ be an uniformly Rockland order $m$ differential operator with constant coefficients. Let $Q$ be an order $n$ differential operator with compactly supported coefficients. 
For any $\alpha, \beta \geqslant 0$, such that $m(\alpha + \beta) - n >0$, we have
$$(P+1)^{-\alpha}Q(P+1)^{-\beta}\in\mathcal{L}_{\frac{d_{\hom}}{m(\alpha+\beta)-n},\infty}.$$
\end{lemma}
\begin{proof} Without loss of generality, we assume $Q$ is monomial, that is, $Q = M_{a_{\alpha}}X^{\alpha}$ with $\len(\alpha) = n$ and $a_{\alpha}\in C_c^{\infty}(G)$. 
By Lemma \ref{easy complex powers lemma}, $(P+1)^{-\beta}:L_2(G)\to W^{m\beta}_2(G)$ is a bounded mapping. 
By Lemma \ref{lemma-X_alpha-bdd}, $X^{\alpha}: W^{m\beta}_2(G)\to W^{m\beta-n}_2(G)$ is bounded.
Thus, $(1-\Delta)^{\frac{m\beta-n}{{ 2v}}} X^{\alpha} (P+1)^{-\beta}\colon L_2(G) \to L_2(G) $ is bounded.
For the same reason, $(1-\Delta)^{\frac{m\alpha}{{ 2v}}}(P+1)^{-\alpha}:L_2(G)\to L_2(G)$ is bounded,  due to the self-adjointness of $P$ and $\Delta$, we obtain that $(P+1)^{-\alpha}(1-\Delta)^{\frac{m\alpha}{{ 2v}}}$ {has a bounded extension}. We now write
\begin{multline*}
(P+1)^{-\alpha}Q(P+1)^{-\beta}\\
=(P+1)^{-\alpha}(1-\Delta)^{\frac{m\alpha}{{ 2v}}}
\cdot (1-\Delta)^{-\frac{m\alpha}{{ 2v}}}M_{a_{\alpha}}(1-\Delta)^{-\frac{m\beta-n}{{ 2v}}} 
\cdot (1-\Delta)^{\frac{m\beta-n}{{ 2v}}} X^{\alpha} (P+1)^{-\beta}.
\end{multline*}
Taking into account that $m(\alpha+\beta)-n>0,$ we apply Lemma \ref{specific_cwikel_estimates_from_jfa_lemma} to conclude that the second factor, $(1-\Delta)^{-\frac{m\alpha}{{ 2v}}}M_{a_{\alpha}}(1-\Delta)^{-\frac{m\beta-n}{{ 2v}}},$ belongs to $\mathcal{L}_{\frac{d_{\hom}}{m(\alpha+\beta)-n},\infty}$. The other two factors are already proved to be bounded. By H\"older's inequality  \eqref{eq-holder}, we prove the assertion.
\end{proof}

The following lemma is nothing but a minor adaptation of the results of \cite[Section 5]{MSZ-stratified-23}.
\begin{theorem}\label{asterisque pre-verification lemma 2}
Let $P\geqslant 0$ be an uniformly Rockland order $m$ differential operator with constant coefficients, let $f \in C^\infty_c(G)$ and let $\alpha,\beta,\gamma\in \mathbb{R}.$ 
    \begin{enumerate}[\rm (i)]
        \item{} If $\alpha+\gamma+1=\beta,$ then the operator
        \[
        (1+P)^{-\frac{\alpha}{m}}[(1+P)^{\frac{\beta}{m}},M_f](1+P)^{-\frac{\gamma}{m}}
        \]
        is bounded on $L_2(G).$
        \item{} If $\alpha+\gamma+1>\beta,$ then 
        \[
        (1+P)^{-\frac{\alpha}{m}}[(1+P)^{\frac{\beta}{m}},M_f](1+P)^{-\frac{\gamma}{m}} \in \mathcal{L}_{\frac{d_{\hom}}{\alpha+\gamma+1-\beta},\infty}.
        \]
    \end{enumerate}
{All the operators above are \emph{a priori} well-defined as endomorphisms of the Schwartz space $\Sc(G),$ but in fact have bounded extension to $L_2(G).$ Without ambiguity, we use the same notation to denote their bounded extension.}
\end{theorem}
\begin{proof}
This can be proved by an almost verbatim repetition of the arguments of \cite[Section 5]{MSZ-stratified-23}, which proved the same assertions with $m=2$ and $P=-\Delta.$ Making minor modifications to take care of the case when $m>2,$ we can follow precisely the same arguments to conclude the result. {The argument is supplied in more detail, in wider generality, in Appendix \ref{psido_commutator_appendix}.}
\end{proof}

\begin{lemma}\label{asterisque verification lemma} Let $P\geqslant 0$ be an uniformly Rockland order $m$ differential operator with constant coefficients. Let $0\leqslant f\in C^{\infty}_c(G).$ If $d_{\hom}>2,$ then the assumptions in Theorem \ref{asterisque key result} are met for $p=d_{\hom}$, $A=M_{f^2}$, and $B=(P+1)^{-\frac{1}{m}}.$
\end{lemma}
\begin{proof} 
Using Lemma \ref{asterisque pre-verification lemma} with $Q=M_{f^2},$ $\alpha=\frac{d_{\hom}}{m}$ and $\beta=0,$ we infer
$$B^pA = (P+1)^{-\frac{d_{\hom}}{m}}M_{f^2}\in\mathcal{L}_{1,\infty}.$$
This yields the condition \eqref{ast1} in Theorem \ref{asterisque key result}.

Similarly, using Lemma \ref{asterisque pre-verification lemma} with $Q=M_f,$ $\alpha=\frac1{m}$ and $\beta=0,$ we infer
$$A^{\frac12}BA^{\frac12} = M_f(P+1)^{-\frac{1}{m}}M_f\in\mathcal{L}_{d_{\hom},\infty}.$$
This yields the condition \eqref{ast3} in Theorem \ref{asterisque key result}.

Similarly, applying Lemma \ref{asterisque pre-verification lemma 2} yields
\[
B^{q-2}[B,A] = (1+P)^{-\frac{(q-2)}{m}}[(1+P)^{-\frac{1}{m}},M_f^2] \in \mathcal{L}_{\frac{d_{\hom}}{q},\infty}\subset \mathcal{L}_1 
\]
for $q>d_{\hom}.$ This yields \eqref{ast2} in Theorem \ref{asterisque key result}.
Finally, Lemma \ref{asterisque pre-verification lemma 2} implies
\[
[B,A^{\frac12}] = [(1+P)^{-\frac{1}{m}},M_f] \in \mathcal{L}_{\frac{d_{\hom}}{2},\infty}.
\]
This verifies Condition \eqref{ast4}.	
\end{proof}

\begin{lemma}\label{asterisque summary lemma} Let $P\geqslant 0$ be an uniformly Rockland order $m$ differential operator with constant coefficients. Assume that $d_{\hom}>2.$ If $0\leqslant f\in C^{\infty}_c(G),$ then the mapping
$${\rm Tr}(M_{f^{2z}}(P+1)^{-\frac{z}{m}})-{\rm Tr}\Big(\Big(M_f(P+1)^{-\frac1{m}}M_f\Big)^z\Big),\quad\Re(z)>d_{\hom},$$
extends analytically to a half-plane $\{\Re(z)>d_{\hom}-1\}$.
{Here, $P+1\geqslant 1 $ is positive definite with spectrum away from $0$, thus, by spectral theory, operator $(P+1)^{-\frac{z}{m}}$ is well-defined.}
\end{lemma}
\begin{proof} The assertion follows from Theorem \ref{asterisque key result} and Lemma \ref{asterisque verification lemma}. 
\end{proof}

The following lemma implicitly shows that one of the traces in Lemma \ref{asterisque summary lemma} admits an analytic continuation, by explicit computation. This generalizes  \cite[Corollary 7.2]{MSZ-stratified-23}.

\begin{lemma}\label{wiener-ikehara compute lemma} 
Let $P\geqslant 0$ be an uniformly Rockland and homogeneous order $m$ differential operator with constant coefficients. Let $0\leqslant f\in C^{\infty}_c(G).$ We have
$${\rm Tr}(M_{f^{2z}}(P+1)^{-\frac{z}{m}})=\frac{\Gamma(\frac{z-d_{\hom}}{m})}{\Gamma(\frac{z}{m})}\cdot\int_Gf^{2z}\cdot \tau(e^{-P}),\quad \Re(z)>d_{\hom}.$$
\end{lemma}
\begin{proof} { For $\Re(z)>d_{\hom},$} we write
$$M_f(P+1)^{-\frac{z}{m}}=M_f(1-\Delta)^{-\frac{z}{{ 2v}}}\cdot (1-\Delta)^{\frac{z}{{ 2v}}}(P+1)^{-\frac{z}{m}}.$$
The first factor belongs to $\mathcal{L}_1(L_2(G))$ by Lemma \ref{specific_cwikel_estimates_from_jfa_lemma}. The second factor is bounded by Lemma \ref{easy complex powers lemma}. Thus, $M_f(P+1)^{-\frac{z}{m}}\in \mathcal{L}_1(L_2(G)).$ This immediately yields $M_{f^z}(P+1)^{-\frac{z}{m}}\in \mathcal{L}_1(L_2(G)).$  

If $f_1,f_2\in L_2(G)$ and if $T\in L_1({\rm VN}(G),\tau),$ then $M_{f_1}TM_{f_2}\in\mathcal{L}_1(L_2(G))$ and
$$\tau(M_{f_1}TM_{f_2})=\int_G f_1f_2\cdot \tau(T).$$
Hence,
\begin{multline*}
    {\rm Tr}(M_{f^{2z}}(P+1)^{-\frac{z}{m}})={\rm Tr}(M_{f^z}\cdot M_{f^z}(P+1)^{-\frac{z}{m}})\\
    ={\rm Tr}(M_{f^z}(P+1)^{-\frac{z}{m}}M_{f^z})=\int_Gf^{2z}\cdot \tau((P+1)^{-\frac{z}{m}}).
\end{multline*}
Using Lemma \ref{lem:continuous_limit_formula} {in the Appendix}, we write (it is exactly at this point we are using homogeneity of the operator $P$)
\begin{align*}
\tau((P+1)^{-\frac{z}{m}})
    & =\frac1{\Gamma(\frac{d_{{\rm hom}}}{m})}\int_0^{\infty}x^{\frac{d_{{\rm hom}}}{m}-1}(x+1)^{-\frac{z}{m}}dx\cdot \tau(e^{-P})\\
    & \stackrel{x=\frac{u}{1-u}}{=}\frac1{\Gamma(\frac{d_{{\rm hom}}}{m})}\int_0^1u^{\frac{d_{{\rm hom}}}{m}-1}(1-u)^{\frac{z-d_{{\rm hom}}}{m}-1}du\cdot \tau(e^{-P})\\
    & =\frac1{\Gamma(\frac{d_{{\rm hom}}}{m})} B(\frac{d_{{\rm hom}}}{m},\frac{z-d_{{\rm hom}}}{m})\tau(e^{-P})
    =\frac{\Gamma(\frac{z-d_{\hom}}{m})}{\Gamma(\frac{z}{m})}\tau(e^{-P}). 
\end{align*}
This completes the proof.
\end{proof}

The following is the Wiener-Ikehara theorem \cite[Theorem 14.1]{Shubin-psido-2001}, stated in the language of operator theory.
\begin{theorem}\label{wiener-ikehara} Let $p>0$ and let $0\leqslant V\in\mathcal{L}_{1,\infty}.$ If there exists $c_V>0$ such that the function
$$z\to {\rm Tr}(V^z)-\frac{c_V}{z-p},\quad \Re(z)>p,$$
extends continuously to the closed half-plane $\{\Re(z)\geqslant p\},$ then 
$$\lim_{t\to\infty}t\mu(t,V)^p=\frac{c_V}{p}.$$
\end{theorem}

\begin{lemma}\label{wiener-ikehara summary lemma} Let $P\geqslant 0$ be an uniformly Rockland order $m$ and homogeneous differential operator with constant coefficients. Assume that $d_{\hom}>2.$ Let $0\leqslant f\in C^{\infty}_c(G).$ Then
$$\lim_{t\to\infty}t\mu(t,M_f(P+1)^{-\frac{1}{m}}M_f)^{d_{\hom}}= 
\frac1{\Gamma(\frac{d_{\hom}}{m}+1)}\int_Gf^{2d_{\hom}}\cdot \tau(e^{-P}).$$ 
\end{lemma}
\begin{proof} Set
$$A=M_{f^2},\quad B=(P+1)^{-\frac{1}{m}},\quad V=A^{\frac12}BA^{\frac12},\quad c_V=\frac{m}{\Gamma(\frac{d_{\hom}}{m})}\int_Gf^{2d_{\hom}}\cdot \tau(e^{-P}).$$
By Lemma \ref{wiener-ikehara compute lemma} and properties of Gamma function, the function
$$z\to{\rm Tr}(A^zB^z),\quad \Re(z)>d_{\hom},$$
admits an analytic continuation to punctured half-plane $\{\Re(z)>0\}\backslash\{d_{\rm hom}-mk, k=0,1,2,\cdots,\floor*{\frac{d_{\hom}}{m}} \}$ (in particular, to the punctured half-plane $\{\Re(z)>d_{\hom}-1\}\backslash\{d_{\hom}\}$). Also by Lemma \ref{wiener-ikehara compute lemma} and property of Gamma function, the point $z=d_{\hom}$ is a simple pole and
$${\rm Res}_{z=d_{\hom}}{\rm Tr}(A^zB^z)=c_V.$$
On the other hand, it follows from Lemma \ref{asterisque summary lemma} that the function
$$z\to{\rm Tr}(A^zB^z)-{\rm Tr}(V^z),\quad \Re(z)>d_{\hom},$$
admits an analytic continuation to half-plane $\{\Re(z)>d_{\hom}-1\}.$ Hence, the function
$$z\to {\rm Tr}(V^z),\quad \Re(z)>d_{\hom},$$
admits an analytic continuation to punctured half-plane $\{\Re(z)>d_{\hom}-1\}\backslash\{d_{\hom}\}.$ The point $z=d_{\hom}$ is a simple pole and
$${\rm Res}_{z=d_{\hom}}{\rm Tr}(V^z)=c_V.$$
Hence, the operator $V$ satisfies the assumptions in Theorem \ref{wiener-ikehara}. Applying Theorem \ref{wiener-ikehara}, we complete the proof.
\end{proof}

In the following lemmas we repeatedly use the principle that if $A$ and $B$ are positive operators then
\begin{equation}\label{general_raising_powers_principle}
A,B\in \Lc_{p,\infty},\; A-B \in (\mathcal{L}_{p,\infty})_0\Rightarrow A^\gamma-B^\gamma \in (\mathcal{L}_{\frac{p}{\gamma},\infty})_0.
\end{equation} 
If $\gamma$ is an integer, we do not need to assume that $A$ and $B$ are positive, and this follows from H\"older's inequality and elementary algebra. For the case of general $\gamma>0,$ see \cite[Lemma 1.3.18]{LMSZ-SingularTrace-2023}. Note that if $0<\gamma<1,$ then we do not need to assume that $A,B \in \Lc_{p,\infty},$ only that $A-B \in (\Lc_{p,\infty})_0$. {This is Ricard's strengthening of the Birman-Koplienko-Solomyak Theorem \cite{Birman-Koplienko-Solomjak1975,Ricard-2018}}.

\begin{lemma}\label{constant asymp aux lemma} Let $P\geqslant 0$ be an uniformly Rockland order $m$ differential operator with constant coefficients. For every positive $f\in C^{\infty}_c(G),$ $\alpha>0$ and integers $k>0,$ we have
\begin{enumerate}[{\rm (i)}]
\item{}\label{constant asymp aux lemma 1}
$$\big(M_f(P+1)^{-\frac{\alpha}{m}}M_f\big)^k-M_{f^{2k}}(P+1)^{-\frac{k\alpha}{m}}\in (\mathcal{L}_{\frac{d_{\hom}}{k\alpha},\infty})_0$$
\item{}\label{constant asymp aux lemma 2}
\[
(M_f(P+1)^{-\frac{\alpha}{m}}M_f)^k-M_{f^{k}}(P+1)^{-\frac{k\alpha}{m}}M_{f^k}\in (\mathcal{L}_{\frac{d_{\hom}}{k\alpha},\infty})_0
\]
\end{enumerate}
\end{lemma}
\begin{proof} 
{\bf Step 1:}
By Lemma \ref{asterisque pre-verification lemma 2}, we have
\[
M_f(P+1)^{-\frac{\alpha}{m}}M_f-M_f^2(1+P)^{-\frac{\alpha}{m}} \in (\mathcal{L}_{\frac{d_{\hom}}{\alpha},\infty})_0.
\]
Applying \eqref{general_raising_powers_principle}, we conclude that for all integers $k>0$ that
$$\big(M_f(P+1)^{-\frac{\alpha}{m}}M_f\big)^k-\big(M_{f^2}(P+1)^{-\frac{\alpha}{m}}\big)^k\in (\mathcal{L}_{\frac{d_{\hom}}{k\alpha},\infty})_0$$

{\bf Step 2:} We claim that for all integers $k>0$
$$\big(M_{f^2}(P+1)^{-\frac{\alpha}{m}}\big)^k-M_{f^{2k}}(P+1)^{-\frac{k\alpha}{m}}\in (\mathcal{L}_{\frac{d_{\hom}}{k\alpha},\infty})_0.$$

We prove the claim by induction on $k.$ The case $k=1$ is trivial.

Suppose the claim is true for $k\geqslant 1.$ We write
\begin{align*}
& \big(M_{f^2}(P+1)^{-\frac{\alpha}{m}}\big)^{k+1}-M_{f^{2k+2}}(P+1)^{-\frac{\alpha(k+1)}{m}}\\
& \quad = \big(M_{f^2}(P+1)^{-\frac{\alpha}{m}}\big)^{k+1}
- M_{f^{2k}} \big[ M_{f^{2}}, (P+1)^{-\frac{k\alpha}{m}} \big ] (P+1)^{-\frac{\alpha}{m}}\\
& \qquad 	- M_{f^{2k}} (P+1)^{-\frac{k\alpha}{m}} M_{f^{2}} (P+1)^{-\frac{\alpha}{m}}  \\
& \quad	= \big(\big(M_{f^2}(P+1)^{-\frac{\alpha}{m}}\big)^{k} - M_{f^{2k}} (P+1)^{-\frac{k\alpha}{m}} \big) M_{f^{2}} (P+1)^{-\frac{\alpha}{m}} \\
& \qquad	- M_{f^{2k}} \big[ M_{f^{2}}, (P+1)^{-\frac{k\alpha}{m}} \big ] (P+1)^{-\frac{\alpha}{m}}.
\end{align*}

The first summand belongs to $ (\mathcal{L}_{\frac{d_{\hom}}{\alpha(k+1)},\infty})_0$ by the inductive hypothesis and H\"older's inequality. 
The second summand belongs to $ (\mathcal{L}_{\frac{d_{\hom}}{\alpha(k+1)},\infty})_0$ by Lemma \ref{asterisque pre-verification lemma 2}.
Hence, the right hand side belongs to $ (\mathcal{L}_{\frac{d_{\hom}}{\alpha(k+1)},\infty})_0$ and so does the left hand side. This yields the step of induction and the claim follows.

{\bf Step 3:} Part \ref{constant asymp aux lemma 1} follows by combining Step 1 and Step 2. 

By Lemma \ref{asterisque pre-verification lemma 2}, we have
\[
M_{f^{2k}}(1+P)^{-\frac{k\alpha}{m}}-M_f^k(1+P)^{-\frac{k\alpha}{m}}M_f^k \in (\Lc_{\frac{d_{\hom}}{k\alpha},\infty})_0.
\]
Hence, we also have \ref{constant asymp aux lemma 2}.
\end{proof}

\begin{lemma}\label{constant_asymptotics_power_k}
Let $k\in \mathbb{N},$ and let $d_{\hom}>2.$ If $P\geqslant 0$ is an order $m$ homogeneous uniformly Rockland differential operator {with constant coefficients}, then 
\[
\lim_{t\to\infty}t^{\frac{m}{d_{\hom}}}\mu(t,M_f^{2m}(P+1)^{-1})=
\left(\frac1{\Gamma(\frac{d_{\hom}}{m}+1)}\int_Gf^{2d_{\hom}}\cdot \tau(e^{-P})\right)^{\frac{m}{d_{\hom}}} 
\]
\end{lemma}
\begin{proof}
Raising the equality in Lemma \ref{wiener-ikehara summary lemma} to the power $2m,$ we have
\[
\lim_{t\to\infty}t^{2m}\mu(t,(M_f(P+1)^{-\frac1m}M_f)^{2m})^{d_{\hom}} = \left(\frac1{\Gamma(\frac{d_{\hom}}{m}+1)}\int_Gf^{2d_{\hom}}\cdot \tau(e^{-P})\right)^{2m}.
\]
By Lemma \ref{constant asymp aux lemma}, it follows that
\[
\lim_{t\to\infty} t^{2m}\mu(t,M_{f}^{2m}(P+1)^{-2}M_f^{2m})^{d_{\hom}} = \left(\frac1{\Gamma(\frac{d_{\hom}}{m}+1)}\int_Gf^{2d_{\hom}}\cdot \tau(e^{-P})\right)^{2m}.
\]
Now raising to the power $\frac{1}{2d_{\hom}},$ and using the identity $\mu(T^*T)=\mu(T)^2,$ we conclude the argument.
\end{proof}

\begin{proof}[Proof of Theorem \ref{main theorem constant}] Choose $0\leqslant \phi\in C^\infty_c(G)$ be such that $f=\phi f$ and $\phi^{\frac{1}{2m}}$ is smooth. Set $h_{\epsilon}=(f+\epsilon\phi)^{\frac1{2m}}\in C^\infty_c(G).$ Since $P-P^{{\rm top}}$ has order $m-1,$ it follows from the resolvent identity that
\[
M_{h_{\epsilon}}^{2m}(1+P)^{-1}-M_{h_{\epsilon}}^{2m}(1+P^{{\rm top}})^{-1} \in (\Lc_{\frac{d}{m},\infty})_0.
\]
Hence, by Lemma \ref{easy_bs_lemma}, for all $f \in C^\infty_c(G)$ we have
$$\lim_{t\to\infty}t^{\frac{m}{d_{\hom}}}\mu(t,M_{h_{\epsilon}}^{2m}(P+1)^{-1})=\|h_{\epsilon}^{2m}\|_{\frac{d_{\hom}}{m}}\cdot \left(\frac1{\Gamma(\frac{d_{\hom}}{m}+1)} \tau(e^{-P^{{\rm top}}})\right)^{\frac{m}{d_{\hom}}}.$$
In other words,
$$\lim_{t\to\infty}t^{\frac{m}{d_{\hom}}}\mu(t,M_{f+\epsilon\phi}(P+1)^{-1})=\|f+\epsilon\phi\|_{\frac{d_{\hom}}{m}}\cdot \left(\frac1{\Gamma(\frac{d_{\hom}}{m}+1)} \tau(e^{-P^{{\rm top}}})\right)^{\frac{m}{d_{\hom}}}.$$
We then deduce the general result from Lemma \ref{bs convergence lemma}. 
\end{proof}

\section{Proof of spectral asymptotic theorem}\label{sec_abstract_asymptotics}


In this section, we prove an abstract theorem on continuity of singular value asymptotic. That is, if a set of operators $\{A_g\}_{g\in G}\subset\mathcal{L}_{p,\infty}$ approximates operator $A\in\mathcal{L}_{p,\infty}$ in some sense, and each $A_g$ satisfies some singular value asymptotic, then operator inherits a similar singular value asymptotic.

Throughout the section we are going to use the notation defined in the following theorem. 

\begin{theorem}\label{abstract asymptotics theorem} 
Let $0<p<\infty,$ $A\in\mathcal{L}_{p,\infty}$ and let $\{A_g\}_{g\in G}\subset\mathcal{L}_{p,\infty}.$ Let $a\in C_b(G,\mathbb{C}^{N})$, $N\geqslant 1$  be uniformly continuous and let $f\in L_p(G),$ $\alpha\in C_b(G).$ All other notations below are defined in Section \ref{sec-Graded-Lie-Group} and \ref{sec-trace-ideal}. Suppose that
\begin{enumerate}[{\rm (i)}]
\item\label{aata} for every $\phi\in C^{\infty}_c(G)$ and for every $g\in G,$ there exists a limit
$$\lim_{t\to\infty}t^{\frac1p}\mu(t,M_{\phi}A_g)=\alpha(g)\|f\phi\|_{L_p(G)}.$$
\item\label{aatb} for every $\phi\in C^{\infty}_c(G)$ and for every $g\in G,$ we have
$${\rm dist}_{\mathcal{L}_{p,\infty}}(M_{\phi}(A-A_g),(\mathcal{L}_{p,\infty})_0)\leqslant\|f\phi\|_{L_p(G)}\|a-a(g)\|_{L_{\infty}({\rm supp}(\phi))}.$$
\item\label{aatc} for every $\phi\in C^{\infty}_b(G),$ we have
$${\rm dist}_{\mathcal{L}_{p,\infty}}(M_{\phi}A,(\mathcal{L}_{p,\infty})_0)\leqslant\|f\phi\|_{L_p(G)}.$$
\item\label{aatd} for every $\phi\in C^{\infty}_c(G)$ and for every $g\in G,$ we have
$$[A,M_{\phi}],[A_g,M_{\phi}]\in(\mathcal{L}_{p,\infty})_0.$$
\end{enumerate}
Under those assumptions, there exists a limit
$$\lim_{t\to\infty}t^{\frac1p}\mu(t,A)=\|f\alpha\|_{L_p(G)}.$$	 
\end{theorem}

Theorem \ref{abstract asymptotics theorem} is the key to our entire approach. In Section \ref{Section-special-case}, we apply it with $A = M_f(1+P)^{-1}$ where $f\in C_c^\infty(G)$ and $P\geqslant 0$ is an uniformly Rockland differential operator of order $m>0.$ The local approximation will be $A_g = M_f(1+P^{{\rm top}}_g)^{-1}.$

\subsection{Proof of abstract asymptotic theorem}

In the following lemmas, $\omega(a,\cdot)$ denotes the modulus of continuity of a uniformly continuous function $a:G\to\mathbb{C}^{N}.$ That is,
\begin{equation}\label{eq-def-onega-a}
\omega(a,h) \coloneqq \sup_{\mathrm{dist}(g_1,g_2)\leqslant h} \|a(g_1)-a(g_2)\|_{\mathbb{C}^N},\quad h>0.
\end{equation}
Here, $\mathrm{dist}$ is a translation invariant and homogeneous metric on $G.$ It is well-defined (finite) because $a$ is uniformly continuous.

Conditions \ref{aatb} asserts that we can approximate $A$ by $A_g$ locally. The following lemma shows that the family $\{A_g\}_{g\in G}$ can be patched together to an operator that globally approximates $A.$

\begin{lemma}\label{abstract final lemma} Assume the conditions in Theorem \ref{abstract asymptotics theorem} are satisfied. Let $\{ \phi_n \}_{n=1}^{N}\subset C^{\infty}_c(G)$ be a sequence of pairwise disjointly supported functions. Choose a set of points $\{ g_n \}_{n=1}^{N}\subset G$, such that each point $g_n \in {\rm supp}(\phi_n)$, $1\leqslant  n \leqslant  N$. We have
\begin{multline}\label{eq-distance-op-thm}
{\rm dist}_{\mathcal{L}_{p,\infty}}\Big(A-\sum_{n=1}^{N}M_{\phi_n}A_{g_n}M_{\phi_n},(\mathcal{L}_{p,\infty})_0\Big) \leqslant 2^{\frac1p}\Big[\|f(1-\sum_{n=1}^{N}\phi_n^2)\|_{L_p(G)}\\
+\max_{1\leqslant  n \leqslant  N}\|\phi_n\|_{L_\infty(G)}^2\cdot\|f\|_{L_p(G)}\cdot \omega\big(a,\max_{1\leqslant  n \leqslant  N}{\rm diam}({\rm supp}(\phi_n))\big)\Big].
\end{multline}
\end{lemma}
\begin{proof} To lighten the notations, we write ${\rm dist}$ instead of ${\rm dist}_{\mathcal{L}_{p,\infty}}$ everywhere in this proof.

By assumption \ref{aatd}, we have

$$M_{\sum_{n=1}^{N}\phi_n^2}A-\sum_{n=1}^{N}M_{\phi_n}AM_{\phi_n}=\sum_{n= 1}^{N}M_{\phi_n}\cdot [M_{\phi_n},A] \in (\mathcal{L}_{p,\infty})_0.$$
Thus,  
\begin{multline}\label{eq-distance-op-decompose}
{\rm dist}\Big(M_{\sum_{n=1}^{N}\phi_n^2}A-\sum_{n=1}^{N}M_{\phi_n}A_{g_n}M_{\phi_n},(\mathcal{L}_{p,\infty})_0\Big)\\
={\rm dist}\Big(\sum_{n=1}^{N}M_{\phi_n}(A-A_{g_n})M_{\phi_n},(\mathcal{L}_{p,\infty})_0\Big).
\end{multline}

Since the functions $\{\phi_n\}_{n=1}^{N}$ are pairwise disjointly supported, it follows from Lemma \ref{lemma-orthogonal-Pythagoras-2} that
\begin{multline}\label{eq-distance-op-decompose-2}
\Big[{\rm dist}\Big(M_{\sum_{n=1}^{N}\phi_n^2}A-\sum_{n=1}^{N}M_{\phi_n}A_{g_n}M_{\phi_n},(\mathcal{L}_{p,\infty})_0\Big)\Big]^p \\
\leqslant \sum^{N}_{n = 1}\Big[{\rm dist}\Big(M_{\phi_n}(A-A_{g_n})M_{\phi_n},(\mathcal{L}_{p,\infty})_0\Big)\Big]^p.
\end{multline}

Now let us analyse each summand. For all $1\leqslant  n \leqslant  N$, we have
\begin{multline}\label{eq-distance-op-decompose-3}
{\rm dist}\Big(M_{\phi_n}(A-A_{g_n})M_{\phi_n},(\mathcal{L}_{p,\infty})_0\Big) \leqslant \|\phi_n\|_{L_{\infty}(G)}{\rm dist}\Big(M_{\phi_n}(A-A_g),(\mathcal{L}_{p,\infty})_0\Big)\\
\leqslant \|\phi_n\|_{L_{\infty}(G)}\cdot \|f\phi_n\|_{L_p(G)}\|a-a(g_n)\|_{L_{\infty}({\rm supp}(\phi_n))}.
\end{multline}
Here, the first inequality follows from the obvious fact
$${\rm dist}_{\mathcal{L}_{p,\infty}}\Big(XY,(\mathcal{L}_{p,\infty})_0\Big)\leqslant {\rm dist}_{\mathcal{L}_{p,\infty}}\Big(X,(\mathcal{L}_{p,\infty})_0\Big)\|Y\|_{\infty},$$
and the second inequality follows from the assumption \ref{aatb}.

Combining \eqref{eq-distance-op-decompose-2} and \eqref{eq-distance-op-decompose-3}, we obtain
\begin{multline}\label{eq-distance-op-decompose-5}
\Big[ {\rm dist}\Big(M_{\sum_{n=1}^{N}\phi_n^2}A-\sum_{n=1}^{N}M_{\phi_n}A_{g_n}M_{\phi_n},(\mathcal{L}_{p,\infty})_0\Big)\Big]^p \\
\leqslant  \max_{1\leqslant  n \leqslant  N}\|\phi_n\|_{L_{\infty}(G)}^{p}\cdot [\omega(a,\max_{1\leqslant  n \leqslant  N}{\rm diam}({\rm supp}(\phi_n)))]^p\cdot \sum_{n=1}^N \|f \phi_n \|_{L_p(G)}^p \\
\leqslant \max_{1\leqslant  n \leqslant  N}\|\phi_n\|_{L_{\infty}(G)}^{2p}\cdot [\omega(a,\max_{1\leqslant  n \leqslant  N}{\rm diam}({\rm supp}(\phi_n)))]^p\cdot \|f\|_{L_p(G)}^p.
\end{multline}

On the other hand, it follows from the assumption \ref{aatc} that
\begin{equation}\label{eq-distance-op-decompose-6}
{\rm dist}\Big(A-M_{\sum_{n=1}^{N}\phi_n^2}A,(\mathcal{L}_{p,\infty})_0\Big)\leqslant \|f(1-\sum_{n=1}^{N}\phi_n^2)\|_{L_p(G)}.
\end{equation}
Combining \eqref{eq-distance-op-decompose-5} and \eqref{eq-distance-op-decompose-6} and using the triangle inequality, we complete the proof.
\end{proof}

Naturally, we wish to have a sequence of pairwise disjointly supported functions $\{\phi_n\}_{n=1}^N\subset C^{\infty}_c(G)$ such that the right hand side in \eqref{eq-distance-op-thm} is arbitrarily small. In the next lemma, we explicitly construct such functions. Moreover, the assertion of the Lemma \ref{function system lemma} is general, that is, the construction can be made as long as the space $G$ satisfies the conditions of Lemma \ref{quasi_metric_space_construction}, not necessarily have to be nilpotent Lie group.

\begin{lemma}\label{function system lemma} There exists a sequence $\{\phi_{n,l}\}_{l\geqslant 0, 1\leqslant  n \leqslant  N_l}\subset C^{\infty}_c(G)$ such that
\begin{enumerate}[{\rm (i)}]
\item $0\leqslant \phi_{n,l}\leqslant 1$ for every $l\geqslant 0$ and for every $1\leqslant  n \leqslant  N_l.$
\item For every $l\geqslant 0,$ the functions $\{\phi_{n,l}\}_{n=1}^{N_l}$ are pairwise disjointly supported.
\item $\sum_{1\leqslant  n \leqslant  N_l}\phi_{n,l}^2\to 1$ almost everywhere as $l\to\infty.$
\item For every $l\geqslant 0,$ $1\leqslant n\leqslant N_l,$ ${\rm diam}(\supp{\phi_{n,l}})\leqslant 2^{-l}.$
\end{enumerate}
\end{lemma}
\begin{proof}
For a given $l\geqslant 0,$
by Lemma \ref{quasi_metric_space_construction},  
we can choose a sequence $\{x_{n,l}\}_{n\geqslant 1}\subset G$, 
such that $$\bigcup_{n\geqslant 1}B(x_{n,l},2^{-l})=G.$$

By compactness of closed ball in $G$, there exists $N_l$,
such that $$B(1_G,2^l)\subset\bigcup_{1\leqslant  n \leqslant  N_l}B(x_{n,l},2^{-l}).$$
We can rewrite the right hand side to be disjoint union. 
Denote
$$B_{n,l}=B(x_{n,l},2^{-l})\big\backslash\bigcup_{1\leqslant k<n}B(x_{k,l},2^{-l}),\quad 1\leqslant  n \leqslant  N_l.$$
We have
$$B(1_G,2^l)\subset\bigcup_{1\leqslant  n \leqslant  N_l}B_{n,l}.$$

For every $1\leqslant  n \leqslant  N_l,$ choose a function $\phi_{n,l}\in C^{\infty}_c(G)$ compactly supported in the interior of  $B_{n,l}$, such that $0\leqslant \phi_{n,l}\leqslant 1$ and
\begin{equation}\label{eq-measure-partition-difference}
m(B_{n,l}\cap \{\phi_{n,l} \neq 1\})\leqslant \frac1{2^l\cdot N_l}.
\end{equation}
This is possible because {the} boundary of $B_{n,l}$ is made of finitely many pieces of smooth surfaces, thus $m(\partial B_{n,l})=0.$ Hence, for each $\varepsilon>0$, there exists a compact subset $K$ in the interior of $B_{n,l}$ with $m(B_{n,l}\setminus K) < \varepsilon.$ The conditions (i), (ii) and (iv) are automatically satisfied.

Let us show that the condition (iii) also holds. Fix $R>0$, there exists $l_R$, such that for all $l> l_R$, $2^l>R$, and 
$$B(1_G,R)\subset\bigcup_{1\leqslant  n \leqslant  N_l}B_{n,l}.$$
Denote
\begin{align*}
A_{R,l}& = B(1_G,R)\cap \big\{\sum_{1\leqslant  n \leqslant  N_l}\phi_{n,l}^2\neq 1\big\}\\
& \subseteq  \big(\bigcup_{1\leqslant  n \leqslant  N_l}B_{n,l}\big)\bigcap\big(\bigcup_{1\leqslant n \leqslant N_l}\{\phi_{n,l}\neq 1\}\big)= \bigcup_{1\leqslant  n \leqslant  N_l}\big(B_{n,l}\cap \{\phi_{n,l}\neq 1\}\big).
\end{align*}
Combining with \eqref{eq-measure-partition-difference}, we infer $m(A_{R,l})\leqslant 2^{-l}.$ By Borel-Cantelli lemma, we have
\begin{equation*}
m(\limsup_{l\to\infty}A_{R,l})=0
\end{equation*}
On the other hand, we also have
\begin{equation*}
B(1_G,R)\backslash \big\{\lim_{l\to \infty} \sum_{1\leqslant  n \leqslant  N_l}\phi_{n,l}^2=1\big\}\subset \limsup_{l\to\infty}A_{R,l}.
\end{equation*}
Since $R$ is arbitrary, this yields the condition (iii).
\end{proof}

In the following lemma, we build a sequence of functions $\{\alpha_l\}_{l\geqslant 0}$ to approximate the function $\alpha \in C_b(G)$ chosen in Theorem \ref{abstract asymptotics theorem}. This is a preparation for using Dominated Convergence Theorem in further proof.

\begin{lemma}\label{alphal lemma} Let $\{\phi_{n,l}\}_{l\geqslant 0, 1\leqslant  n \leqslant  N_l}\subset C^{\infty}_c(G)$ be as in Lemma \ref{function system lemma}. For every $l\geqslant 0$ and every $1\leqslant  n \leqslant  N_l$, pick a point $g_{n,l}\in{\rm supp}(\phi_{n,l})$.  Setting
\begin{equation}\label{eq-step-approximate-alpha-func}
\alpha_l=\sum_{n=1}^{N_l}\alpha(g_{n,l})\phi_{n,l}^2,\quad l\geqslant 0,	
\end{equation}
we have
\begin{enumerate}[{\rm (i)}]
\item $\alpha_l\to\alpha$ almost everywhere.
\item $\sup_{l\geqslant 0}\|\alpha_l(g)\|_{L_{\infty}(G)} \leqslant \Vert \alpha \Vert_{L_{\infty}(G)} $.
\end{enumerate}
\end{lemma}
\begin{proof} We claim that $\alpha_l(g)\to\alpha(g)$ for every $g\in G$ such that $\sum_{1\leqslant  n \leqslant  N_l}\phi_{n,l}^2(g)\to 1$ as $l\to\infty.$ Note that the latter condition is satisfied by an almost every $g\in G$ by Lemma \ref{function system lemma} (iii). Hence, the claim yields the first assertion of the lemma.

To see the claim, fix such $g\in G.$ Since $\{\phi_{n,l}\}_{n=1}^{N_l}$ are disjointly supported for any given $l\geqslant 0$, it follows that for every $l,$ there exists a unique $n_l\in[1,N_l]$ such that $\phi_{n_l,l}(g)\neq 0,$ while all other $\phi_{n,l}(g)$ are $0.$ Thus, 
$$\phi_{n_l,l}^2(g)=\sum_{1\leqslant  n \leqslant  N_l}\phi_{n,l}^2(g) \to 1,\quad l\to \infty.$$ 
It follows that $g\in{\rm supp}(\phi_{n_l,l}).$ By Lemma \ref{function system lemma} (iv), ${\rm dist}(g,g_{n_l,l})\leqslant 2^{-l}.$ Thus,
$$\alpha(g_{n,l})\to\alpha(g),\quad l\to\infty.$$
Note that
$$|\alpha_l(g)-\alpha(g)|=|\alpha(g_{n_l,l})\phi_{n_l,l}^2(g)-\alpha(g)|\leqslant\|\alpha\|_{\infty}|\phi_{n_l,l}^2(g)-1|+|\alpha(g_{n_l,l})-\alpha(g)|.$$
Passing $l\to\infty,$ we infer the claim.

Since $\{\phi_{n,l}\}_{1\leqslant n\leqslant N_l}$ are disjointly supported, $0\leqslant \phi_{n,l}\leqslant 1 $, and $\alpha\in C_b(G)$, by \eqref{eq-step-approximate-alpha-func}, for any $g\in G$,  we have 
$$\sup_{l\geqslant 0} |\alpha_l(g)|= \sup_{l\geqslant 0} |\alpha(g_{n_l,l})| \leqslant \Vert\alpha \Vert_{L_{\infty}} < \infty.$$
This yields the second assertion of the lemma.
\end{proof}

\begin{proof}[Proof of Theorem \ref{abstract asymptotics theorem}] Let $\{\phi_{n,l}\}_{l\geqslant 0, 1\leqslant  n \leqslant  N_l}\subset C^{\infty}_c(G)$ be as in Lemma \ref{function system lemma}. Let $g_{n,l}\in{\rm supp}(\phi_{n,l})$ for every $1\leqslant  n \leqslant  N_l$ and for every $l\geqslant0.$ Set
\begin{equation}\label{eq-A-l}
B_l=\sum_{n=0}^{N_l}M_{\phi_{n,l}}A_{g_{n,l}}M_{\phi_{n,l}},\quad l\geqslant 0.
\end{equation}
By assumption \ref{aata}, there exist limits
$$\lim_{t\to\infty}t^{\frac1p}\mu(t,M_{\phi_{n,l}^2}A_{g_{n,l}})=\alpha(g_{n,l})\|f\phi_{n,l}^2\|_{L_p(G)},\quad 1\leqslant  n \leqslant  N_l.$$
By assumption \ref{aatd},
$$M_{\phi_{n,l}}A_{g_{n,l}}M_{\phi_{n,l}}-M_{\phi_{n,l}^2}A_{g_{n,l}}=M_{\phi_{n,l}}\cdot [A_{g_{n,l}},M_{\phi_{n,l}}]\in(\mathcal{L}_{p,\infty})_0.$$
Hence, by \eqref{eq-singular-value-0-finite-rank-ideal}, for each summand in \eqref{eq-A-l}, there exists limit
$$\lim_{t\to\infty}t^{\frac1p}\mu(t,M_{\phi_{n,l}}A_{g_{n,l}}M_{\phi_{n,l}})=\alpha(g_{n,l})\|f\phi_{n,l}^2\|_{L_p(G)},\quad 1\leqslant  n \leqslant  N_l.$$
Since the operators
$$\{M_{f\phi_{n,l}}A_{g_{n,l}}M_{\phi_{n,l}}\}_{n=1}^{N_l}$$
are pairwise orthogonal, it follows from Lemma \ref{third disjoint lemma} that 
\begin{multline*}
\lim_{t\to\infty}t^{\frac1p}\mu(t,B_l)= \Big(\sum_{n=1}^{N_l} \alpha(g_{n,l})^p\|f\phi_{n,l}^2\|_{L_p(G)}^p\Big)^{\frac1p}\\
=\|f \cdot \big( \sum_{n=1}^{N_l} \alpha(g_{n,l}) \phi_{n,l}^2\big) \|_{L_p(G)}=\|f\alpha_l\|_{L_p(G)},
\end{multline*}
where $\alpha_l$ is defined in \eqref{eq-step-approximate-alpha-func}. The second  equality holds because $\{\phi_{n_l,l}\}_{1\leqslant n \leqslant N_l}$ are disjointly supported and $\alpha$ is positive (this is deduced from assumption (i) in Theorem \ref{abstract asymptotics theorem}). The last equality follows from the definition of $\alpha_l$ in \eqref{eq-step-approximate-alpha-func}.

Using Lemma \ref{abstract final lemma}, Lemma \ref{function system lemma} and the Dominated Convergence Theorem, we obtain
$${\rm dist}_{\mathcal{L}_{p,\infty}}\Big(A-B_l,(\mathcal{L}_{p,\infty})_0\Big)\to0,\quad l\to\infty.$$
By Lemma \ref{bs convergence lemma}, the following limits exist and are equal.
$$\lim_{t\to\infty}t^{\frac1p}\mu(t,A)=\lim_{l\to\infty}\|f\alpha_l\|_{L_p(G)}.$$
The assertion follows from the Dominated Convergence Theorem and Lemma \ref{alphal lemma}.
\end{proof}

\subsection{Proof of spectral asymptotic theorem: special case}\label{Section-special-case}

As always, $d_{\hom}$ denotes the homogeneous dimension of the group $G$ as in Section \ref{sec-intro}.

\begin{lemma}\label{lemma-bdd-prepare} 
    Let $P\geqslant 0$ be {an} uniformly Rockland order $m$ differential operator on $G.$ Let $g \in G$ be arbitrary. The following operators are bounded:
    \begin{align*}
        (1+P)^{-1} \colon & L_2(G)  \to W^m_2(G).\\
        (1+P_g^{\mathrm{top}})(1+P)^{-1} \colon & L_2(G) \to L_2(G).\\
        (1+P)^{-1} (1+P_g^{\mathrm{top}}) \colon & L_2(G) \to L_2(G).
    \end{align*}
    {
    Here, $(1+P)^{-1} (1+P_g^{\mathrm{top}})$ is \emph{a priori} defined on the Schwartz space $\Sc(G)$ and extended continuously to $L_2(G)$. Without ambiguity, we use the same notation to denote the extension operator.}
\end{lemma}
\begin{proof} 
{
By Theorem \ref{general elliptic estimate theorem}, the operators
$1+P$ and $1+P_g^{\mathrm{top}}$ both extend to topological linear isomorphisms
from $W^{m+s}_2(G)$ to $W^s_2(G)$ for any $s\in \mathbb{R}.$ It follows that $(1+P_g^{\mathrm{top}})(1+P)^{-1}$ extends to a bounded operator bounded on $L_2(G),$ and the other assertions are similar.
}
\end{proof}

\begin{lemma}\label{q first lemma} Let $P\geqslant0$ be {an} uniformly Rockland order $m$ differential operator on $G.$ If $f\in C^{\infty}_c(G),$ then $M_f(1+P)^{-1}\in\mathcal{L}_{\frac{d_{\hom}}{m},\infty}.$ Furthermore, there exits a constant $c_P$ such that
$${\rm dist}_{\mathcal{L}_{\frac{d_{\hom}}{m},\infty}}\Big(M_f(1+P)^{-1}, (\mathcal{L}_{\frac{d_{\hom}}{m}})_0\Big)
\leqslant c_{P,1} \cdot \Vert f \Vert_{L_{\frac{d_{\hom}}{m}}(G)}.$$
\end{lemma}
\begin{proof}  We write
$$M_f(1+P)^{-1}=M_f(1+P_0^{\mathrm{top}})^{-1}\cdot (P_0^{\mathrm{top}}+1)(1+P)^{-1}.$$
The first factor, $M_f(1+P_0^{\mathrm{top}})^{-1}$, belongs to $\mathcal{L}_{\frac{d_{\hom}}{m},\infty}$ by Lemma \ref{asterisque pre-verification lemma}.
The second factor, $(1+P_0^{\mathrm{top}})(1+P)^{-1}$, extends to a bounded operator on $L_2(G)$ due to Lemma \ref{lemma-bdd-prepare}.
Appealing to \eqref{eq-singular-value-multi} and \eqref{eq-quasi-norm-L-p-infity}, 
we infer that $M_f(1+P)^{-1}\in\mathcal{L}_{\frac{d_{\hom}}{m},\infty}$, and we have 
\begin{multline}\label{eq-MfP-lemma-1}
{\rm dist}_{\mathcal{L}_{\frac{d_{\hom}}{m},\infty}}\Big(M_f(1+P)^{-1}, (\mathcal{L}_{\frac{d_{\hom}}{m}})_0\Big)\\
\leqslant {\rm dist}_{\mathcal{L}_{\frac{d_{\hom}}{m},\infty}}\Big(M_f(1+P_0^{\mathrm{top}})^{-1}, (\mathcal{L}_{\frac{d_{\hom}}{m}})_0\Big)\cdot \Big\|(1+P_0^{\mathrm{top}})(1+P)^{-1}\Big\|_{\infty}.
\end{multline}
{
Note that if $T \in \mathcal{L}_{p,\infty},$ then
\[
    \mathrm{dist}_{\mathcal{L}_{p,\infty}}(T,(\mathcal{L}_{p,\infty})_0) = \limsup_{t\to\infty} t^{\frac1p}\mu(t,T).
\]
Thus it follows from Theorem \ref{main theorem constant} that}
\begin{equation}\label{eq-MfP-lemma-2}
{\rm dist}_{\mathcal{L}_{\frac{d_{\hom}}{m},\infty}}\Big(M_f(1+P_0^{\mathrm{top}})^{-1}, (\mathcal{L}_{\frac{d_{\hom}}{m}})_0\Big)=C_{m,G}\cdot \Vert f \Vert_{\frac{d_{\hom}}{m}}
\end{equation}

Combining \eqref{eq-MfP-lemma-1} and \eqref{eq-MfP-lemma-2}, we prove the assertion.
\end{proof}

\begin{lemma}\label{q second lemma} Let $P\geqslant0$ be uniformly Rockland order $m$ differential operator on $G.$ If $f\in C^{\infty}_c(G),$ then
$$[M_f,(1+P)^{-1}]\in\mathcal{L}_{\frac{d_{\hom}}{m+1},\infty}.$$
\end{lemma}
\begin{proof} The operator $[P,M_f]$ is a differential operator of order $m-1$ with smooth compactly supported coefficients. We write
$$[P,M_f]=\sum_{{\rm len}(\alpha)\leqslant m-1}M_{a_{\alpha}}X^{\alpha},$$
where $a_{\alpha}\in C^{\infty}_c(G)$ for every $\alpha$ with ${\rm len}(\alpha)\leqslant m-1.$ Thus,
\begin{align*}
&(1-\Delta)^{-\frac{m}{{ 2v}}}[P,M_f](1-\Delta)^{-\frac{m}{{ 2v}}}\\
&=\sum_{{\rm len}(\alpha)\leqslant m-1}(1-\Delta)^{-\frac{m}{{ 2v}}}M_{a_{\alpha}}(1-\Delta)^{-\frac{1}{{ 2v}}}\cdot (1-\Delta)^{\frac{1}{{ 2v}}}X^{\alpha}(1-\Delta)^{-\frac{m}{{ 2v}}}.
\end{align*}
By Lemma \ref{specific_cwikel_estimates_from_jfa_lemma}, we have
$$(1-\Delta)^{-\frac{m}{{ 2v}}}M_{a_{\alpha}}(1-\Delta)^{-\frac{1}{{ 2v}}}\in\mathcal{L}_{\frac{d_{\hom}}{m+1},\infty}.$$
Clearly,
$$(1-\Delta)^{\frac{1}{{ 2v}}}X^{\alpha}(1-\Delta)^{-\frac{m}{{ 2v}}}\in\mathcal{B}(L_2(G)).$$
Thus,
\[
(1-\Delta)^{-\frac{m}{{ 2v}}}[P,M_f](1-\Delta)^{-\frac{m}{{ 2v}}}\in\mathcal{L}_{\frac{d_{\hom}}{m+1},\infty}.
\]
We decompose the commutator as follows
\begin{align*}
    &[M_f,(1+P)^{-1}]=(1+P)^{-1}[P,M_f](1+P)^{-1}\\
                    &=(1+P)^{-1}(1-\Delta)^{\frac{m}{{ 2v}}}\cdot (1-\Delta)^{-\frac{m}{{ 2v}}}[P,M_f](1-\Delta)^{-\frac{m}{{ 2v}}} \cdot (1-\Delta)^{\frac{m}{{ 2v}}}(1+P)^{-1}.
\end{align*}
By Lemma \ref{lemma-bdd-prepare}, the first factor, $(1+P)^{-1}(1-\Delta)^{\frac{m}{{ 2v}}}$, and the last factor, $(1-\Delta)^{\frac{m}{{ 2v}}}(1+P)^{-1}$ , are bounded on $L_2(G)$. The assertion follows now from the preceding paragraph.
\end{proof}

\begin{lemma}\label{q third lemma} Let $P$ and $Q$ be order $m$ differential operators on $G.$ If $P$ is uniformly Rockland and positive, then $Q(1+P)^{-1}$ is bounded.
\end{lemma}
\begin{proof} By Lemma \ref{lemma-bdd-prepare}, $(1+P)^{-1}:L_2(G)\to W^m_2(G)$ is bounded. By Lemma \ref{differential_operators_are_bounded}, $Q:W^m_2(G)\to L_2(G)$ is bounded. The assertion follows.
\end{proof}

\begin{lemma}\label{uniform approximation lemma} Let $P\geqslant 0$ be uniformly Rockland order $m$ differential operator on { $G,$ and denote $\{a_{\alpha}\}$ for some choice of coefficients of $P,$ that is
\[
    P = \sum_{\len(\alpha)\leqslant m} M_{a_{\alpha}}X^{\alpha}.
\]

}There exists a constant $c_{P,2}$ such that, for every $f\in L_{\infty}(G)$ and for every $g\in G,$
$$\big\|M_f(P-P_g)(1+P^{\rm top}_g)^{-1}\big\|_{\infty}\leqslant c_{P,2}\sum_{\len(\alpha)\leqslant m}\|f\cdot (a_{\alpha}-a_{\alpha}(g))\|_{L_{\infty}(G)}.$$
\end{lemma}
\begin{proof} 
We write
\begin{multline*}
\|M_f(P-P_g)(1+P^{\rm top}_g)^{-1}\|_{\infty}\\
\leqslant \|M_f(P-P_g)\|_{W^m_2(G)\to L_2(G)}\|(1+P^{\rm top}_g)^{-1}\|_{L_2(G)\to W^m_2(G)}.
\end{multline*}

For the second factor on the right hand side, let us recall Lemma 4.33 in \cite{LMSZ25_elliptic} and Definition \ref{definition of ellipticity} (same as \cite[Definition 4.10]{LMSZ25_elliptic}) , according to which there exists a constant $c'_{P}$, uniform in $g\in G$, such that $\|(1+P^{\rm top}_g)^{-1}\|_{L_2(G)\to W^m_2(G)}\leqslant c'_P$.

Now we write
\begin{equation*}
M_f(P-P_g)=\sum_{\len(\alpha)\leqslant m}M_{f\cdot (a_{\alpha}-a_{\alpha}(g))}X^{\alpha},
\end{equation*}
where $X^{\alpha}$ is given in \eqref{eq-P-expression}.
Thus, we obtain
\begin{multline*}
\|M_f(P-P_g)\|_{W^m_2(G)\to L_2(G)}\\
\leqslant \sum_{\len(\alpha)\leqslant m}\|f\cdot (a_{\alpha}-a_{\alpha}(g))\|_{L_{\infty}(G)}\|X^{\alpha}\|_{W^m_2(G)\to L_2(G)}.
\end{multline*}

We refer to Lemma  \ref{differential_operators_are_bounded} and the claim above follows.
\end{proof}

By Lemma \ref{q first lemma}, both $M_f(1+P)^{-1}$ and $M_f(1+P^{\rm top}_g)^{-1}$ belong to $\mathcal{L}_{\frac{d_{\hom}}{m},\infty}$. Now we calculate the distance between them.

\begin{lemma}\label{s7 p minus pg locally} Let $P\geqslant 0$ be uniformly Rockland order $m$ differential operator on $G.$ There exists a constant $c_{P,3}$ such that, for every $f\in C^{\infty}_c(G)$ and for every $g\in G,$ 
\begin{multline*}
{\rm dist}_{\mathcal{L}_{\frac{d_{\hom}}{m},\infty}}
\Big(M_f(1+P)^{-1}-M_f(1+P^{\rm top}_g)^{-1},(\mathcal{L}_{\frac{d_{\hom}}{m},\infty})_0\Big)\\
\leqslant c_{P,3}\|f\|_{L_{\frac{d_{\hom}}{m}}(G)}\cdot\sum_{\len(\alpha)= m}\|a_{\alpha}-a_{\alpha}(g)\|_{L_{\infty}({\rm supp}(f))}.
\end{multline*}
\end{lemma}
\begin{proof}
We have
\begin{align}\label{eq-MfP-MfP_top}
& M_f(1+P)^{-1}-M_f(1+P^{\rm top}_g)^{-1} \\
& \quad = M_f(1+P)^{-1}(P^{\rm top}_g-P)(1+P^{\rm top}_g)^{-1}\nonumber \\
& \quad = [M_f, (1+P)^{-1}](P^{\rm top}_g-P)(1+P^{\rm top}_g)^{-1}
+ (1+P)^{-1} M_f (P^{\rm top}_g-P)(1+P^{\rm top}_g)^{-1}\nonumber\\
& \quad =  [M_f, (1+P)^{-1}](P^{\rm top}_g-P)(1+P^{\rm top}_g)^{-1}\nonumber \\
& \qquad + (1+P)^{-1} M_f (P^{\rm top}_g-P_g)(1+P^{\rm top}_g)^{-1}
+ (1+P)^{-1} M_f (P_g-P)(1+P^{\rm top}_g)^{-1} \nonumber
\end{align}

For first summand, 	
by Lemma \ref{q second lemma}, $[M_f,(1+P)^{-1}]\in\mathcal{L}_{\frac{d_{\hom}}{m+1},\infty}.$
By Lemma \ref{top dag commute} and taking into account that $P=P^{\dagger}$, we have $(P^{\rm top}_g)^{\dagger}=(P^{\dagger})^{\rm top}_g = P^{\rm top}_g $. 
Thus, we can apply Lemma \ref{q third lemma} and obtain 
\begin{equation}\label{eq-P_top-P_bounded}
(P^{\rm top}_g-P)(1+P^{\rm top}_g)^{-1}\in\mathcal{L}_{\infty}.
\end{equation}
Combining with Remark \ref{remark-inclusion-Lpq}, \eqref{eq-singular-value-multi}, and \eqref{eq-P_top-P_bounded}, we have 
\begin{equation}\label{eq-Mf-1}
[M_f,(1+P)^{-1}]\cdot (P^{\rm top}_g-P)(1+P^{\rm top}_g)^{-1}\in (\mathcal{L}_{\frac{d_{\hom}}{m},\infty})_0.
\end{equation}

For the second summand in \eqref{eq-MfP-MfP_top}, we have
\begin{multline}\label{eq-Mf-2}
(1+P)^{-1} M_f (P^{\rm top}_g-P_g)(1+P^{\rm top}_g)^{-1} \\
= (1+P)^{-1} (1-\Delta)^{\frac{m}{{ 2v}}} \cdot (1-\Delta)^{-\frac{m}{{ 2v}}} M_f(1-\Delta)^{-\frac{1}{{ 2v}}} 
\cdot (1-\Delta)^{\frac{1}{{ 2v}}}(P^{\rm top}_g-P_g)(1+P^{\rm top}_g)^{-1} 
\end{multline}
As stated in the proof of Lemma \ref{q second lemma}, $(1+P)^{-1}(1-\Delta)^{\frac{m}{{ 2v}}}$ is bounded on $L_2(G)$. By the Cwikel type estimate from \cite[Theorem 5.1, (iii)]{MSZ-stratified-23}, we have $(1-\Delta)^{-\frac{m}{2}} M_f(1-\Delta)^{-\frac{1}{{ 2v}}} \in\mathcal{L}_{\frac{d_{\hom}}{m+1},\infty}$. Thus, we obtain 
\begin{equation}\label{eq-Mf-3}
(1+P)^{-1}M_f(P^{\rm top}_g-P_g)(1+P^{\rm top}_g)^{-1} \in\mathcal{L}_{\frac{d_{\hom}}{m+1},\infty} \subset \big(\mathcal{L}_{\frac{d_{\hom}}{m},\infty}\big)_0
\end{equation}

Therefore, combining \eqref{eq-MfP-MfP_top}, \eqref{eq-Mf-1} and \eqref{eq-Mf-3}, we arrive at
\begin{align*}
& {\rm dist}_{\mathcal{L}_{\frac{d_{\hom}}{m},\infty}}\Big(M_f(1+P)^{-1}-M_f(1+P^{\rm top}_g)^{-1},(\mathcal{L}_{\frac{d_{\hom}}{m},\infty})_0\Big)\\
& \quad \leqslant {\rm dist}_{\mathcal{L}_{\frac{d_{\hom}}{m},\infty}}\Big((1+P)^{-1}M_f\cdot M_{\chi_{{\rm supp}(f)}}(P_g-P)(1+P^{\rm top}_g)^{-1},(\mathcal{L}_{\frac{d_{\hom}}{m},\infty})_0\Big)\\
& \quad \leqslant {\rm dist}_{\mathcal{L}_{\frac{d_{\hom}}{m},\infty}}\Big((1+P)^{-1}M_f, (\mathcal{L}_{\frac{d_{\hom}}{m},\infty})_0\Big)\cdot\Big\|M_{\chi_{{\rm supp}(f)}}(P_g-P)(1+P^{\rm top}_g)^{-1}\Big\|_{\infty}.
\end{align*}

The assertion follows now from Lemmas \ref{q first lemma} and \ref{uniform approximation lemma}.
\end{proof}

\begin{lemma}\label{duhamel lemma} The mapping $g\mapsto e^{-P_g^{\rm top}}$ is a bounded and continuous $L_p(\mathrm{VN}(G),\tau)$-valued function for every $p>\max\{1,\frac{d_{{\rm hom}}}{m}\}.$
\end{lemma}
\begin{proof}
By Duhamel's formula, for $g,h\in G,$ we have
$$e^{-P_g^{\rm top}}-e^{-P_h^{\rm top}} = \int_0^1 e^{-(1-\theta)P_g^{\rm top}}(P_h^{\rm top}-P_g^{\rm top})e^{-\theta P_h^{\rm top}}\,d\theta.$$
Thus,
$$\|e^{-P_g^{\rm top}}-e^{-P_h^{\rm top}}\|_{L_p(\mathrm{VN}(G),\tau)}\leqslant\sup_{\theta\in[0,1]}\|e^{-(1-\theta)P_g^{\rm top}}(P_h^{\rm top}-P_g^{\rm top})e^{-\theta P_h^{\rm top}}\|_{L_p(\mathrm{VN}(G),\tau)}.$$

If $\theta\in[0,\frac12],$ then
\begin{multline*}
    \|e^{-(1-\theta)P_g^{\rm top}}(P_h^{\rm top}-P_g^{\rm top})e^{-\theta P_h^{\rm top}}\|_{L_p(\mathrm{VN}(G),\tau)}\leqslant\|e^{-\frac12P_g^{\rm top}}(P_h^{\rm top}-P_g^{\rm top})\|_{L_p(\mathrm{VN}(G),\tau)} \\
    \leqslant \|e^{-\frac14P_g^{\rm top}}\|_{L_p(\mathrm{VN}(G),\tau)}\|e^{-\frac14P_g^{\rm top}}(P_h^{\rm top}-P_g^{\rm top})\|_{L_{\infty}(\mathrm{VN}(G),\tau)}.
\end{multline*}
Clearly,
$$e^{-\frac14P_g^{\rm top}}=e^{-\frac14P_g^{\rm top}}(1+P_g^{{\rm top}})\cdot (1+P_g^{{\rm top}})^{-1}(1-\Delta)^{\frac{m}{2}}\cdot (1-\Delta)^{-\frac{m}{2}}.$$
Thus,
\begin{align*}
    &\|e^{-\frac14P_g^{\rm top}}\|_{L_p(\mathrm{VN}(G),\tau)}\\
    & \qquad  \leqslant \|e^{-\frac14P_g^{\rm top}}(1+P_g^{{\rm top}})\|_{\infty}\|(1+P_g^{{\rm top}})^{-1}(1-\Delta)^{\frac{m}{2}}\|_{\infty}\|(1-\Delta)^{-\frac{m}{2}}\|_{L_p(\mathrm{VN}(G),\tau)},\\
    & \|e^{-\frac14P_g^{\rm top}}(P_h^{\rm top}-P_g^{\rm top})\|_{L_{\infty}(\mathrm{VN}(G),\tau)}\\
    & \qquad \leqslant \|e^{-\frac14P_g^{\rm top}}(1+P_g^{{\rm top}})\|_{\infty}\|(1+P_g^{{\rm top}})^{-1}(1-\Delta)^{\frac{m}{2}}\|_{\infty}\|(1-\Delta)^{-\frac{m}{2}}(P_h^{\rm top}-P_g^{\rm top})\|_{L_{\infty}(\mathrm{VN}(G),\tau)}.
\end{align*}
Consequently,
\begin{multline*}
    \sup_{\theta\in[0,\frac12]}\|e^{-(1-\theta)P_g^{\rm top}}(P_h^{\rm top}-P_g^{\rm top})e^{-\theta P_h^{\rm top}}\|_{L_p(\mathrm{VN}(G),\tau)}\\
    \leqslant c_P\|(-\Delta)^{-\frac{m}{2}}(P_h^{\rm top}-P_g^{\rm top})\|_{L_{\infty}(\mathrm{VN}(G),\tau)}.
\end{multline*}

Similarly,
\begin{multline*}
    \sup_{\theta\in[\frac12,1]}\|e^{-(1-\theta)P_g^{\rm top}}(P_h^{\rm top}-P_g^{\rm top})e^{-\theta P_h^{\rm top}}\|_{L_p(\mathrm{VN}(G),\tau)}\\
    \leqslant c_P\|(P_h^{\rm top}-P_g^{\rm top})(-\Delta)^{-\frac{m}{2}}\|_{L_{\infty}(\mathrm{VN}(G),\tau)}.
\end{multline*}
This completes the proof.
\end{proof}

\begin{lemma}\label{continuity of the weight} The mapping
$$g\to \tau(e^{-P^{\rm top}_g}),\quad g\in G,$$
is bounded and continuous.
\end{lemma}
\begin{proof} Let $p\in\mathbb{N}$ be such that $p>\max\{1,\frac{d_{{\rm hom}}}{m}\}.$ Applying Lemma \ref{duhamel lemma} to the operator $p^{-1}P,$ we obtain the continuity of the mapping $g\mapsto e^{-p^{-1}P_g^{\rm top}}$ in $L_p(\mathrm{VN}(G),\tau).$ By H\"older inequality, the mapping $g\mapsto e^{-p^{-1}P_g^{\rm top}}$ is a bounded and continuous in $L_1(\mathrm{VN}(G),\tau).$ This completes the proof.
\end{proof}

\begin{proof}[Proof of Theorem \ref{main theorem special case}] Let $p=\frac{d_{{\rm hom}}}{m}$ and set
$$T=M_f(1+P)^{-1},\quad T_g=M_f(1+P_g^{\rm top})^{-1},\quad g\in G,$$	
$$a(g)=(a_{\alpha})_{\len(\alpha)\leqslant m},\quad  \alpha(g)=\Big(\frac1{\Gamma(\frac{d_{\hom}}{2m}+1)}\tau(e^{-P_g^{{\rm top}}})\Big)^{\frac{m}{d_{{\rm hom}}}},\quad g\in G.$$

That $T,T_g\in\mathcal{L}_{p,\infty}$ follows from Lemma \ref{q first lemma}. It is immediate that $a\in C_b(G,N_T)$ where $N_T=\sum_{\len(\alpha)\leqslant m}1.$ Clearly, $a\in C^{\infty}_b(G,N_T).$ By Proposition 5.4 in \cite{Folland1975}, $a$ is Lipschitz (with respect to a translation invariant homogeneous metric on $G$) and, hence, uniformly continuous. By definition, $f\in C^{\infty}_c(G)\subset L_p(G).$ That $\alpha\in C_b(G)$ follows from Lemma \ref{continuity of the weight}.

The assumption \ref{aata} in Theorem \ref{abstract asymptotics theorem} follows from Theorem \ref{main theorem constant}. The assumption \ref{aatb} in Theorem \ref{abstract asymptotics theorem} follows from Lemma \ref{s7 p minus pg locally}. The assumption \ref{aatc} in Theorem \ref{abstract asymptotics theorem} follows from Lemma \ref{q first lemma}. The assumption \ref{aatd} in Theorem \ref{abstract asymptotics theorem} follows from Lemma \ref{q second lemma}. 

The assertion follows now by applying Theorem \ref{abstract asymptotics theorem}.
\end{proof}

\subsection{A commutator lemma}
The following lemma is essentially a commutator version of \eqref{general_raising_powers_principle}.
\begin{lemma}\label{hsz lemma} Let $A,B\in\mathcal{L}_{\infty},$ $p>0,$ $\theta\in(0,1),$ and assume that $B\geqslant 0.$ Then
\begin{enumerate}
\item{} $[A,B] \in \mathcal{L}_{p,\infty}\Longrightarrow [B^{\theta},A] \in \mathcal{L}_{\frac{p}{\theta},\infty}.$
\item{} $[A,B] \in (\mathcal{L}_{p,\infty})_0 \Longrightarrow [B^\theta,A] \in (\mathcal{L}_{\frac{p}{\theta},\infty})_0.$
\end{enumerate}
If we also assume that $B \in \mathcal{L}_{p,\infty},$ then we get the same implications for $\theta>1.$
\end{lemma}
\begin{proof}
The assertion for $0<\theta<1$ is essentially a very special case of \cite[Corollary 7.1]{HSZ2019}. For $\theta> 1,$ let $n = \lfloor \theta\rfloor$ and $\alpha = \theta-n$ and write
\[
[A,B^{\theta}] = B^{n}[A,B^{\alpha}]+\sum_{k=0}^{n-1} B^{n-k-1}[A,B]B^{k+\alpha}.
\]
Then we apply the $0<\theta<1$ case to $[A^{\alpha},B]$ and deduce the result for $\theta=n+\alpha$ by H\"older's inequality.
\end{proof}

\begin{lemma}\label{alt_lemma}
For all $r>1$ and $0<p\leqslant \infty,$ if $A$ and $B$ are bounded operators, then
$$\|AB\|_{rp,\infty} \leqslant e^{\frac{1}{rp}}\||A|^r|B|^r\|_{p,\infty}^{\frac1r}.$$
\end{lemma}
\begin{proof}
The Araki-Lieb-Thirring inequality \cite[Theorem 2]{Kosaki-alt-1992}, asserts that if $r>1,$ then for every $k\geqslant 0,$ we have
$$\prod_{j=0}^k\mu(j,|AB|^r) \leqslant \prod_{j=0}^k\mu(j,|A|^r|B|^r).$$
Taking $k=0,$ we immediately obtain the $p=\infty$ case of the lemma. Now assume that $p<\infty.$ It follows from the Araki-Lieb-Thirring inequality that for every $k\geqslant 0$ we have
$$\mu^{k+1}(k,|AB|^r) \leqslant \prod_{j=0}^k\frac{\||A|^r|B|^r\|_{p,\infty}}{(j+1)^{\frac1p}}=((k+1)!)^{-\frac1p}\||A|^r|B|^r\|_{p,\infty}^{k+1}.$$
Taking the $(k+1)$-st root, we obtain
$$\mu(k,|AB|^r) \leqslant ((k+1)!)^{-\frac{1}{p(k+1)}}\||A|^r|B|^r\|_{p,\infty}.$$
Using the numerical inequality\footnote{This can be verified by examining a Riemann sum for $\int_1^{k+1} \log(x)\,dx.$}
$$\frac{(k+1)^{k+1}}{(k+1)!}\leqslant e^{k+1},\quad k\geqslant 0,$$
we obtain
$$((k+1)!)^{-\frac{1}{p(k+1)}}\leqslant (\frac{e}{k+1})^{\frac1p},\quad k\geqslant 0.$$
Therefore, for every $k\geqslant 0,$
$$(k+1)^{\frac1p}\mu(k,|AB|^r) \leqslant e^{\frac1p}\||A|^r|B|^r\|_{p,\infty}.$$
Taking the supremum over $k\geqslant 0,$ we deduce the result from the definition of the quasi-norm in $\Lc_{pr,\infty}.$
\end{proof} 
Note that Lemma \ref{alt_lemma} implies that if $A$ and $B$ are positive bounded operators with $AB\in \Lc_{p,\infty},$ then for all $0<\theta<1,$ 
$A^{\theta}B^{\theta} \in \Lc_{\frac{p}{\theta},\infty}.$

\begin{lemma}\label{square_root_lemma}
Let $p>0$ and let $0\leqslant A,B \in \Bc(H).$ If $A^{\frac12}B \in \Lc_{p,\infty}$ and if $[B,A^{\frac14}] \in (\Lc_{p,\infty})_0,$ then 
$$A^{\frac12}BA^{\frac12}-(A^{\frac14}B^{\frac12}A^{\frac14})^2 \in (\Lc_{p,\infty})_0.$$
\end{lemma}
\begin{proof} Using Lemma \ref{alt_lemma} with $r=2$ and the assumption $A^{\frac12}B\in \Lc_{p,\infty},$ we deduce $A^{\frac14}B^{\frac12} \in \Lc_{2p,\infty}.$

Using Leibniz rule and taking into account that $A$ is bounded and that $[B,A^{\frac14}]\in (\Lc_{p,\infty})_0,$ we infer that $[B,A^{\frac12}] \in (\Lc_{p,\infty})_0.$ Applying Lemma \ref{hsz lemma} with $\theta = \frac12,$ we deduce
$$[B^{\frac12},A^{\frac12}] \in (\Lc_{2p,\infty})_0.$$
Therefore
\begin{align*}
A^{\frac12}BA^{\frac12}-(A^{\frac14}B^{\frac12}A^{\frac14})^2 &= A^{\frac12}BA^{\frac12}-A^{\frac14}B^{\frac12}A^{\frac12}B^{\frac12}A^{\frac14}\\
&= A^{\frac12}[B,A^{\frac14}]A^{\frac14}-A^{\frac14}[B^{\frac12},A^{\frac12}]B^{\frac12}A^{\frac14}\\
&\in (\Lc_{p,\infty})_0+(\Lc_{2p,\infty})_0\cdot \Lc_{2p,\infty}.
\end{align*}
H\"older's inequality implies that
$$A^{\frac12}BA^{\frac12}-(A^{\frac14}B^{\frac12}A^{\frac14})^2 \in (\Lc_{p,\infty})_0.$$
\end{proof}

\begin{corollary}\label{dyadic_roots_corollary}
Let $A$ and $B$ be positive bounded linear operators, and let $n\geqslant 0.$ If $A^{\frac12}B \in \Lc_{p,\infty}$ and if $[B,A^{2^{-n-1}}] \in (\Lc_{p,\infty})_0,$ then 
$$A^{\frac12}BA^{\frac12}-(A^{2^{-n-1}}B^{2^{-n}}A^{2^{-n-1}})^{2^n} \in (\Lc_{p,\infty})_0.$$
\end{corollary}
\begin{proof} For $0\leqslant k\leqslant n-1,$ set $A_k=A^{2^{-k}},$ $B_k=B^{2^{-k}}$ and $p_k=2^kp.$ Using Lemma \ref{alt_lemma} with $r=2^k,$ we obtain
$$A_k^{\frac12}B_k\in \Lc_{p_k,\infty}.$$
Using Lemma \ref{hsz lemma} with $\theta=2^{-k}$ and the assumption $[B,A^{2^{-n-1}}] \in (\Lc_{p,\infty})_0,$ we obtain
$$[B_k,A^{2^{-n-1}}] \in (\Lc_{p_k,\infty})_0.$$
It follows now from the Leibniz rule that
$$[B_k,A_k^{\frac14}] \in (\Lc_{p_k,\infty})_0.$$
In other words, the assumptions in Lemma \ref{square_root_lemma} are met for $A_k,$ $B_k$ and $p_k.$

Hence, by Lemma \ref{square_root_lemma}, we have
$$A_k^{\frac12}B_kA_k^{\frac12}-(A_k^{\frac14}B_k^{\frac12}A_k^{\frac14})^2 \in (\Lc_{p_k,\infty})_0.$$
Thus,
$$(A_k^{\frac12}B_kA_k^{\frac12})^{2^k}-(A_k^{\frac14}B_k^{\frac12}A_k^{\frac14})^{2^{k+1}}\in (\Lc_{p,\infty})_0.$$
Consequently,
$$A^{\frac12}BA^{\frac12}-(A^{2^{-n-1}}B^{2^{-n}}A^{2^{-n-1}})^{2^n}=\sum_{k=0}^{n-1}(A_k^{\frac12}B_kA_k^{\frac12})^{2^k}-(A_k^{\frac14}B_k^{\frac12}A_k^{\frac14})^{2^{k+1}}\in (\Lc_{p,\infty})_0.$$
\end{proof}

\begin{lemma}\label{PS_Lipschitz_lemma} For every $1<p<\infty,$ there exists {a universal constant $C$} such that for every bounded operator $X,$ for every bounded self-adjoint operator $Y$ with $[X,Y] \in\Lc_{p,\infty}$ and for every bounded $C^1$ function $f$ on $\R$ we have
$$\|[X,f(Y)]\|_{p,\infty} \leqslant C\|f'\|_{\infty}\|[X,Y]\|_{p,\infty}.$$
{Moreover, if} $[X,Y] \in (\Lc_{p,\infty})_0,$ then $[X,f(Y)] \in (\Lc_{p,\infty})_0.$
\end{lemma}
\begin{proof} Since $1<p<\infty,$ the inequality follows from \cite{PS-acta}.

To see the {``moreover" part,} we may assume without loss of generality that $\|Y\|_{\infty}\leqslant 1.$ Fix a sequence of polynomials $\{f_n\}_{n\geqslant0}$ such that $f_n\to f$ in $C^1[-1,1].$ We have
$$\|[X,f_n(Y)]-[X,f(Y)]\|_{p,\infty} \leqslant C\|f-f_n\|_{C^1[-1,1]}\|[X,Y]\|_{p,\infty}.$$
It follows from the Leibniz rule that $[X,f_n(Y)]\in (\Lc_{p,\infty})_0.$ Since $(\Lc_{p,\infty})_0$ is norm-closed in $\Lc_{p,\infty},$ the "moreover" part follows.
\end{proof}

Using an integral formula from \cite[{Theorem 5.2.1}]{SZ-Asterisque}, we have the following means of "inserting a power $q>1$" into a product 
$A^{\frac12}BA^{\frac12}.$

\begin{theorem}\label{asterisque_integral_formula}
Let $1<q<r,$ and let $A$ and $B$ be positive bounded operators obeying the following conditions:
\begin{enumerate}[{\rm (i)}]
\item{}\label{aif0} $B^{q-1}A^{q-1} \in \Lc_{\frac{r}{q-1},\infty},$
\item{}\label{aif1} $A^{\frac12}BA^{\frac12} \in \Lc_{r,\infty},$
\item{}\label{aif2} $[B,A^{\frac12}] \in (\Lc_{r,\infty})_0,$
\item{}\label{aif4} $[B^{q},A^{\frac{q}{2}}]\in (\Lc_{\frac{r}{q},\infty})_0.$
\item{}\label{aif3} $B^{q-1}[B,A^{q-1}]A \in (\Lc_{\frac{r}{q},\infty})_0.$
\end{enumerate}
Under those assumptions,
$$(A^{\frac12}BA^{\frac12})^q-A^{\frac{q}{2}}B^{q}A^{\frac{q}{2}} \in (\Lc_{\frac{r}{q},\infty})_0.$$
\end{theorem}
\begin{proof} Since $\frac{r}{q}>1,$ the ideal $(\Lc_{\frac{r}{q},\infty})_0$ is a separable Banach space. Hence, by the Pettis measurability theorem, in order to prove that the integral $\int_{-\infty}^{\infty} f(s)\,ds$ of a continuous $(\Lc_{\frac{r}{q},\infty})_0$-valued function converges as a Bochner integral in $(\Lc_{\frac{r}{q},\infty})_0,$ it suffices to prove that
$$\int_{-\infty}^\infty \|f(s)\|_{\frac{r}{q},\infty}\,ds < \infty.$$

We use the following integral formula, proved in \cite[Theorem 5.2.1]{SZ-Asterisque}:
$$B^qA^q-(A^{\frac12}BA^{\frac12})^q = T_q(0)-\int_{-\infty}^{\infty} \widehat{g}_q(s)T_q(s)\,ds,$$
where
$$T_q(s) = B^{q-1}[BA^{\frac12},A^{q-\frac12+is}](A^{\frac12}BA^{\frac12})^{-is}+B^{is}[BA^{\frac12},A^{\frac12+is}](A^{\frac12}BA^{\frac12})^{q-1-is}$$
and where $g_q$ is a Schwartz class function whose precise formula is not relevant here.

We show that $T_q(s)\in (\Lc_{\frac{r}{q},\infty})_0$ for all $s \in \mathbb{R},$ with norm satisfying
$$\|T_q(s)\|_{\frac{r}{q},\infty} \lesssim (1+|s|).$$
Since $\widehat{g}_q$ is Schwartz class, this suffices to conclude that $B^qA^q-(A^{\frac12}BA^{\frac12})^q \in (\Lc_{\frac{r}{q},\infty})_0.$ This implies the stated result, due to \ref{aif4}.

Applying the Leibniz rule to the first summand in the definition of $T_q(s)$ gives
\begin{align*}
T_q(s) &= B^{q-1}[B,A^{q-1}]A\cdot A^{is}(A^{\frac12}BA^{\frac12})^{-is}\\ 
&\quad +B^{q-1}A^{q-1}\cdot [B,A^{\frac12+is}]\cdot A^{\frac12}(A^{\frac12}BA^{\frac12})^{-is}\\   
&\quad +B^{is}\cdot [B,A^{\frac12+is}]\cdot A^{\frac12}\cdot (A^{\frac12}BA^{\frac12})^{q-1-is}.
\end{align*}
By Lemma \ref{PS_Lipschitz_lemma}, we have $[B,A^{\frac12+is}] \in (\Lc_{r,\infty})_0.$ Hence, by the H\"older inequality and the assumptions it follows that $T_q(s) \in (\Lc_{\frac{r}{q},\infty})_0.$

Using the triangle and H\"older inequalities, we write
\begin{align*}
\|T_q(s)\|_{\frac{r}{q},\infty}&\leqslant c_{q,r}\Big(\|B^{q-1}[B,A^{q-1}]A\|_{\frac{r}{q},\infty}\\
&\quad +\|B^{q-1}A^{q-1}\|_{\frac{r}{q-1},\infty}\|[B,A^{\frac12+is}]\|_{r,\infty}\|A\|_{\infty}^{\frac12}\\
&\quad +\|[B,A^{\frac12+is}]\|_{r,\infty}\|A\|_{\infty}^{\frac12}\|A^{\frac12}BA^{\frac12}\|_{r,\infty}^{q-1}\Big).
\end{align*}
Here we have used the fact that $B^{is},$ $A^{is}$ and $(A^{\frac12}BA^{\frac12})^{is}$ are contractions. By Lemma \ref{PS_Lipschitz_lemma}, we have
$$\|T_q(s)\|_{\frac{q}{r},\infty} \lesssim (1+|s|).$$
This concludes the proof.
\end{proof}

\begin{theorem}\label{ultimate_asterisque_esque_theorem}
Let $A$ and $B$ be positive bounded operators. If for all $\alpha,\beta>0$ we have
\begin{enumerate}[{\rm (i)}]
\item{} $A^{\alpha}B^{\beta} \in \Lc_{\frac{p}{\beta},\infty},$
\item{} $B^{\alpha}[B^{\beta},A^{\delta}]B^{\gamma} \in (\Lc_{\frac{p}{\alpha+\beta+\gamma},\infty})_0$
\end{enumerate}
then for all $q>0$ we have
$$(A^{\frac12}BA^{\frac12})^q - A^{\frac{q}{2}}B^qA^{\frac{q}{2}} \in (\Lc_{\frac{p}{q},\infty})_0.$$
\end{theorem}
\begin{proof} For every $l\geqslant0,$ set $A_l=A^{2^{-l}}$ and $B_l=B^{2^{-l}}.$ Let $j,k\geqslant 0$ be large enough so that $2^{j+k}p>2^kq>1.$ It is easy to check that the assumptions in Theorem \ref{asterisque_integral_formula} are met for $A_{j+k}$ and $B_{j+k}.$ Applying Theorem \ref{asterisque_integral_formula} applied to $A_{j+k}$ and $B_{j+k},$ we write
$$(A_{j+k}^{\frac12}B_{j+k}A_{j+k}^{\frac12})^{2^kq}-A_{j+k}^{2^{k-1}q}B_{j+k}^{2^kq}A_{j+k}^{2^{k-1}q} \in (\Lc_{\frac{2^{j+k}p}{2^kq},\infty})_0=(\Lc_{\frac{2^jp}{q},\infty})_0.$$
By Corollary \ref{dyadic_roots_corollary}
$$(A_{j+k}^{\frac12}B_{j+k}A_{j+k}^{\frac12})^{2^k}-A_j^{\frac12}B_jA_j^{\frac12}\in (\Lc_{2^jp,\infty})_0.$$
By \eqref{general_raising_powers_principle},
$$(A_{j+k}^{\frac12}B_{j+k}A_{j+k}^{\frac12})^{2^kq}-(A_j^{\frac12}B_jA_j^{\frac12})^q\in (\Lc_{\frac{2^jp}{q},\infty})_0.$$
Thus,
$$(A_j^{\frac12}B_jA_j^{\frac12})^q-A_j^{\frac{q}{2}}B_j^qA_j^{\frac{q}{2}} \in (\Lc_{\frac{2^jp}{q},\infty})_0.$$
Again applying \eqref{general_raising_powers_principle} with $\gamma=2^j,$
$$(A_j^{\frac12}B_jA_j^{\frac12})^{2^jq}-\big(A_j^{\frac{q}{2}}B_j^qA_j^{\frac{q}{2}}\big)^{2^j}\in (\Lc_{\frac{p}{q},\infty})_0.$$
By Corollary \ref{dyadic_roots_corollary}
$$(A_j^{\frac12}B_jA_j^{\frac12})^{2^j}-A^{\frac12}BA^{\frac12}\in (\Lc_{p,\infty})_0.$$
By \eqref{general_raising_powers_principle},
$$(A_j^{\frac12}B_jA_j^{\frac12})^{2^jq}-(A^{\frac12}BA^{\frac12})^q\in (\Lc_{\frac{p}{q},\infty})_0.$$
Thus,
$$(A^{\frac12}BA^{\frac12})^q-\big(A_j^{\frac{q}{2}}B_j^qA_j^{\frac{q}{2}}\big)^{2^j}\in (\Lc_{\frac{p}{q},\infty})_0.$$
The assertion follows now from Corollary \ref{dyadic_roots_corollary} and \eqref{general_raising_powers_principle}.
\end{proof}

\subsection{Proof of spectral asymptotic theorem: general case}\label{Section-general-case}

For this section, we need a quantitative form of Lemma \ref{specific_cwikel_estimates_from_jfa_lemma}.

The estimates for the weak Schatten $\mathcal{L}_{\frac{d_{\hom}}{\gamma},\infty}$ quasi-norm are given in terms of spaces that we denote $\ell_{p}(L_q)(G).$ These were introduced in \cite[Definition 6.3]{MSZ-stratified-23}, based on a construction of Birman and Solomyak \cite{BirmanSolomyak1969}. The spaces $\ell_{p}(L_q)(G)$ are special cases of the decomposition spaces defined by Feichtinger and Gr\"{o}bner \cite{FeichtingerGrobner1985}.

\begin{definition}\label{decomposition_spaces_definition}
    Let $\Gamma\subset G$ be a discrete subset, and let $U\subset G$ be a bounded open set. We say that $(\Gamma,U)$ is a \emph{covering of bounded multiplcity} if the set of translates $\{\gamma U\}_{\gamma\in \Gamma}$ is a cover of $G$ such that there exists $N>0$ such that every $g\in G$ is contained in at most $N$ of the sets $\{\gamma U\}_{\gamma\in \Gamma}.$ For $p,q\in (0,\infty],$ the space $\ell_p(L_2(G))$ is defined as the set of all locally integrable functions $f$ on $G$ such that
    \[
        \|f\|_{\ell_p(L_q)(G)} := \|\{\|f\chi_{\gamma U}\|_{L_p(\gamma U)}\}_{\gamma \in \Gamma}\|_{\ell_q(\Gamma)} < \infty.
    \]
\end{definition}
It was proved in Section 6 of \cite{MSZ-stratified-23} that coverings of bounded multiplicity always exist and that the space $\ell_p(L_q)(G)$ is independent of the choice of covering, up to equivalence of quasi-norms.

The following lemma is very similar to Theorem 5.1 of \cite{MSZ-stratified-23}, and so we defer the proof to Appendix \ref{cwikel_appendix}.
\begin{lemma}\label{quantitative_cwikel_estimate}
Let $f$ be a locally integrable function on $G.$ Let $\gamma>0.$
\begin{enumerate}[{\rm (i)}]
\item{} If $2<\frac{d_{\hom}}{\gamma}<\infty,$ then
\[
\|M_f(1-\Delta)^{-\frac{\gamma}{2v}}\|_{\frac{d_{\hom}}{\gamma},\infty} \leqslant C_{\gamma,G}\|f\|_{L_{\frac{d_{\hom}}{\gamma}}(G)}.
\]
\item{} If $0 < \frac{d_{\hom}}{\gamma}\leqslant 2$ and $q>2,$ then
\[
\|M_f(1-\Delta)^{-\frac{\gamma}{2v}}\|_{\frac{d_{\hom}}{\gamma},\infty} \leqslant C_{\gamma,G}\|f\|_{\ell_{\frac{d_{\hom}}{\gamma}}(L_q)(G)}.
\]
\end{enumerate}
\end{lemma}

\begin{lemma}\label{improved_cwikel_estimate}
Let $P\geqslant 0$ be an uniformly Rockland differential operator of order $m,$ and let $f$ be a locally integrable function on $G.$
\begin{enumerate}[{\rm (i)}]
\item{} If $2<\frac{d_{\hom}}{\gamma}<\infty,$ then
\[
\|M_f(1+P)^{-\frac{\gamma}{m}}\|_{\frac{d_{\hom}}{\gamma},\infty} \leqslant C_{\gamma,G,m}\|f\|_{L_{\frac{d_{\hom}}{\gamma}}(G)}.
\]
\item{} If $0 < \frac{d_{\hom}}{\gamma}\leqslant 2$ and $q>2,$ then
\[
\|M_f(1+P)^{-\frac{\gamma}{m}}\|_{\frac{d_{\hom}}{\gamma},\infty} \leqslant C_{\gamma,G,m}\|f\|_{\ell_{\frac{d_{\hom}}{\gamma}}(L_q)(G)}.
\]
\end{enumerate}
\end{lemma}
\begin{proof}
By Lemma \ref{ellipticity of the product}, $(P+1)^n$ is uniformly Rockland for every $n\in\mathbb{N}.$ By Theorem \ref{elliptic regularity theorem}, $(P+1)^n:W^{mn}_2(G)\to L_2(G)$ is self-adjoint. Thus, $(P+1)^{-n}:L_2(G)\to W^{mn}_2(G)$ is bounded. In other words, the operator $(1-\Delta)^{\frac{mn}{{ 2v}}}(P+1)^{-n}:L_2(G)\to L_2(G)$ is bounded for every $n\in\mathbb{N}.$ By the Hadamard 3 lines theorem, the operator
$$(1-\Delta)^{\frac{\gamma}{{ 2v}}}(1+P)^{-\frac{\gamma}{m}}$$
is bounded on $L_2(G)$ for all $\gamma>0.$ Hence, the factors of $(1-\Delta)^{-\frac{\gamma}{{ 2v}}}$ in \eqref{basic_cwikel_1} and \eqref{basic_cwikel_2} can be replaced by $(1+P)^{-\frac{\gamma}{m}},$ and this completes the proof.
\end{proof}

\begin{proof}[Proof of Theorem \ref{main theorem general case}] Let $0\leqslant f \in C^\infty_c(G)$ be arbitrary. Let $0\leqslant \phi\in C^\infty_c(G)$ be such that $f\phi=f,$ and $\phi^{\alpha}$ is smooth for every $\alpha>0.$ Let $\varepsilon>0$ and set $f_{\varepsilon} = (f+\varepsilon \phi)^{\frac{m}{2\gamma}}.$ So-defined $f_{\varepsilon}$ is a smooth compactly supported function such that $f_{\varepsilon}^{\alpha}$ is smooth for every $\alpha>0.$

Applying Theorem \ref{main theorem special case} to $f_{\epsilon}^2$ and using the fact that
$$M_{f_{\epsilon}}(1+P)^{-1}M_{f_{\epsilon}}-M_{f_{\epsilon}^2}(1+P)^{-1}\in(\mathcal{L}_{\frac{d_{{\rm hom}}}{m},\infty})_0,$$
we obtain
$$\lim_{t\to\infty} t^{\frac{m}{d_{\hom}}}\mu\Big(t,M_{f_{\epsilon}}(1+P)^{-1}M_{f_{\epsilon}}\Big) = \Big(\frac1{\Gamma(\frac{d_{\hom}}{m}+1)}\int_G f_{\epsilon}(g)^{\frac{2d_{\hom}}{m}}\tau(e^{-P_g^{{\rm top}}})\,dg\Big)^{\frac{m}{d_{{\rm hom}}}}.$$
Therefore,
$$\lim_{t\to\infty} t^{\frac{\gamma}{d_{\hom}}}\mu\Big(t,\big(M_{f_{\epsilon}}(1+P)^{-1}M_{f_{\epsilon}}\big)^{\frac{\gamma}{m}}\Big) = \Big(\frac1{\Gamma(\frac{d_{\hom}}{m}+1)}\int_G f_{\epsilon}(g)^{\frac{2d_{\hom}}{m}}\tau(e^{-P_g^{{\rm top}}})\,dg\Big)^{\frac{\gamma}{d_{{\rm hom}}}}.$$
Recall that $f^{\alpha} \in C^\infty_c(G)$ for every $\alpha>0.$ Applying Theorem \ref{ultimate_asterisque_esque_theorem} with $B = (1+P)^{-1},$ $A = M_{f_{\epsilon}^2}$ yields
$$(M_{f_{\epsilon}}(1+P)^{-1}M_{f_{\epsilon}})^{\frac{\gamma}{m}}-M_{f_{\epsilon}}^{\frac{\gamma}{m}}(1+P)^{-\frac{\gamma}{m}}M_{f_{\epsilon}}^{\frac{\gamma}{m}}\in (\Lc_{\frac{d_{\hom}}{\gamma},\infty})_0.$$
By Lemma \ref{easy_bs_lemma}, it follows that
$$\lim_{t\to\infty}t^{\frac{\gamma}{d_{\hom}}}\mu(t,M_{f_{\epsilon}}^{\frac{\gamma}{m}}(1+P)^{-\frac{\gamma}{m}}M_{f_{\epsilon}}^{\frac{\gamma}{m}})= \Big(\frac1{\Gamma(\frac{d_{\hom}}{m}+1)}\int_G f_{\epsilon}(g)^{\frac{2d_{\hom}}{m}}\tau(e^{-P_g^{{\rm top}}})\,dg\Big)^{\frac{\gamma}{d_{{\rm hom}}}}.$$
Taking into account that
$$M_{f_{\epsilon}}^{\frac{\gamma}{m}}(1+P)^{-\frac{\gamma}{m}}M_{f_{\epsilon}}^{\frac{\gamma}{m}}-M_{f_{\epsilon}}^{\frac{2\gamma}{m}}(1+P)^{-\frac{\gamma}{m}}\in (\Lc_{\frac{d_{\hom}}{\gamma},\infty})_0$$
and using Lemma \ref{easy_bs_lemma}, we infer 
$$\lim_{t\to\infty}t^{\frac{\gamma}{d_{\hom}}}\mu(t,M_{f_{\epsilon}}^{\frac{2\gamma}{m}}(1+P)^{-\frac{\gamma}{m}})= \Big(\frac1{\Gamma(\frac{d_{\hom}}{m}+1)}\int_G f_{\epsilon}(g)^{\frac{2d_{\hom}}{m}}\tau(e^{-P_g^{{\rm top}}})\,dg\Big)^{\frac{\gamma}{d_{{\rm hom}}}}.$$
Appealing to the definition of $f_{\epsilon},$ we write
$$\lim_{t\to\infty}t^{\frac{\gamma}{d_{\hom}}}\mu(t,M_{f+\epsilon\phi}(1+P)^{-\frac{\gamma}{m}})= \Big(\frac1{\Gamma(\frac{d_{\hom}}{m}+1)}\int_G f_{\epsilon}(g)^{\frac{d_{\hom}}{\gamma}}\tau(e^{-P_g^{{\rm top}}})\,dg\Big)^{\frac{\gamma}{d_{{\rm hom}}}}.$$

By Lemma \ref{improved_cwikel_estimate},
$$M_{f+\epsilon\phi}(1+P)^{-\frac{\gamma}{m}}\to M_f(1+P)^{-\frac{\gamma}{m}},\quad \epsilon\downarrow 0,$$
in $\Lc_{\frac{d_{\hom}}{\gamma},\infty}.$ The assertion follows now from the preceding paragraph and Lemma \ref{bs convergence lemma}.
\end{proof}

\begin{proof}[Proof of Corollary \ref{weaked_smoothness_corollary}]
Assume initially that $2<\frac{d_{\hom}}{\gamma}<\infty,$ and let $f \in L_{\frac{d_{\hom}}{\gamma}}(G).$ There exists a sequence $\{f_n\}_{n=0}^\infty\subset C^\infty_c(G)$ such that $f_n\to f$ in $L_{\frac{d_{\hom}}{\gamma}}(G),$
and hence by Lemma \ref{improved_cwikel_estimate}, the sequence
\[
\{M_{f_n}(1+P)^{-\frac{\gamma}{m}}\}_{n=0}^\infty
\]
is convergent in $\Lc_{\frac{d_{\hom}}{\gamma},\infty},$ call its limit $T.$ By Lemma \ref{bs convergence lemma} and Theorem \ref{main theorem general case}, the limits
\[
\lim_{t\to\infty} t^{\frac{\gamma}{d_{\hom}}}\mu(t,T) = \lim_{n\to\infty} \Big(\frac1{\Gamma(\frac{d_{\hom}}{m}+1)}\int_G |f_n(g)|^{\frac{d_{\hom}}{m}}\tau(e^{-P_g^{{\rm top}}})\,dg\Big)^{\frac{\gamma}{d_{{\rm hom}}}}
\]
exist and are equal. Since the function $g\mapsto \tau(e^{-P_g^{{\rm top}}})$ is uniformly bounded and continuous, we deduce that
\[
\lim_{t\to\infty} t^{\frac{\gamma}{d_{\hom}}}\mu(t,T) = \Big(\frac1{\Gamma(\frac{d_{\hom}}{m}+1)}\int_G |f(g)|^{\frac{d_{\hom}}{m}}\tau(e^{-P_g^{{\rm top}}})\,dg\Big)^{\frac{\gamma}{d_{{\rm hom}}}}
\]
{ By the Sobolev embedding theorem for graded Lie groups \cite[Theorem 4.2.25]{FischerRuzhansky2016}}, the Sobolev space $W^{\gamma}_2(G)$ embeds continuously into $L_{q}(G),$ where $\frac{1}{q}=\frac12- \frac{\gamma}{d_{\hom}}.$ By H\"older's inequality,
it follows that we have the convergence
\[
M_{f_n}\to M_f
\]
in the norm topology of $\Bc(W^{\gamma}_2(G),L_2(G)).$ Therefore, in the norm topology of $\Bc(L_2(G)),$ we have
\[
M_{f_n}(1+P)^{\frac{\gamma}{m}}\to M_f(1+P)^{-\frac{\gamma}{m}}.
\]
Therefore $T = M_f(1+P)^{-\frac{\gamma}{m}},$ and this completes the proof in the case that $2<\frac{d_{\hom}}{\gamma}<\infty.$ The case $0<\frac{d_{\hom}}{\gamma}\leqslant 2$ is similar.
\end{proof}
{
\subsection{Positive and negative eigenvalues}

Given a self-adjoint operator $T,$ denote $T_{+}$ and $T_-$ for the positive and negative parts of $T,$ repsectively. That is, $T_+ = \frac12(|T|+T)$ and $T_- = \frac12(|T|-T),$ so that $T=T_+-T_-.$ The positive eigenvalues of $T$ are simply the singular values of $T_+,$
and the negative eigenvalues are the negatives of the singular values of $T_-.$

\begin{theorem}\label{asterisque_but_now_there_are_signs}
    Let $p>0.$ Let $B\geqslant 0$ be a positive bounded operator, and let $A$ be self-adjoint. If
\begin{enumerate}[{\rm (i)}]
\item{}\label{signed_cwikel} $AB^{\frac12} \in \Lc_{2p,\infty},$
\item{}\label{signed_commutator_1} $ [B^{\frac12},A_{\pm}]B \in (\Lc_{\frac{2p}{3},\infty})_0$
\item{}\label{signed_commutator_2} $[B^{\frac32},A_{\pm}] \in (\Lc_{\frac{2p}{3},\infty})_0.$
\end{enumerate}
then
\[
    (B^{\frac12}AB^{\frac12})_{\pm} - B^{\frac{1}{2}}A_{\pm}B^{\frac{1}{2}} \in (\Lc_{p,\infty})_0.
\]
\end{theorem}
\begin{proof}
    Note that by the polar decomposition, \ref{signed_cwikel} implies that $A_{\pm}B^{\frac12} \in \Lc_{2p,\infty}.$
    
    Using $|A|^2=A^2,$ we write
    \begin{align*}
        |B^{\frac12}AB^{\frac12}|^2-(B^{\frac12}|A|B^{\frac12})^2 &= B^{\frac12}ABAB^{\frac12}-B^{\frac12}|A|B|A|B^{\frac12}\\
                                                                &= [B^{\frac12},A]BAB^{\frac12}+AB^{\frac32}AB^{\frac12}-[B^{\frac12},|A|]B|A|B^{\frac12}\\
                                                                &\quad - |A|B^{\frac32}|A|B^{\frac12}\\
                                                                &= [B^{\frac12},A]BAB^{\frac12}+[A,B^{\frac32}]AB^{\frac12}+B^{\frac32}A^2B^{\frac12}\\
                                                                &\quad -[B^{\frac12},|A|]B|A|B^{\frac12}-B^{\frac32}|A|^2B^{\frac12}-[|A|,B^{\frac32}]|A|B^{\frac12}\\
                                                                &= [B^{\frac12},A]BAB^{\frac12}+[A,B^{\frac32}]AB^{\frac12}\\
                                                                &\quad - [B^{\frac12},|A|]B|A|B^{\frac12}-[|A|,B^{\frac32}]|A|B^{\frac12}.\\
                                                                &\in (\Lc_{\frac{2p}{3},\infty})_0\cdot \Lc_{2p,\infty}\subseteq (\Lc_{\frac{p}{2},\infty})_0.
    \end{align*}
    Taking the square root using \eqref{general_raising_powers_principle}, we have
    \[
        |B^{\frac12}AB^{\frac12}| - B^{\frac12}|A|B^{\frac12} \in (\Lc_{p,\infty})_0.
    \]
    Since $T_{\pm} = \frac12(|T|\pm T),$ it follows that
    \[
        (B^{\frac12}AB^{\frac12})_{\pm} - B^{\frac12}A_{\pm}B^{\frac12} \in (\Lc_{p,\infty})_0.
    \]
\end{proof}

\begin{lemma}\label{cut_to_negative_part}
    Let $f\in C^\infty_c(G)$ be such that $f_{\pm}\in C^\infty_c(G),$ and let $P$ be an uniformly Rockland operator of order $m$ as in Theorem \ref{main theorem general case}. Then
    \[
        ((1+P)^{-\frac{\gamma}{2m}}M_f(1+P)^{-\frac{\gamma}{2m}})_{\pm}-(1+P)^{-\frac{\gamma}{2m}}M_{f_{\pm}}(1+P)^{-\frac{\gamma}{2m}} \in (\Lc_{\frac{d_{\hom}}{\gamma},\infty})_0. 
    \]
\end{lemma}
\begin{proof}
    By Lemma \ref{improved_cwikel_estimate} and Theorem \ref{big_fractional_power_theorem}, we have
    
%
\[
    M_{f}(1+P)^{-\frac{\gamma}{2m}} \in \Lc_{\frac{2d_{\hom}}{\gamma},\infty}
\]
and
\[
    [(1+P)^{-\frac{3\gamma}{2m}},M_{f_{\pm}}], [(1+P)^{-\frac{\gamma}{2m}},M_{f_{\pm}}](1+P)^{-\frac{\gamma}{m}} \in (\Lc_{\frac{2d_{\hom}}{3\gamma},\infty})_0.
\]
This verifies the requirements of Theorem \ref{asterisque_but_now_there_are_signs} for $A=M_f$ and $B = (1+P)^{-\frac{\gamma}{m}}$ and $p = \frac{d_{\hom}}{\gamma}.$ 

It follows that
\[
        ((1+P)^{-\frac{\gamma}{2m}}M_f(1+P)^{-\frac{\gamma}{2m}})_{\pm}-(1+P)^{-\frac{\gamma}{2m}}M_{f_{\pm}}(1+P)^{-\frac{\gamma}{2m}} \in (\Lc_{\frac{d_{\hom}}{\gamma},\infty})_0,\quad f_{\pm} \in C^\infty_c(G).
\]
\end{proof}

Before proving our main theorem, we display an approximation lemma:
\begin{lemma}\label{lem-dense-decomposition-f-matrix}
    $\forall f\in C_c^{\infty}(G, M_n(\C))$, with $f^*(g)=f(g)$ for all $g\in G$, there exists a series of functions $\{f_{\epsilon}\}_{\epsilon > 0} \subset C_c^{\infty}(G, M_n(\C))$, such that:
    \begin{enumerate}
        \item $f_{\epsilon}^*(g) = f_{\epsilon}(g)$, $\forall g\in G, \forall \epsilon >0$.
        \item $ (f_{\epsilon})_{\pm} \in C^{\infty}_c(G, M_n(\C))$, $\forall \epsilon >0$.
        \item $ \forall p, q >0$, $f - f_{\epsilon}\to 0$, in $L_p$ and $\ell_p(L_q)(G)$, when $\epsilon \to 0$.
    \end{enumerate}
\end{lemma}
\begin{proof}
    Let $\epsilon>0$, choose a function $\psi_{\epsilon} \in C^{\infty}(\R)$, such that
    \begin{enumerate}
        \item $\psi_{\epsilon}(x)=0$, when $|x|<\epsilon/2$,
        \item $\psi_{\epsilon}(x)=x$, when $|x|>\epsilon$,
        \item $\psi_{\epsilon} \geqslant 0$, for $x>0$. 
        \item $\psi_{\epsilon} \leqslant 0$, for $x<0$.
        \item $|\psi_{\epsilon}(x)-x|\leqslant \epsilon$, for $x\in\mathbb{R}.$
    \end{enumerate}
    Since $f$ and $\psi$ are smooth and $f$ is self-adjoint valued, we define
    \begin{equation*}
        f_\epsilon  \coloneqq \psi_{\epsilon} \circ f \in  C^{\infty}_c(G, M_n(\C))
    \end{equation*}
    by functional calculus.
    Since $\psi$ is real valued, $f_{\epsilon}^* = f_{\epsilon}$.
    Notice that $|\psi_{\epsilon}| \in C^{\infty}(\R)$, thus
    \begin{equation*}
        |f_\epsilon| = |\psi_{\epsilon}| \circ f  \in  C^{\infty}_c(G, M_n(\C)),
    \end{equation*}
    Evidently, $(f_{\epsilon})_{+} = \frac{1}{2}(f_{\epsilon} + |f_{\epsilon}|)$, $(f_{\epsilon})_{-} = \frac{1}{2}(-f_{\epsilon} + |f_{\epsilon}|)$ are both smooth.
    By the construction of $\psi_{\epsilon}$, we have $|\psi_{\epsilon}(x)-x|\leqslant \epsilon$, for all $x\in \R$. Hence,
    $$\|f_{\epsilon}(g)-f(g)\|_{M_n(\mathbb{C})}\leq \epsilon,\quad g\in G.$$
    Thus,
    \begin{align*}
        \Vert f_{\epsilon} - f\Vert_{L_{p}(G)}
        & =\Big(\int_G \|(f_{\epsilon}-f)(g)\|_{L_p(M_n(\mathbb{C}),{\rm Tr})}^pdg\Big)^{\frac1p}\\
        & \leqslant n^{\frac1p}\Big(\int_G \|(f_{\epsilon}-f)(g)\|_{M_n(\mathbb{C})}^pdg\Big)^{\frac1p}\\
        & \leqslant n^{\frac1p}\cdot m({\rm supp}(f))^{\frac1p}\cdot \sup_{g\in G}\|f_{\epsilon}(g)-f(g)\|_{M_n(\mathbb{C})}\\
        & \leqslant n^{\frac1p}\cdot m({\rm supp}(f))^{\frac1p}\cdot \epsilon.
    \end{align*}
    A similar computation shows that $f_{\varepsilon}\to f$ in the $\ell_p(L_q)(G)$ norm.
\end{proof}

\begin{proof}[Proof of Theorem \ref{positive_negative_parts_theorem}]
    For $f$ such that $f_{\pm}\in C^\infty_c(G)$, this is an immediate consequence of Theorem \ref{main theorem general case}, Lemma \ref{cut_to_negative_part}    
    and Lemma \ref{easy_bs_lemma}. 
    
    By Lemma \ref{improved_cwikel_estimate}, Lemma \ref{lem-dense-decomposition-f-matrix}, and an approximation argument using Lemma \ref{easy_bs_lemma}, the result of the theorem still holds for general real $f \in C^\infty_c(G)$, and self-adjoint valued $f\in C^\infty_c(G, M_n(\C))$. 
    
    Since $C^\infty_c(G)$ itself is dense in $L_{\frac{d_{\hom}}{\gamma}}(G)$ and in $\ell_{\frac{d_{\hom}}{\gamma}}(L_q)(G),$ the same approximation argument completes the proof for general $f.$
\end{proof}
}


\appendix

\section{Computations for the group algebra trace}\label{appdx-alg-trace}
In this appendix we give computational formulas 
for $\tau(f(P)),$ where $P$ is an uniformly Rockland positive-definite element of $\Uc(\gf),$ $f$ is a function on the spectrum of $P$ and $\tau$ is the von Neumann trace on the group von Neumann algebra $\mathrm{VN}(G).$ In order to do this, we must first verify that $f(P)$ belongs to the domain of the trace. 

The key fact is that if $P\geqslant 0$ is a homogeneous uniformly Rockland operator, then the functional
\[
    f\mapsto \tau(f(P))
\]
is represented by a measure on $[0,\infty).$ Once the measure is known to exist, it can be computed up to a constant by the homogeneity of $P$ and \eqref{trace_scaling}. To prove this, we use the following Sobolev embedding result.
\begin{theorem}\label{sobolev_embedding}\cite[Theorem 4.4.25(ii)]{FischerRuzhansky2016}
    Let $G$ be a graded Lie group. For sufficiently large $N,$ depending on $G,$ we have 
    \[
        W^N_2(G)\subset C(G).
    \]
\end{theorem}

First, we need to verify that elements of $\Uc(\gf),$ identified with their image under $\lambda,$ are affiliated to $\mathrm{VN}(G).$

\begin{lemma}\label{affiliated_to_M} The hermitian elements of $\Uc(\gf),$ thought of as right invariant unbounded operators on $L_2(G),$ are affiliated with $\mathrm{VN}(G).$\footnote{Recall that a self-adjoint operator $A$ is affiliated to a von Neumann algebra $\mathcal{M}$ if the spectral projections of $A$ belong to $\mathcal{M}.$ Equivalently, $f(A) \in \mathcal{M}$ for all bounded continuous functions $f$}

\end{lemma}
\begin{proof} Let $X \in \gf$, its action $\lambda(X)$ on $u \in L_2(G)$ is defined as 
\begin{equation}\label{eq-def-right-action-Lie-alg}
    \lambda(X)u(g) = \frac{d}{dt}\Big|_{t=0} u(e^{-tX}g).
\end{equation}
Since $\lambda(X)$ is the generator of a unitary semigroup, Stone's theorem implies that $\lambda(X)$ is anti-self-adjoint. The representation $\lambda$ of $\gf$ on $L_2(G)$ extends to a representation of $\Uc(\gf)$ on the subspace of smooth vectors.

Since $\lambda(X)$ commutes with the right-action $\rho(g): u(\cdot)\to u(\cdot g)$ for $u\in L_2(G)$, it follows that for any $P \in \Uc(\gf),$ we have
$$\lambda(P)\rho(g) = \rho(g)\lambda(P).$$
Hence, $\rho(g)$ commutes with any spectral projection of $\lambda(P).$ Since $\mathrm{VN}(G)$ is the commutant of $\rho(G),$ this shows that the spectral projections of $\lambda(P)$ belong to $\mathrm{VN}(G).$
\end{proof}

\begin{lemma}\label{lem:exponentials_are_trace_class} Let $P\geqslant 0$ be a positive uniformly Rockland element of $\Uc(\gf).$ We have
$$e^{-tP} \in L_1(\mathrm{VN}(G),\tau),\quad t>0.$$
Moreover, $\exp(-tP) = \lambda(h_t)$ for some left-convolution kernel $h_t \in (L_2(G)\ast L_2(G))\cap C(G),$ and $\tau(e^{-tP}) = h_t(1_G).$
\end{lemma}
\begin{proof} 
    Since $\exp(-tP) = \exp(-(t/2)P)^2,$ is suffices to show that $\exp(-tP)\in L_2(\mathrm{VN}(G),\tau)$ for all $t>0.$ Equivalently,
    $\exp(-tP) = \lambda(h_t)$ for some $h_t \in L_2(G).$ By Theorem \ref{sobolev_embedding}, for sufficienty large $N>0$ we have
    \[
        \delta_{1_G} \in W^{-N}_2(G).
    \]
    Since $P$ is self-adjoint, $\exp(-tP)$ maps $W^{-N}_2(G)$ into $L_2(G).$ Thus,
    \[
        \exp(-tP)\delta_{1_G} \in L_2(G).
    \]
    The Schwartz kernel theorem implies that
    \[
        \exp(-tP) = \lambda(\exp(-tP)\delta_{1_G}).
    \]    
    That is, $\exp(-tP)\in L_2(\mathrm{VN}(G),\tau).$

\end{proof}

\begin{corollary}
    Let $P\geqslant 0$ be uniformly Rockland of order $m>0.$ For all $f \in C_c([0,\infty)),$ $f(P) \in L_1(\VN(G),\tau).$
\end{corollary}
\begin{proof}
    Since $P$ is affiliated to $\mathrm{VN}(G)$ and the function $g(x) = f(x)e^x$ is bounded, we have $g(P)\in \mathrm{VN}(G).$ Since $\exp(-P)\in L_1(\mathrm{VN}(G),\tau),$ it follows that $f(P) = g(P)\exp(-P)\in L_1(\mathrm{VN}(G),\tau).$
\end{proof}

\begin{lemma}\label{lem:continuous_limit_formula} Let $P\geqslant 0$ be elliptic, constant coefficient order $m$ homogeneous differential operator. If $f \in L_1(\mathbb{R}_+,x^{\frac{d_{\hom}}{m}-1}dx),$ then $f(P)\in L_1(\mathrm{VN}(G),\tau)$ and
$$\tau(f(P)) = \frac{1}{\Gamma(\frac{d_{\hom}}{m})}\int_0^\infty x^{\frac{d_{\hom}}{m}-1}f(x)\,dx\cdot \tau(e^{-P}).$$
\end{lemma}
\begin{proof} Recall that order $m$ homogeneity means that
\begin{equation}\label{eq-P-homogeneous-rescale}
\alpha_t(P) = t^m P,\quad t>0.
\end{equation}

By Lemma \ref{lem:exponentials_are_trace_class}, $e^{-tP}$ belongs to the domain of $\tau$ for every $t.$ Since $P$ is homogeneous of order $m$, \eqref{eq-P-homogeneous-rescale} and functional calculus implies that
\begin{equation*}
    \alpha_{t^{1/m}}(e^{-P}) = e^{-tP}.
\end{equation*} 
Taking the trace $\tau$ of both sides, \eqref{trace_scaling} implies
$$\tau(e^{-tP}) = \tau(\alpha_{t^{1/m}}(e^{-P})) = t^{-\frac{d_{\hom}}{m}}\tau(e^{-P})=\frac{1}{\Gamma(\frac{d_{\hom}}{m})}\int_0^{\infty}x^{\frac{d_{{\rm hom}}}{m}-1}e^{-tx}dx\cdot \tau(e^{-P}).$$
In particular, the assertion of the lemma holds for $f$ in the linear span of the functions $\{x\to e^{-tx}\}_{t>0}.$ By the Monotone Convergence Theorem and the normality of $\tau,$ the assertion of the lemma holds for $0\leqslant f \in L_1(\mathbb{R}_+,x^{\frac{d_{\hom}}{m}-1}dx)$ and, hence, for every $f \in L_1(\mathbb{R}_+,x^{\frac{d_{\hom}}{m}-1}dx).$
\end{proof}

The Plancherel theorem expresses the von Neumann trace $\tau$ in terms of the unitary dual of $G.$ In this case, if $A \in L_1(\mathrm{VN}(G),\tau)$ we have
\[
\tau(A) = \int_{\widehat{G}} \mathrm{Tr}_{H_{\pi}}(\pi(A))\,d\pi.
\]
Here, $\widehat{G}$ is the space of unitary irreducible unitary representations $(\pi,H_{\pi})$ of $G,$
and $d\pi$ denotes the correctly normalised Plancherel measure. This is a straightforward consequence of the corresponding assertion for $L_2(G),$ as in \cite[Section 18.8]{DixmierCStar1977}.
Hence,
\[
\tau(e^{-P}) = \int_{\widehat{G}} \mathrm{Tr}_{H_{\pi}}(e^{-\pi(A)})\,d\pi.
\]
There are a few examples where this integral can be explicitly computed.

\begin{example}\label{example-Rd-alg-trace}
Let $G=\mathbb{R}^d$ be { equipped with its usual abelian Lie group structure and trivial grading}, and let $P\in\Uc(\mathbb{R}^d)$ be a homogeneous positive  uniformly Rockland differential operator of order $m.$ If $p$ is its principal symbol, then
$$\tau(f(P))=\frac1{m(2\pi)^d}\int_0^{\infty}x^{\frac{d}{m}-1}f(x)dx\cdot \int_{\mathbb{S}^{d-1}}p(s)^{-\frac{d}{m}}ds.$$
In particular, if $m=2$ and $P=-\sum_{j,k=1}^da_{jk}\partial_j\partial_k$ then
$$\tau(f(P))= \frac{(4\pi)^{-\frac{d}{2}}}{\Gamma(\frac{d}{2})}\int_0^\infty x^{\frac{d}{2}-1}f(x)\,dx\cdot \det(A)^{-\frac12}$$
where $A = \{a_{j,k}\}_{j,k=1}^d.$
\end{example}
\begin{proof} If $G = \mathbb{R}^d,$ then $\widehat{G} = \mathbb{R}^d,$ where the irreducible representations $(\pi_{\xi},\mathbb{C})$ are labelled by $\xi \in \mathbb{R}^d$ with the Plancherel measure
$$d\pi_{\xi} = (2\pi)^{-d}d\xi.$$
We write the corresponding operator as $P=p(\nabla).$ Thus,
$$\tau(e^{-P}) = (2\pi)^{-d} \int_{\mathbb{R}^d} e^{-p(\xi)}\,d\xi.$$
Passing to polar coordinates and using the homogeneity of $p,$ we rewrite the latter equality as
$$\tau(e^{-P}) = (2\pi)^{-d} \int_0^\infty\int_{\mathbb{S}^{d-1}} r^{d-1} e^{-r^mp(s)}dsdr = \frac{\Gamma(\frac{d}{m})}{m(2\pi)^d}\int_{\mathbb{S}^{d-1}}p(s)^{-\frac{d}{m}}ds.$$
Hence, for any function $f$ satisfying the hypothesis of Lemma \ref{lem:continuous_limit_formula}, we have
$$\tau(f(P)) = \frac{1}{m(2\pi)^{d}}\int_0^\infty x^{\frac{d}{m}-1}f(x)\,dx\cdot \int_{\mathbb{S}^{d-1}} p(s)^{-\frac{d}{m}}ds.$$

If $m=2$ and $P=-\sum_{j,k=1}^da_{j,k}\partial_j\partial_k$ then
$$\tau(e^{-P})=(2\pi)^{-d}\int_{\mathbb{R}^d}e^{-p(\xi)}\,d\xi=2^{-d}\pi^{-\frac{d}{2}}\det(A)^{-\frac12}$$
and the "in particular" part again follows from Lemma \ref{lem:continuous_limit_formula}.
\end{proof}
{

\begin{example}
    Similar to the previous example, we consider $\mathbb{R}^d$ with a non-standard grading. To do so, the sphere $\mathbb{S}^{d-1}$ is replaced by an anisotropic version, and the Lebesgue measure on the sphere is changed appropriately.

    Let $G = \mathbb{R}^d$ be equipped with its usual abelian Lie group structure, but equipped with the anisotropic dilations
    \[
        \delta_t = \bigoplus_{j=1}^d t^{v_j},\quad t >0.
    \]
    where $\boldsymbol{v} := (v_1,\ldots,v_d)$ is a $d$-tuple of positive integers. This is a graded Lie group, with $\gf = \bigoplus_{k=1}^\infty \gf_k$ where $\gf_k$ is the subspace of vectors whose $j$th component vanishes whenever $v_j\neq k,$ and $d_{\hom} = v_1+\cdots+v_d.$ 
    
    Let $v = \mathrm{lcm}(v_1,\ldots,v_d)$ and define
    \[
        \mathbb{S}^{d-1}_{\boldsymbol{v}} = \{(s_1,\ldots,s_d)\in \mathbb{R}^d\;:\; \sum_{j=1}^d s_j^{\frac{2v}{v_j}}=1\}.
    \]
    Let $\mu_{\boldsymbol{v}}$ be the Borel measure on $\mathbb{S}^{d-1}_{\boldsymbol{v}}$ defined on a Borel set $A\subseteq \mathbb{S}^{d-1}_v$ by
    \[
        \mu_{\boldsymbol{v}}(A) := d_{\hom}\big|\bigcup_{0\leq t \leq 1} \delta_tA\big|
    \]
    where $|\cdot|$ is the Lebesgue measure in $\mathbb{R}^d.$
    
    Let $P\in \Uc(\gf)$ be homogeneous, positive definite and uniformly Rockland of order $m$ with respect to $\delta_t.$ If $p$ is its symbol, then
    \[
        \tau(f(P)) = \frac{1}{m(2\pi)^d}\int_0^\infty x^{\frac{d_{\hom}}{m}-1}f(x)\,dx \cdot \int_{\mathbb{S}^{d-1}_{\boldsymbol{v}}} p(s)^{-\frac{d_{\hom}}{m}}\,d\mu_v(s).
    \]
\end{example}
\begin{proof}
    This computation is very similar to Example \ref{example-Rd-alg-trace}. Since $G$ is precisely the same group as in that example, we again have $\widehat{G} = \mathbb{R}^d$ with Plancherel measure
    \[
        d\pi_{\xi} = (2\pi)^{-d}d\xi.
    \]
    Writing $P = p(\nabla),$ we again have
    \[
        \tau(e^{-P}) = (2\pi)^{-d}\int_{\mathbb{R}^d} e^{-p(\xi)}\,d\xi.
    \] 
    The measure $\mu_{\boldsymbol{v}}$ is defined so that
    \begin{align*}
        \tau(e^{-P}) &= (2\pi)^{-d}\int_0^\infty\int_{\mathbb{S}_{\boldsymbol{v}}^{d-1}} r^{d_{\hom}-1}e^{-r^mp(s)}\,d\mu_{\boldsymbol{v}}(s)dr\\
                    &= \frac{\Gamma(\frac{d_{\hom}}{m})}{m(2\pi)^d}\int_{\mathbb{S}^{d-1}_{\boldsymbol{v}}}p(s)^{-\frac{d_{\hom}}{m}}\,d\mu_{\boldsymbol{v}}(s).
    \end{align*}
    From this point, the computation proceeds exactly as in Example \ref{example-Rd-alg-trace}.
\end{proof}
}

In the following example we work with the Heisenberg group $\mathbb{H}^n.$ We follow closely the notation of \cite[Chapter 6]{FischerRuzhansky2016}. The $(2n+1)$-dimensional Heisenberg algebra is spanned by $\{X_1,\ldots,X_{2n},T\}$ satisfying $[X_j,X_{n+k}] = \delta_{j,k}T.$ with all other Lie brackets vanishing.

\begin{example} Let $G=\mathbb{H}^{2n+1}$ and let $P\in\Uc(\mathbb{H}^{2n+1})$ be a homogeneous positive uniformly Rockland differential operator of order $m.$ We have
$$\tau(f(P))=(2\pi)^{-3n-1}\frac{2}{m}\int_0^\infty x^{\frac{2n+2}{m}-1}f(x)\,dx\cdot \sum_{s\in\{1,-1\}}\mathrm{Tr}_{L_2(\mathbb{R}^n)}(\pi_s(P)^{-\frac{2n+2}{m}}).$$
In particular, if $m=2$ and $P=-\sum_{j,k=1}^{2n}a_{jk}X_jX_k,$ then
$$\tau(f(P))=\frac{(2\pi)^{-(3n+1)}}{2^{n-1}n!}\det(A)^{-\frac12}\int_0^\infty x^{\frac{d_{\hom}}{2}-1}f(x)\,dx\cdot \int_{-\infty}^\infty \det(\frac{i\Omega As}{\sinh(i\Omega As)})^{\frac12}\,ds$$
where $A = \{a_{j,k}\}_{j=1}^{2n}.$
\end{example}
\begin{proof} The unitary dual of the $(2n+1)$-dimensional Heisenberg group is, up to a Plancherel null set, parametrised by $s\in\mathbb{R}\backslash\{0\}.$ The representation $\pi_s$ corresponding to $s \in \mathbb{R}\backslash\{0\}$ is the Schr\"odinger representation on $L_2(\mathbb{R}^n)$ 
$$\pi_s(X_j)=i\sqrt{|s|}p_j,\quad \pi_s(X_{j+n}) = i\mathrm{sgn}(s)\sqrt{|s|}q_j,\quad \pi_s(T) = is,\quad 1\leqslant j\leqslant n.$$
Here, $\{q_j\}_{j=1}^n$ and $\{p_j\}_{j=1}^n$ are position and momentum operators. In this parametrisation, the Plancherel measure is given by
$$d\pi_s = (2\pi)^{-(3n+1)}|s|^n\,ds.$$
See \cite[Section 6.2.3]{FischerRuzhansky2016}.

Hence,
$$\tau(e^{-P}) = (2\pi)^{-(3n+1)}\int_{-\infty}^\infty |s|^n \mathrm{Tr}_{L_2(\mathbb{R}^n)}(e^{-\pi_s(P)})\,ds.$$
If $P$ is homogeneous of order $m$ and $m$ is even, then
$$\pi_s(P) = |s|^{\frac{m}{2}}\pi_{\mathrm{sgn}(s)}(P).$$
It follows that
\begin{align*}
& (2\pi)^{3n+1}\tau(e^{-P})\\
& \quad =\int_0^\infty s^n \mathrm{Tr}_{L_2(\mathbb{R}^n)}(e^{-s^{\frac{m}{2}}\pi_{+1}(P)})ds+\int_0^\infty s^n \mathrm{Tr}_{L_2(\mathbb{R}^n)}(e^{-s^{\frac{m}{2}}\pi_{-1}(P)})\,ds\\ 
& \quad =\mathrm{Tr}_{L_2(\mathbb{R}^n)}\Big(\int_0^\infty s^n e^{-s^{\frac{m}{2}}\pi_{+1}(P)}ds\Big)+\mathrm{Tr}_{L_2(\mathbb{R}^n)}\Big(\int_0^\infty s^n e^{-s^{\frac{m}{2}}\pi_{-1}(P)}ds\Big).
\end{align*}
Using the formula
$$\frac{m}{2}\Gamma(\frac{2n+2}{m})A^{-\frac{2n+2}{m}}=\int_0^\infty s^n e^{-s^{\frac{m}{2}}A}\,ds,$$
we obtain the first assertion of the lemma.

We now specialise to $m=2.$ Let the operator $P$ have the form $P=-\sum_{j,k=1}^{2n} a_{j,k}X_jX_k$ for some positive scalar matrix $A = \{a_{j,k}\}_{j,k=1}^n.$ By the Williamson theorem, there exists a symplectic matrix $S$ and a diagonal matrix $D$ such that $A = S^{\ast}DS.$ The entries of $D = \{\lambda_{j}\delta_{j,k}\}_{j,k=1}^{2n}$ satisfy $\lambda_j=\lambda_{j+n},$ and $\{\lambda_j\}_{j=1}^n$ are the positive eigenvalues of $i\Omega A,$ where $\Omega$ is the standard $2n\times 2n$ symplectic matrix. Since symplectic automorphisms preserve the Lie algebra of the $(2n+1)$-dimensional Heisenberg group, it follows that
\begin{multline*}
    \mathrm{Tr}_{L_2(\mathbb{R}^n)}(e^{-s\pi_{\pm1}(P)})
    =\mathrm{Tr}_{L_2(\mathbb{R}^n)}(\exp(s\sum_{k=1}^{n} \lambda_{k}\pi_{\pm1}(X_k^2+Y_k^2)))\\
    =\mathrm{Tr}_{L_2(\mathbb{R}^n)}(\exp(-s\sum_{k=1}^{n} \lambda_{k}(p_k^2+q_k^2)))
    =\prod_{k=1}^n{\rm Tr}_{L_2(\mathbb{R})}(\exp(-\lambda_k(p^2+q^2))).
\end{multline*}
As is well-known, $p^2+q^2$ has eigenvalues $\{2l+1\}_{l\geqslant 0},$ and hence
$$\mathrm{Tr}_{L_2(\mathbb{R}^n)}(e^{-s\pi_{\pm1}(P)})= \prod_{k=1}^n \left(\sum_{l=0}^\infty \exp(-s\lambda_k(2l+1))\right)=\prod_{k=1}^n \frac{1}{2\sinh(\lambda_ks)}.$$

Since $\{\lambda_k\}_{k=1}^n$ are the positive eigenvalues of $i\Omega A$ and since the negative eigenvalues of $i\Omega A$ are $\{-\lambda_k\}_{k=1}^n,$ it follows that
$$\Big(\prod_{k=1}^n\lambda_k\Big)^{-2}=(-1)^n{\rm det}(i\Omega A)=(-1)^ni^{2n}{\rm det}(\Omega){\rm det}(A)={\rm det}(A),$$
$$\Big(\prod_{k=1}^n \frac{\lambda_k}{\sinh(\lambda_ks)}\Big)^2=\det(\frac{i\Omega As}{\sinh(i\Omega As)}).$$
Thus,
$$\prod_{k=1}^n \frac{1}{\sinh(\lambda_ks)}=\Big(\prod_{k=1}^n\lambda_k\Big)^{-1}\cdot \Big(\prod_{k=1}^n \frac{\lambda_k}{\sinh(\lambda_ks)}\Big)={\rm det}(A)^{-\frac12}\cdot \det(\frac{i\Omega As}{\sinh(i\Omega As)})^{\frac12}.$$

Therefore,
\begin{multline*}
    (2\pi)^{3n+1}\tau(e^{-P})= 2\int_0^\infty s^{n}\prod_{k=1}^n \frac{1}{2\sinh(s\lambda_k)}ds\\
    =2^{1-n}\det(A)^{-\frac12}\int_{-\infty}^\infty \det(\frac{i\Omega As}{\sinh(i\Omega As)})^{\frac12}ds.
\end{multline*}
The "in particular" part again follows from Lemma \ref{lem:continuous_limit_formula}.
\end{proof}

\section{Cwikel-type estimates for graded Lie groups}\label{cwikel_appendix}
We have relied a number of times on weak-Schatten class estimates for operators of the form $M_f(1-\Delta)^{-\frac{\gamma}{2v}}$ for $\gamma>0$ and $f$ is a function on $G.$ Estimates of this type were derived in \cite[{\reva Theorem 1.1}]{MSZ-stratified-23}, but only in the special case where $G$ is stratified. Nevertheless, the main results and most of the proofs in \cite{MSZ-stratified-23} continue to hold with at most minor modification in the general case. In this appendix we explain the necessary modifications. The main goal is to prove Lemmas \ref{specific_cwikel_estimates_from_jfa_lemma} and \ref{quantitative_cwikel_estimate}. Since the proofs are similar we will omit some unnecessary details.

However, it should be noted that due to a simple argument that was overlooked in \cite{MSZ-stratified-23}, the $\gamma = \frac{d_{\hom}}{2}$ case of Lemma \ref{quantitative_cwikel_estimate} is actually stronger than the corresponding result in \cite{MSZ-stratified-23}.

\subsection{Proof of Lemma \ref{specific_cwikel_estimates_from_jfa_lemma}}
The following result is similar to Lemma 4.4 of \cite{MSZ-stratified-23}, with the key difference that we use Lemma \ref{lem:continuous_limit_formula}.
\begin{proposition}\label{functional_calculus_is_where_it_should_be}
    Let $g\in L_{1,\mathrm{loc}}(\mathbb{R}_+).$
    We have $g(-\Delta) \in L_{p,\infty}(\mathrm{VN}(G),\tau)$ if and only if $g \in L_{p,\infty}([0,\infty),x^{\frac{d_{\hom}}{2v}-1}\,dx),$ with an equality of quasinorms.
\end{proposition}
\begin{proof}
    Consider the bounded linear mapping
    \begin{equation}\label{functional_calculus}
        L_{\infty}((0,\infty),x^{\frac{d_{\hom}}{2v}-1}\,dx)\to \mathrm{VN}(G),\quad g\mapsto g(-\Delta).
    \end{equation}
    We have in Lemma \ref{lem:continuous_limit_formula} computed $\tau(f(P))$ for any homogeneous uniformly Rockland operator $P,$ for arbitrary $f\in C_c((0,\infty)).$ In particular, if $g$ is continuous and compactly supported, we have
    \[
        \tau(g(-\Delta)) = \frac{1}{\Gamma(\frac{d_{\hom}}{2v})}\int_0^\infty x^{\frac{d_{\hom}}{2v}-1}g(x)\,dx\cdot \tau(e^{\Delta}).
    \]
    Hence, the linear map in \eqref{functional_calculus} is trace-preserving. It hence induces an isometry of weak $L_p$-spaces.

\end{proof}

    In Theorem 3.3 of \cite{MSZ-stratified-23}, it is proved that if $p>2,$ $G$ is a unimodular group, $f \in L_p(G)$ and $T \in L_{p,\infty}(\mathrm{VN}(G),\tau),$ then $M_fT\in \Lc_{p,\infty}(L_2(G)).$ Applying this to $T = g(P)$ gives the following:
    \begin{theorem}\label{easy_cwikel}
        For $p>2,$ we have
        \[
            \|M_fg(-\Delta)\|_{\Lc_{p,\infty}(L_2(G))} \leqslant c_p\|f\|_{L_p(G)}\|g\|_{L_{p,\infty}([0,\infty),x^{\frac{d_{\hom}}{2v}-1}\,dx)}.
        \]
        In particular, if $p>2$ and $f \in L_p(G)$ then
        \[
            M_f(1-\Delta)^{-\frac{d_{\hom}}{2vp}} \in \Lc_{p,\infty}(L_2(G))
        \]
        with quasi-norm bounded by $\|f\|_{L_p(G)}.$
    \end{theorem}
    
    We push this to $0 < p \leqslant 2$ by taking repeated commutators of $M_f$ with $(1-\Delta)^{-1}.$ A very similar argument appeared in \cite[{\reva Lemma 5.2}]{MSZ-stratified-23}.

    \begin{proof}[Proof of Lemma \ref{specific_cwikel_estimates_from_jfa_lemma}]
        The second claim is simply the statement that $M_f$ is a bounded multiplier from $W^{2v\beta}_2(G)$ to $W^{-2v\alpha}_2(G)$ when $-\alpha=\beta.$ This is a consequence of Lemma \ref{multiplication_lemma}. Hence we assume that $\alpha+\beta>0$ and focus on the first claim.
    
        \textbf{Step 1.} We show that for all 
        \[
            0 \leqslant \alpha < \frac{d_{\hom}}{4v}
        \]
        and integer $n\geqslant 0$ that
        \[
            (1-\Delta)^nM_f(1-\Delta)^{-n-\alpha} \in \Lc_{\frac{d_{\hom}}{2v\alpha},\infty}.
        \]
        For $n=0,$ this follows from Theorem \ref{easy_cwikel}, with 
        \[
            p = \frac{d_{\hom}}{2v\alpha} > 2.
        \]
        We write
        \[
            (1-\Delta)^{n+1}M_f(1-\Delta)^{-n-1-\alpha} = (1-\Delta)^nM_f(1-\Delta)^{-n-1-\alpha}+(1-\Delta)^n[P,M_f](1-\Delta)^{-n-1-\alpha}.
        \]
        The first summand belongs to $\Lc_{\frac{d_{\hom}}{2v\alpha},\infty}$ by the inductive hypothesis. The second is a linear combination of terms of the form 
        \[
            (1-\Delta)^nM_{X^{\gamma}f}(1-\Delta)^{-n-\alpha}\cdot (1-\Delta)^{n+\alpha}X^{\beta}(1-\Delta)^{-n-1-\alpha},\quad \mathrm{len}(\beta)\leqslant 2v-1.
        \]
        The product
        \[
         (1-\Delta)^{n+\alpha}X^{\beta}(1-\Delta)^{-n-1-\alpha}.
        \]
        is bounded for all $n.$ Hence, by the inductive hypothesis, we conclude that
        \[
            (1-\Delta)^{n+1}M_f(1-\Delta)^{-n-1-\alpha} \in \Lc_{\frac{d_{\hom}}{2v\alpha},\infty}.
        \]
        This completes Step 1.\\
        \textbf{Step 2.}
        We also show that for all integers $n<0$ we have
        \[
            (1-\Delta)^nM_f(1-\Delta)^{-n-\alpha} \in \Lc_{\frac{d_{\hom}}{2v\alpha},\infty}.
        \]
        This can be proved by induction on $n,$ identically to Step 1. The details are omitted.       \\
        \textbf{Step 3.}
        By the Phragmen-Lindel\"of principle, we conclude that for all $z\in \mathbb{C}$ we have
        \[
            (1-\Delta)^zM_f(1-\Delta)^{-z-\alpha} \in \Lc_{\frac{d_{\hom}}{2v\alpha},\infty}.
        \]
        \textbf{Step 4.}
        By Step 2, for all $k\geqslant 0$ and $z \in \mathbb{C}$ we have
        \[
            (1-\Delta)^{z+k\alpha}M_f(1-\Delta)^{-z-(k+1)\alpha} \in \Lc_{\frac{d_{\hom}}{2v\alpha},\infty}.
        \]        
        Writing $f = f_1\cdots f_n,,$ by H\"older's inequality we have
        \begin{align*}
            &(1-\Delta)^{z}M_f(1-\Delta)^{-z-n\alpha}\\
            &= (1-\Delta)^zM_{f_1}(1-\Delta)^{-z-\alpha}(1-\Delta)^{z+\alpha}M_{f_2}(1-\Delta)^{-z-2\alpha} \cdot (1-\Delta)^{z+2\alpha}M_{f_2}(1-\Delta)^{-z-3\alpha}\\
            &\quad\cdots (1-\Delta)^{z+(n-1)\alpha}M_{f_n}(1-\Delta)^{-z-n\alpha}\\
            &\in \Lc_{\frac{d_{\hom}}{2vn\alpha},\infty}.
        \end{align*}
        Since $0<\alpha< \frac{d_{\hom}}{4v}$ and $n\geqslant 1$ are arbitrary, we conclude the theorem.       
    \end{proof}
\subsection{Proof of Lemma \ref{quantitative_cwikel_estimate}}
    
    The easiest case of Lemma \ref{quantitative_cwikel_estimate} is when $p>2,$ which is already contained in Theorem \ref{easy_cwikel}. To move down to $0<p<2,$ we first restrict to compactly supported $f.$
    \begin{lemma}\label{rough_specific_cwikel}
        Let
        \[
            0<p \leqslant 2 < q \leqslant \infty.
        \]
        Let $K$ be a compact subset of $G.$ For all $f$ supported in $K,$ we have
        \[
            \|M_f(1-\Delta)^{-\frac{d_{\hom}}{2vp}}\|_{p,\infty} \leqslant C_{K,p,q}\|f\|_{L_q(G)}
        \]
    \end{lemma}
    \begin{proof}
        Let $\psi\in C^\infty_c(G)$ be such that $f\psi=f.$ Let $r>0$ satisfy
        \[
                \frac{1}{r} = \frac{1}{p}-\frac{1}{q}.
        \]
        We have
        \[
            M_f(1-\Delta)^{-\frac{d_{\hom}}{2vp}} = M_f(1-\Delta)^{-\frac{d_{\hom}}{2vq}}\cdot (1-\Delta)^{\frac{d_{\hom}}{2vq}}M_\psi(1-\Delta)^{-\frac{d_{\hom}}{2vp}}.
        \]
        By Theorem \ref{easy_cwikel} and Lemma \ref{specific_cwikel_estimates_from_jfa_lemma}, this is the product of an element of $\Lc_{q,\infty}$ with an element $\Lc_{r,\infty}.$ By H\"older's inequality, we conclude that
        \[
            \|M_f(1-\Delta)^{-\frac{d_{\hom}}{2vp}}\|_{p,\infty} \lesssim_{p,q,\psi} \|f\|_{L_q(G)}.
        \]
    \end{proof}
    
    A verbatim repetition of the argument of \cite[Theorem 1.1(ii)]{MSZ-stratified-23} delivers the following:
    \begin{corollary}\label{bs_space_corollary}
        Let $0<p< 2<q\leqslant\infty,$ and $f \in \ell_{p,\infty}(L_q)(G).$ Then 
        \[
            \|M_f(1-\Delta)^{-\frac{d_{\hom}}{2vp}}\|_{p,\infty} \leq C_{p,q}\|f\|_{\ell_{p}(L_q)(G)}.
        \]
    \end{corollary}
    The following argument simplifies and strengthens the result of \cite[{\reva Theorem 1.1(iii)}]{MSZ-stratified-23} by obtaining $p=2$ as an interpolation of Lemma \ref{easy_cwikel} and Lemma \ref{rough_specific_cwikel}. For interpolation of noncommutative $L_p$ spaces used here, we refer to \cite[Proposition 7.5.2]{DdPS-vapour}.
    
\begin{proof}[Proof of Lemma \ref{quantitative_cwikel_estimate}]
    By Theorem \ref{easy_cwikel} we have
\begin{equation}\label{basic_cwikel_1}
\|M_f(1-\Delta)^{-\frac{\gamma}{{ 2v}}}\|_{\frac{d_{\hom}}{\gamma},\infty} \leqslant C_{\gamma,G}\|f\|_{L_{\frac{d_{\hom}}{\gamma}}(G)},\quad 2< \frac{d_{\hom}}{\gamma}<\infty
\end{equation}
and by Corollary \ref{bs_space_corollary}, we also have
\begin{equation}\label{basic_cwikel_2}
\|M_f(1-\Delta)^{-\frac{\gamma}{2v}}\|_{\frac{d_{\hom}}{\gamma},\infty} \leqslant C_{\gamma,G,q}\|f\|_{\ell_{\frac{d_{\hom}}{\gamma}}(L_q)(G)},\quad 0<\frac{d_{\hom}}{\gamma}<2,\quad q>2.
\end{equation}
We can extend the above statement to the case $\frac{d_{\hom}}{\gamma}=2$ using Lemma \ref{alt_lemma}. Let $0<p<2, r=\frac{2}{p},$ $A = M_{f}$
and $B = (1-\Delta)^{-\frac{d_{\hom}}{4v}}.$ Since $r>1,$ the Lemma implies
\[
\|M_f(1-\Delta)^{-\frac{d_{\hom}}{4v}}\|_{2,\infty} = \|AB\|_{rp,\infty} \leqslant e^{\frac{1}{rp}}\||A|^r|B|^r\|_{p,\infty}^{\frac1r} = e^{\frac12}\|M_{|f|}^{\frac{2}{p}}(1-\Delta)^{-\frac{d}{2vp}}\|_{p,\infty}^{\frac{p}{2}}.
\]
Let $q>2,$ and choose $0<p<2$ sufficiently close to $2$ such that $\frac{pq}{2}>2.$ Applying \eqref{basic_cwikel_2}, with $q$ replaced by $\frac{pq}{2},$ we have
\[
\|M_f(1-\Delta)^{-\frac{d_{\hom}}{4v}}\|_{2,\infty} \leqslant C_{\gamma,G,q}\||f|^{\frac{2}{p}}\|_{\ell_p(L_{\frac{pq}{2}})(G)}^{\frac{p}{2}} \lesssim \|f\|_{\ell_2(L_q)(G)}.
\]
This extends \eqref{basic_cwikel_2} to the case $\frac{d_{\hom}}{\gamma}=2.$
\end{proof}

\section{Commutators with fractional powers of "elliptic" operators}\label{psido_commutator_appendix}
In this appendix, $P\geqslant 0$ is an uniformly Rockland differential operator on $G$ of order $m.$ We are not assuming that $P$ has constant coefficients, in fact 
\[
P = \sum_{\len(\alpha)\leqslant m} M_{a_{\alpha}}X^{\alpha}
\]
where $a_{\alpha} \in C^\infty_b(G).$ By Theorem \ref{elliptic regularity theorem}, 
$P$ is self-adjoint and we can make sense of unbounded self-adjoint operator 
\[
J = (1+P)^{\frac1m}
\]
on $L_2(G).$ 

In this appendix we prove the following theorem:
\begin{theorem}\label{big_fractional_power_theorem} Let $J$ and $m$ be as above. For all $f \in C^\infty_c(G),$ and $\alpha,\beta,\gamma\in \mathbb{R}$ we have
\begin{enumerate}[\rm (i)]
\item\label{specific_cwikel_for_J} If $\alpha+\gamma>0,$ then
$$J^{-\alpha}M_fJ^{-\gamma} \in \Lc_{\frac{d}{\alpha+\gamma},\infty}.$$
\item\label{order_zero_is_bounded_for_J} If $\alpha+\gamma+1=\beta,$ then the operator
$$J^{-\alpha}[J^{\beta},M_f]J^{-\gamma}$$
is bounded on $L_2(G).$
\item\label{order_negative_is_schatten_for_J} If $\alpha+\gamma+1>\beta,$ then 
$$J^{-\alpha}[J^{\beta},M_f]J^{-\gamma} \in \mathcal{L}_{\frac{d_{\hom}}{\alpha+\gamma+1-\beta},\infty}.$$
\end{enumerate}
\end{theorem}

A visually identical theorem for a stratified Lie group, with $J = (1-\Delta)^{\frac12},$ appeared in \cite[Section 5]{MSZ-stratified-23}. Examination of the proofs there show that essentially the only features of $J$ used are that $J^2$ is a differential operator, and that $M_fJ^{-\alpha} \in \mathcal{L}_{\frac{d_{\hom}}{\alpha},\infty}.$ We can follow almost the same arguments verbatim, with the only difference being that we should use the fact that in Theorem \ref{big_fractional_power_theorem}, $J^m$ is a differential operator rather than $J^2.$

The basic scheme of the proof is the same as in \cite[Section 5]{MSZ-stratified-23}:
\begin{enumerate}[{\rm (i)}]
\item{} First we show part \ref{specific_cwikel_for_J} of Theorem \ref{big_fractional_power_theorem}, by replacing powers of $J$ by powers of $(1-\Delta)^{\frac{1}{2v}}$ using the Sobolev-mapping properties of $J.$
\item{} We then deduce part \ref{order_zero_is_bounded_for_J} for $\beta=m,$ using the fact that $J^m$ is a differential operator and $[J^m,M_f]$ is again a differential operator of order $m-1.$
\item{} We then use a double operator integral trick to deduce that $0<\beta<m$ cases of part \ref{order_zero_is_bounded_for_J} from the $\beta=m$ case.
\item{} Part \ref{order_zero_is_bounded_for_J} in full generality is deduced from the already proved cases by induction.
\item{} Part \ref{order_negative_is_schatten_for_J} is deduced from parts \ref{specific_cwikel_for_J} and \ref{order_zero_is_bounded_for_J}.
\end{enumerate}

\begin{lemma}\label{hard complex powers lemma} For all $z \in \mathbb{C}$ and $s\in \mathbb{R},$ $J^z$ is an isomorphism
$$J^z:W^{s+\Re(z)}_2(G)\to W^s_2(G).$$
\end{lemma}
\begin{proof} Note that the statement of the lemma is virtually identical to Lemma \ref{easy complex powers lemma}, the only difference is that here we allow $P$ to have non-constant coefficients. The proof is also the same.
\end{proof}

\begin{proof}[Proof of Theorem \ref{big_fractional_power_theorem} \ref{specific_cwikel_for_J}] Set
$$X_{\alpha,\gamma} := (1-\Delta)^{-\frac{\alpha}{2v}}M_f(1-\Delta)^{-\frac{\gamma}{2v}}.$$
If $\alpha+\gamma\geqslant0,$ then it follows from Lemma \ref{specific_cwikel_estimates_from_jfa_lemma} that $X_{\alpha,\gamma}\in\mathcal{B}(L_2(G)).$ If $\alpha+\gamma>0,$ then the same lemma gives $X_{\alpha,\gamma} \in \Lc_{\frac{d_{\hom}}{\alpha+\gamma},\infty}.$ By Lemma \ref{hard complex powers lemma}, the operators $J^{-\alpha}(1-\Delta)^{\frac{\alpha}{2v}}$ and $(1-\Delta)^{\frac{\gamma}{2v}}J^{-\gamma}$ are bounded on $L_2(G).$ Thus,
$$J^{-\alpha}M_fJ^{-\gamma} = J^{-\alpha}(1-\Delta)^{\frac{\alpha}{2v}}\cdot X_{\alpha,\beta}\cdot (1-\Delta)^{\frac{\gamma}{2v}}J^{-\gamma} \in \Bc(L_2(G))$$
if $\alpha+\gamma\geqslant0.$ Similarly, $J^{-\alpha}M_fJ^{-\gamma} \in \Lc_{\frac{d_{\hom}}{\alpha+\gamma},\infty}$ if $\alpha+\gamma>0.$
\end{proof}

The following lemma is very similar to Lemma \ref{asterisque pre-verification lemma}.  

\begin{lemma}\label{Jm corollary} Let $m\in\mathbb{Z}_+$ and let $\alpha+\gamma+1=m.$ We have
$$J^{-\alpha}[J^m,M_f]J^{-\gamma} \in \mathcal{B}(H).$$
\end{lemma}
\begin{proof} Since $J^m = 1+P=\sum_{\len{w}\leqslant m} M_{a_{w}}X^{w},$ it follows that
\begin{multline*}
    J^{-\alpha}[J^m,M_f]J^{-\gamma} = \sum_{\len(w)\leqslant m} J^{-\alpha}M_{a_{w}}[X^{w},M_f]J^{-\gamma}\\
    =J^{-\alpha}(1-\Delta)^{\frac{\alpha}{2v}}\cdot \sum_{\len(w)\leqslant m} (1-\Delta)^{-\frac{\alpha}{2v}}M_{a_{w}}[X^{w},M_f](1-\Delta)^{-\frac{\gamma}{2v}}\cdot (1-\Delta)^{\frac{\gamma}{2v}}J^{-\gamma}.
\end{multline*}
Note that
$$[X^{w},M_f]=\sum_{{\rm len}(w')\leqslant m-1}M_{f_{w,w'}}X^{w'}$$
where $f_{w,w'}\in C^{\infty}_c(G)$ for every $(w,w').$ Thus,
\begin{align*}
    & \sum_{\len(w)\leqslant m} (1-\Delta)^{-\frac{\alpha}{2v}}M_{a_{w}}[X^{w},M_f](1-\Delta)^{-\frac{\gamma}{2v}}\\
    & \quad =\sum_{\substack{\len(w)\leqslant m\\ {\rm len}(w')\leqslant m-1}} (1-\Delta)^{-\frac{\alpha}{2v}}M_{a_{w}f_{w,w'}}X^{w'}(1-\Delta)^{-\frac{\gamma}{2v}}\\
    & \quad =\sum_{\substack{\len(w)\leqslant m\\ {\rm len}(w')\leqslant m-1}} (1-\Delta)^{-\frac{\alpha}{2v}}M_{a_{w}f_{w,w'}}(1-\Delta)^{-\frac{\gamma+1-m}{2v}}\cdot (1-\Delta)^{\frac{\gamma-m+1}{2v}} X^{w'}(1-\Delta)^{-\frac{\gamma}{2v}}.
\end{align*}
Using the second part of Lemma \ref{specific_cwikel_estimates_from_jfa_lemma}, we conclude that
$$ \sum_{\len(w)\leqslant m} (1-\Delta)^{-\frac{\alpha}{2v}}M_{a_{w}}[X^{w},M_f](1-\Delta)^{-\frac{\gamma}{2v}}\in\mathcal{B}(H).$$
This completes the proof.
\end{proof}

In order to pass from $[J^m,M_f]$ to $[J^{\beta},M_f],$ we use the following lemma which is similar to one which appeared in \cite[Lemma 5.3]{MSZ-stratified-23}.
\begin{lemma}\label{phibetam lemma} Let $0<\beta<m.$ Define the function 
$$\phi_{\beta,m}(\lambda,\mu)=\frac{\lambda^{\beta}-\mu^{\beta}}{\lambda^m-\mu^m}\lambda^{\frac{m-\beta}{2}}\mu^{\frac{m-\beta}{2}},\quad \lambda,\mu>0.$$
For any positive operator $B,$ we have
$$\|T^{B,B}_{\phi_{\beta,m}}\|_{\mathcal{B}(H)\to\mathcal{B}(H)} < \infty.$$
\end{lemma}
\begin{proof} If $t\in \mathbb{R}$ is such that $\frac{\lambda}{\mu}=e^t,$ then
$$\phi_{\beta,m}(\lambda,\mu) = \phi_{\beta,m}(e^t\mu,\mu) = \frac{\mu^{\beta}(e^{\beta t}-1)}{\mu^m(e^{mt}-1)}\mu^{m-\beta}e^{t\frac{m-\beta}{2}} = \frac{\sinh(\frac{\beta}{2}t)}{\sinh(\frac{m}{2}t)}.$$
As $\beta<m,$ the function $t\mapsto \frac{\sinh(\frac{\beta}{2}t)}{\sinh(\frac{m}{2}t)}$ belongs to the Schwartz class. Let $g_{\beta,m}$ be its Fourier transform. By Fourier inversion,
$$\phi_{\beta,m}(\lambda,\mu)=\frac1{\sqrt{2\pi}}\int_{-\infty}^\infty g_{\beta,m}(s)\lambda^{is}\mu^{-is}\,ds.$$
Therefore, for $X \in \Bc(H),$ 
$$T_{\phi_{\beta,m}}^{B,B}(X) =\frac1{\sqrt{2\pi}}\int_{-\infty}^\infty g_{\beta,m}(s)B^{is}XB^{-is}\,ds.$$
Since $g_{\beta,m}$ is integrable, the boundedness of $T^{B,B}_{\phi_{\beta,m}}$ follows.
\end{proof}

The following lemma is in close analogy to \cite[Lemma 5.4]{MSZ-stratified-23}. Its proof is omitted.
\begin{lemma} Let $A,B\in\Bc(H)$ and assume that $B$ is strictly positive, i.e. there exists a constant $c>0$ such that $B\geqslant c.$ Let $\alpha,\beta\geqslant 0,$ and $0<\beta<m.$ Define $\phi_{\beta,m}$ as in Lemma \ref{phibetam lemma}. We have
$$B^{-\alpha}[B^{\beta},A]B^{-\gamma}=T^{B,B}_{\phi_{\beta,m}}\Big(B^{\frac{\beta-m}{2}-\alpha}[B^m,A]B^{\frac{\beta-m}{2}-\gamma}\Big),\quad 0<\beta<m.$$
\end{lemma}

\begin{corollary}\label{AB_corollary} Let $A$ and $B$ be positive operators, with $A:\mathrm{dom}(B^m)\mapsto \mathrm{dom}(B^m).$ Assume that $A$ is bounded and that $\ker(B) = \{0\}.$ We have
$$\|B^{-\alpha}[B^{\beta},A]B^{-\gamma}\|_{\infty} \leqslant C_{\alpha,\beta,\gamma}\|B^{\frac{\beta-m}{2}-\alpha}[B^m,A]B^{\frac{\beta-m}{2}-\gamma}\|_{\infty}.$$
\end{corollary}
\begin{proof} For every $n\geqslant 1,$ set $P_n = \chi_{(\frac1n,n)}(B),$ $H_n=P_n(H),$ $B_n=B|_{H_n}$ and $A_n=P_nBP_n|_{H_n}.$ Clearly, $A_n$ and $B_n$ satisfy the conditions of Lemma \ref{phibetam lemma}. Hence,
$$\|B_n^{-\alpha}[B_n^{\beta},A_n]B_n^{-\gamma}\|_{\infty} \leqslant C_{\alpha,\beta,\gamma}\|B_n^{\frac{\beta-m}{2}-\alpha}[B_n^m,A_n]B_n^{\frac{\beta-m}{2}-\gamma}\|_{\infty}.$$
Clearly,
$$B_n^{-\alpha}[B_n^{\beta},A_n]B_n^{-\gamma}=P_n\cdot B^{-\alpha}[B^{\beta},A]B^{-\gamma}\cdot P_n,$$
$$B_n^{\frac{\beta-m}{2}-\alpha}[B_n^m,A_n]B_n^{\frac{\beta-m}{2}-\gamma}=P_n\cdot B^{\frac{\beta-m}{2}-\alpha}[B^m,A]B^{\frac{\beta-m}{2}-\gamma}\cdot P_n.$$
Thus,
$$\|P_n\cdot B^{-\alpha}[B^{\beta},A]B^{-\gamma}\cdot P_n\|_{\infty} \leqslant C_{\alpha,\beta,\gamma}\|B^{\frac{\beta-m}{2}-\alpha}[B^m,A]B^{\frac{\beta-m}{2}-\gamma}\|_{\infty}$$
for every $n\geqslant 1.$ The assertion follows from the Fatou property in $\mathcal{B}(H).$
\end{proof}

\begin{lemma}\label{betaltm_case} Let $0<\beta\leqslant m,$ and let $\alpha$ and $\gamma$ be such that $\alpha+\gamma-\beta+1=0.$ We have
$$J^{-\alpha}[J^{\beta},M_f]J^{-\gamma} \in \mathcal{B}(H).$$
\end{lemma}
\begin{proof} Set $A=M_f$ and $B=J.$ Applying Corollary \ref{AB_corollary} and Lemma \ref{Jm corollary} delivers the result. 
\end{proof}

\begin{proof}[Proof of Theorem \ref{big_fractional_power_theorem} \ref{order_zero_is_bounded_for_J}] Lemma \ref{betaltm_case} already proves \ref{order_zero_is_bounded_for_J} for $0\leqslant \beta\leqslant m.$ Denote for brevity
$$X_{\alpha,\beta}=J^{-\alpha}[J^{\beta},M_f]J^{\alpha-\beta+1}.$$
By the Leibniz rule,
$$X_{\alpha,\beta}=X_{\alpha-1,\beta-1}+X_{\alpha-\beta+1,1}.$$
It follows by induction that $X_{\alpha,\beta}$ is bounded for every $\beta\geqslant 0.$

If $\beta\leqslant0,$ then it follows from the resolvent identity that
$$X_{\alpha,\beta}=-X_{\alpha-\beta,-\beta}$$
which is again bounded. This completes the proof.
\end{proof}

\begin{proof}[Proof of Theorem \ref{big_fractional_power_theorem} \ref{order_negative_is_schatten_for_J}]
We write $f=f_1f_2$ with $f_1,f_2\in C^{\infty}_c(G).$ By the Leibniz rule,
\begin{align*}
    J^{-\alpha}[J^{\beta},M_f]J^{-\gamma}
    & =J^{-\alpha}M_{f_1}[J^{\beta},M_{f_2}]J^{-\gamma}+J^{-\alpha}[J^{\beta},M_{f_1}]M_{f_2}J^{-\gamma}\\
    & =J^{-\alpha}M_{f_1}[J^{\beta},M_{f_2}]J^{-\gamma}+J^{-\alpha}[J^{\beta},M_{f_1}]M_{f_2}J^{-\gamma}\\
    & =J^{-\alpha}M_{f_1}J^{\beta-1-\gamma}\cdot J^{\gamma+1-\beta}[J^{\beta},M_{f_2}]J^{-\gamma}\\
    & \qquad +J^{-\alpha}[J^{\beta},M_{f_1}]J^{\alpha-\beta+1}\cdot J^{-\alpha+\beta-1}M_{f_2}J^{-\gamma}.
\end{align*}
In the first (respectively, the second) summand on the right hand side, the first (respectively, the second) factor belongs to $\Lc_{\frac{d_{\hom}}{\alpha+\gamma-\beta+1},\infty}$ by Theorem \ref{big_fractional_power_theorem} \ref{specific_cwikel_for_J} and the second (respectively, the first) factor is bounded by Theorem \ref{big_fractional_power_theorem} \ref{order_zero_is_bounded_for_J}. This completes the proof.
\end{proof}

\section{The Couchet-Yuncken residue}\label{sec-Connes-Trace}
In a recent paper, Couchet and Yuncken defined an analogy of the Wodzicki residue for pseudodifferential operators on a filtered manifold. Their construction took advantage of the groupoid techniques of van Erp and Yuncken \cite{vanErpYuncken2019}. In this section we indicate the relationship between their residue functional and the term appearing in the right hand side of Theorem \ref{main theorem general case}.

A full discussion of the residue is beyond the scope of this text, and therefore we restrict attention to the simplest possible case of translation-invariant operators on a stratified Lie group $G.$ This has the advantage that we can present a short and self-contained exposition, but leaves open the problem of relating Theorem \ref{main theorem general case} to Connes' trace theorem in general. It would be of interest to extend the results here to operators which are not translation-invariant.

We will only consider the residue for operators belonging to the following class.
\begin{definition}
Let $G$ be a stratified Lie group, with dilation semigroup $\{\delta_s\}_{s>0},$ and conjugation action $\{\alpha_s\}_{s>0}$ as in \eqref{conjugation_action}. Define the subspace
\begin{equation*}
    A(G) := \left\{T\in VN(G) \,:\, \text{For all }s>0,\ s^{d_{\hom}}\alpha_s(T) - T \in L_1(\mathrm{VN}(G), \tau) \right\}.
\end{equation*}
\end{definition}
The space $A(G)$ is an analogy of the space of symbols denoted $\Sigma^{-d_{H}}_H(X)$ in \cite{CouYuncken24}. Here, $X = G$ and we have restricted attention to operators in $\mathrm{VN}(G),$ which correspond to right-invariant operators on $G.$ A precise analogy of $\Sigma^{-d_{H}}_H(X)$ would replace $L_1(\mathrm{VN}(G),\tau)$ with the smaller set $\lambda(C^\infty_c(G)),$ but here $L_1(\mathrm{VN}(G),\tau)$ is more convenient.
\begin{lemma}\label{lem_res_is_well_defined}
    For any $T \in A(G),$ the function $a_T:(0,\infty)\to \mathbb{C}$ given by
    \[
        a_T(s) = \tau(s^{d_{\hom}}\alpha_s(T)-T)
    \]
    is continuous and satisfies
    \[
        a_T(sr) = a_T(s)+a_T(r),\quad s,r>0.
    \]
\end{lemma}
\begin{proof}
    Since the mapping $s \mapsto u\circ \delta_s$ for $u\in L_2(G)$ is continuous from $(0,\infty)$ to $L_2(G),$ it is easy to check that the mapping $s\mapsto \alpha_s(T)$ for $T \in \mathrm{VN}(G)$ is $\sigma$-weakly continuous. Hence,
    \[
        s\mapsto s^{d_{\hom}}\alpha_s(T)-T
    \]
    is continuous from $(0,\infty)$ to $\mathrm{VN}(G)$ with the $\sigma$-weak topology. Since $\tau$ is normal, we conclude that $a$ is continuous.

    Given $s,r>0,$ we compute
    \[
        a_T(sr) = r^{d_{\hom}}\tau(\alpha_r(s^{d_{\hom}}\alpha_{s}(T)-T))+\tau(r^{d_{\hom}}\alpha_r(T)-T).
    \]
    By \eqref{trace_scaling}, we conclude
    \[
        a_T(sr) = \tau(s^{d_{\hom}}\alpha_s(T)-T)+\tau(r^{d_{\hom}}\alpha_r(T)-T) = a_T(s)+a_T(r).
    \]
\end{proof}
The observation that $a_T$ is a continuous homomorphism from $(0,\infty)$ to $\mathbb{C},$ and is hence a multiple of a logarithm, is key to Couchet and Yuncken's definition of the residue. This could be viewed as a noncommutative generalisation of a classical
integral formula of Frullani, which will become apparent in Theorem \ref{thm_fractional_powers_residue} below.

The following is Couchet and Yuncken's definition, specialised to $A(G).$
\begin{definition}
    Let $T \in A(G).$ Given $s\neq 1,$ define
    \[
        \mathrm{Res}(T) := \frac{1}{\log(s)}\tau(s^{d_{\hom}}\alpha_{s}(T)-T).
    \]
\end{definition}

\begin{theorem}\label{thm_fractional_powers_residue}
    Let $P\geq 0$ be an elliptic order $m$ homogeneous differential operator with constant coefficients. Then $(1+P)^{-\frac{d_{\hom}}{m}}\in A(G),$ and
    \[
        \Res((1+P)^{-\frac{d_{\hom}}{m}}) = \frac{d_{\hom}}{\Gamma(\frac{d_{\hom}}{m}+1)}\tau(e^{-P}).
    \]
\end{theorem}
\begin{proof}
    Since $P$ is homogeneous of order order $m$, we have
    \[
        \alpha_s(P) = s^{m}P
    \]
    (compare \eqref{eq-P-homogeneous-rescale})
    and therefore
    \[
        \alpha_s((1+P)^{-\frac{d_{\hom}}{m}}) = (1+s^{m}P)^{-\frac{d_{\hom}}{m}}.
    \]
    Thus
    \[
        s^{d_{\hom}}\alpha_s((1+P)^{-\frac{d_{\hom}}{m}})-(1+P)^{-\frac{d_{\hom}}{m}} = f(P).
    \]
    where
    \[
        f(x) = s^{d_{\hom}}(1+s^{m}x)^{-\frac{d_{\hom}}{m}}-(1+x)^{-\frac{d_{\hom}}{m}}.
    \]
    It is easily checked that for every $s>0$ we have
    \[
        f\in L_1((0,\infty),x^{\frac{d_{\hom}}{m}-1}\,dx)
    \]
    and hence by Lemma \ref{lem:continuous_limit_formula}, $f(P) \in L_1(\mathrm{VN}(G),\tau).$ That is, $P\in A(G).$

    We compute the residue using Lemma \ref{lem:continuous_limit_formula}. By definition,
    \[
        \Res((1+P)^{-\frac{d_{\hom}}{m}}) = \frac{1}{\log(s)}\tau(f(P)) = \frac{1}{\log(s)\Gamma(\frac{d_{\hom}}{m})}\int_0^\infty f(x)x^{\frac{d_{\hom}}{m}-1}\,dx \cdot\tau(e^{-P})
    \]
    The integral is
    \[
        \int_0^\infty x^{\frac{d_{\hom}}{m}-1}(s^{d_{\hom}}(1+s^{m}x)^{-\frac{d_{\hom}}{m}}-(1+x)^{-\frac{d_{\hom}}{m}})\,dx = \int_0^\infty \frac{g(s^{m}x)-g(x)}{x}\,dx
    \]
    where $g(x) = x^{\frac{d_{\hom}}{m}}(1+x)^{-\frac{d_{\hom}}{m}}.$ A classical integral formula due to Frullani (see e.g. \cite{Agnew-frullani-1951}) says that
    \[
        \int_0^\infty \frac{g(s^{m}x)-g(x)}{x}\,dx = (\lim_{x\to\infty} g(x)-g(0))\log(s^{m}) = m\log(s).
    \]
    Finally,
    \[
        \Res((1+P)^{-\frac{d_{\hom}}{m}}) = \frac{m}{\Gamma(\frac{d_{\hom}}{m})}\tau(e^{-P}) = \frac{d_{\hom}}{\Gamma(\frac{d_{\hom}}{m}+1)}\tau(e^{-P}).
    \]
\end{proof}

With the formula in Theorem \ref{thm_fractional_powers_residue}, we can restate Theorem \ref{main theorem general case}.
\begin{theorem}\label{thm-Connes-trace-constant}
    Let $\gamma > 0$, and let $P\geqslant 0$ be an elliptic operator of order $m.$ For all $f \in C^\infty_c(G)$ we have
    \[
    \lim_{k\to\infty}k^{\frac{\gamma}{d_{\hom}}}\mu(k,M_f(P+1)^{-\frac{\gamma}{m}})= \frac{1}{d_{\hom}}\Big(\int_G |f(g)|^{\frac{d_{\hom}}{\gamma}} \mathrm{Res}((1+P^{\mathrm{top}}_g)^{-\frac{d_{\hom}}{m}})\,dg\Big)^{\frac{\gamma}{d_{\hom}}}
    \]
\end{theorem}
In particular, taking $\gamma=d_{\hom},$ Theorem \ref{thm-Connes-trace-constant} implies that for any continuous normalised trace $\varphi$ on the ideal $\mathcal{L}_{1,\infty}$ and $f \in C^\infty_c(G),$ we have
\[
    \varphi(M_f(1+P)^{-\frac{d_{\hom}}{m}}) = \frac{1}{d_{\hom}}\int_G f(g)\mathrm{Res}((1+P^{\mathrm{top}}_g)^{-\frac{d_{\hom}}{m}})\,dg.
\]
It would be worthwhile to verify that the right hand side is $\frac{1}{d_{\hom}}\mathrm{Res}(M_f(1+P)^{-\frac{d_{\hom}}{m}}),$ where $\mathrm{Res}$ is the Couchet-Yuncken residue for all order $-d_{\hom}$
pseudodifferential operators on $G$ in the sense of \cite{CouYuncken24}, but that is beyond our scope.

\section*{Acknowledgement}
We would like to express our gratitude to Nigel Higson for his insight and help. Shiqi Liu, Fedor Sukochev, and Dmitriy Zanin were supported by ARC grant DP$230100434$.

\section*{Declarations}
\addtocontents{toc}{\protect\setcounter{tocdepth}{1}}
\subsection*{Financial interests}
The authors have no relevant financial or financial interests to disclose. 
\subsection*{Data Availability}
Data sharing does not apply to this article as no dataset was generated or analyzed during the
current study.
\addtocontents{toc}{\protect\setcounter{tocdepth}{2}}

\bibliographystyle{alpha}
\bibliography{bibliography}

\end{document}